\documentclass[11 pt]{article}
\usepackage[margin=1in]{geometry}
\usepackage{hyperref}       % hyperlinks
\usepackage{url}            % simple URL typesetting
\usepackage{booktabs}       % professional-quality tables
\usepackage{amsfonts}       % blackboard math symbols
\usepackage{nicefrac}       % compact symbols for 1/2, etc.
\usepackage{microtype}      % microtypography
\usepackage{amstext}
\usepackage{amsthm}
\usepackage{amssymb}
\usepackage{amsmath}
\usepackage{makecell}
\usepackage{caption}
\usepackage{subcaption}
\usepackage{graphicx}
\usepackage[export]{adjustbox}

\usepackage[usenames,dvipsnames]{xcolor}
\definecolor{purple}{rgb}{0.3,0.0,.4}
\newcommand{\mnote}[1]{\textcolor{purple}{\textbf{[MU: #1]}}}
\newcommand{\collaboratornote}[1]{\textcolor{dkgreen}{\textbf{[collaborator: #1]}}}
\newcommand{\lnote}[1]{\textbf{(LD: #1)}}
\renewcommand{\mnote}[1]{\textcolor{purple}{}}
\renewcommand{\lnote}[1]{\textbf{}}
\renewcommand{\collaboratornote}[1]{\textcolor{dkgreen}{}}
\newcommand{\newcontent}[1]{\textcolor{blue}{#1}}
\renewcommand{\newcontent}[1]{\textcolor{black}{#1}}
\usepackage{multirow}

\usepackage{algorithm}
\usepackage{algorithmic}
\newtheorem{thm}{\protect\theoremname}
\newtheorem{defn}[thm]{\protect\definitionname}
\newtheorem{lem}[thm]{\protect\lemmaname}
\newtheorem*{lem*}{\protect\lemmaname}
\theoremstyle{definition}
\newtheorem{rem}[thm]{\protect\remarkname}
\usepackage{enumitem}

\usepackage[english]{babel}
\providecommand{\definitionname}{Definition}
\providecommand{\lemmaname}{Lemma}
\providecommand{\remarkname}{Remark}
\providecommand{\theoremname}{Theorem}
\newtheorem{exmp}{Example}[section]
\title{$k$FW: A Frank-Wolfe style algorithm \\
	with stronger subproblem oracles}

\author{Lijun Ding%\footnote{School of Operations Research and Information Engineering, Cornell University,
		, Jicong Fan%\footnote{School of Operations Research and Information Engineering, Cornell University, Ithaca, NY 14850, USA;
		, and Madeleine Udell\footnote{L. Ding is with the Department of Mathematics, Unviersity of Washington, Seattle, WA 98195, USA. E-mail: ljding@uw.edu.
		{Jicong Fan is with the School of Data Science, The Chinese University of Hong Kong, Shenzhen, China. E-mail: fanjicong@cuhk.edu.cn.}
		M. Udell is with the School of Operations Research and Information Engineering, Cornell University,
		Ithaca, NY 14850, USA. E-mail: udell@cornell.edu.
		}}
\DeclareMathOperator*{\argmin}{arg\,min}

\global\long\def\fronorm#1{\|#1\|_{\textup{F}}}

\global\long\def\twonorm#1{\|#1\|_{2}}

\global\long\def\norm#1{\|#1\|}

\global\long\def\opnorm#1{\|#1\|_{\textup{op}}}

\global\long\def\nucnorm#1{\|#1\|_{\textup{nuc}}}

\global\long\def\real{\mathbb{R}}

\global \long \def \EE{\mathbf{E}}

\global \long \def \FF{\mathbf{F}}

\global\long\def\inprod#1#2{\langle#1,#2\rangle}

\global\long\def\symMat{\mathbb{S}}

\global\long\def\dm{n}

\global\long\def\tr{\mathbf{tr}}

\global\long\def\Amap{\mathcal{A}}

\global\long\def\xsol{x_{\star}}
\global\long\def\Xsol{X_{\star}}

\global\long\def\rsol{r_{\star}}

\global\long\def\zsol{Z_{\star}}

\global\long\def\ssol{s_{\star}}

\global\long\def\onevec{\mathbf{1}}

\global\long\def\diag{\mbox{diag}}

\global\long\def\range{\mbox{range}}

\global\long\def\dist{\mbox{dist}}

\global\long\def\nullspace{\mbox{nullspace}}

\global\long\def\rank{\mbox{rank}}

\global\long\def\proj{\mathcal{P}}

\global\long\def\relint{\mbox{relint}}

\global\long\def\EV{\mathbf{EV}}

\global \long \def\conv{\mathbf{conv}}

\global \long \def\bigO{\mathcal{O}}

\newcommand{\myparagraph}[1]{\paragraph{#1.}}
\begin{document}
	\maketitle
	\vspace{-15pt}
	\begin{abstract}
		This paper proposes a new variant of Frank-Wolfe (FW),
		called $k$FW. Standard FW suffers from slow convergence:
		iterates often zig-zag as update directions
		oscillate around extreme points of the constraint set.
		The new variant, $k$FW, overcomes this problem by
		using two stronger subproblem oracles in each iteration.
		The first is a $k$ linear optimization oracle ($k$LOO) that computes
		the $k$ best update directions (rather than just one).
		The second is a $k$ direction search ($k$DS) that minimizes
		the objective over a constraint set represented by the $k$ best update directions
		and the previous iterate.
		When the problem solution admits a sparse representation,
		both oracles are easy to compute,
		and $k$FW converges quickly for smooth convex objectives
		and several interesting constraint sets:
		$k$FW achieves finite $\frac{4L_f^3D^4}{\gamma\delta^2}$ convergence on polytopes and group norm balls,
		and linear convergence on spectrahedra and nuclear norm balls.
		Numerical experiments validate the effectiveness of $k$FW
		and demonstrate an order-of-magnitude speedup over existing approaches.
	\end{abstract}

	\section{Introduction}
	We consider the following optimization problem with decision variable $x$:
	\begin{equation}\label{opt: mainProblem}
	\begin{array}{ll}
	\mbox{minimize} & f(x):\,=g(\Amap x) +\inprod{c}{x} \\
	\mbox{subject to} & x\in\Omega.
	\end{array}
	\end{equation}
	The constraint set $\Omega\subseteq \EE$ is a convex and compact subset of
	a finite dimensional Euclidean space $\EE$
	and has diameter $D$\footnote{The diameter of $\Omega$ is defined as
		$\sup_{x,y\in \Omega}\norm{x-y}$, where $\norm{\cdot}$ is the norm inducded by the inner product.}.
	The map $\Amap: \EE\rightarrow \FF$ is linear, where
	$\FF$ is another finite dimensional Euclidean space.
	We equip both spaces $\EE$ and $\FF$ with real inner products denoted as $\inprod{\cdot}{\cdot}$.
	The vector $c$ is in $\EE$.
	%The constraint set $\Omega\subseteq \EE$ is convex and compact with diameter $D$\footnote{The diameter of $\Omega$ is defined as
	%$\sup_{x,y\in \Omega}\norm{x-y}$, where $\norm{\cdot}$ is the norm inducded by the inner product.},
	The function $g: \FF \rightarrow \real$ is convex  and $L_g$-smooth\footnote{
		That is, the gradient of $g$ is $L_g$-Lipschitz continuous
		with respect to the norm $\norm{\cdot}$.}.
	The smoothness of $g$ implies that $f$ is $L_f$-smooth for some $L_f>0$.
	For ease of exposition,
	we assume Problem \eqref{opt: mainProblem} admits a unique solution.\footnote{
		The main results Theorem \ref{thm: normalFWresult}, \ref{thm: nonexponentialFiniteConvergence} and \ref{thm: LinearConvergence} remain valid
		for multiple optimal solution setting after minor adjustments, see Section \ref{sec: uniqueness}.
		Note that from \cite[Corollary 3.5]{drusvyatskiy2011generic}, the solution is indeed unique for almost all $c$.}

	\myparagraph{Applications}The optimization problem \eqref{opt: mainProblem} appears in
	a wide variety of applications,
	such as sparse vector recovery \cite{chen2001atomic},
	group-sparse vector recovery \cite{yuan2006model},
	combinatorial problems \cite{joulin2014efficient},
	submodular optimization \cite{bach2013learning,zhou2018limited},
	and low rank matrix recovery problems \cite{recht2010guaranteed,jaggi2010simple,yurtsever2017sketchy,ding2018frankwolfe}.

	\myparagraph{Frank-Wolfe and two subproblems}
	In many modern high-dimensional applications,
	Euclidean projection onto the set $\Omega$ is challenging.
	Hence the well-known projected gradient (PG) method and its acceleration version (APG)
	are not well suited for \eqref{opt: mainProblem}.
	Instead, researchers have turned to projection-free methods,
	such as the Frank-Wolfe algorithm (FW) \cite{frank1956algorithm},
	also known as the conditional gradient method \cite[Section 6]{levitin1966constrained}.
	As stated in Algorithm \ref{alg: Frank_wolfe}, FW operates in two computational steps:
	\begin{enumerate}%[noitemsep,topsep=0pt]
		\item \emph{Linear Optimization Oracle (LOO):}
		Find a direction $v_t$ that solves $\min_v \inprod{\nabla f(x_t)}{v}$.
		\item \emph{Line Search:}
		Find $x_{t+1}$ that solves $\min_{x = \eta v_t+(1-\eta)x_t, \eta \in [0,1]}f(x)$.
	\end{enumerate}
	The linear optimization oracle can be computed efficiently for many interesting
	constraint sets $\Omega$ even when projection is prohibitively expensive.
	These sets include the probability simplex,
	the $\ell_1$ norm ball, %the matching polytope,
	%certain flow polytopes,
	and many more polytopes arising from combinatorial optimization,
	the spectrahedron $\mathcal{SP}^\dm = \{X\in \symMat_+^\dm \mid \tr(X)=1\}$,
	and the unit nuclear norm ball $\mathbf{B}_{\nucnorm{\cdot}}=\{X\mid \nucnorm{X}\leq 1\}$.
	We refer the reader to \cite{jaggi2013revisiting,lacoste2015global} for further examples.
	Line search is easy to implement using a closed formula for quadratic $f$, or bisection in general.

	\myparagraph{Slow convergence of FW and the Zigzag}
	However, FW is known to be slow in both theory and practice,
	reaching an accuracy of $\bigO(\frac{1}{t})$ after $t$ iterations.
	This slow convergence is often described pictorially
	%\lnote{how about here description?}
	by the \emph{Zigzag phenomenon}
	depicted in Figure \ref{fig:figurezigzag}.
	%For a polytope $\Omega \subseteq \real^\dm$,
	The Zigzag phenomenon occurs when the optimal solution $\xsol$ of \eqref{opt: mainProblem}
	lies on the boundary of $\Omega$ and is a convex combination
	of $\rsol$ many extreme points $v_1^\star ,\dots,v_{\rsol}^\star \in \Omega$,
	(In Figure \ref{fig:figurezigzag}, $\rsol =2$.)
	\begin{equation}\label{eq: solutionDecomposition}
	\begin{aligned}
	\xsol = \sum_{i=1}^{\rsol} \lambda_i^\star v_i^\star , \quad \lambda_i^\star  >0,\;\text{and}\; \sum_{i=1}^{\rsol} \lambda_i^\star  =1.
	\end{aligned}
	\end{equation}
	When $\Omega$ is a polytope,
	the LOO will alternate
	between the extreme points $v_i^\star $s and
	the line search updates the estimate of $\lambda_i^\star $ slowly as the iterate approaches to $\xsol$.
	A similar Zigzag occurs for other sets such as the spectrahedron and nuclear norm ball.
	A long line of work has explored methods to reduce the complexity of FW using LOO and line search alone
	\cite{guelat1986some,lacoste2013affine,lacoste2015global,garber2015faster,garber2016linear,freund2017extended}.

	\myparagraph{Our key insight: overcoming zigzags with $k$FW}
	Our first observation is
	that the sparsity $\rsol$ is expected to be small
	for most large scale applications mentioned.
	For example,
	the sparsity is the number of nonzeros in sparse vector recovery,
	the number of nonzero groups in group-sparse vector recovery,
	and the rank in low rank matrix recovery.
	Next, note that from the optimality condition (also see Figure \ref{fig:figureconvhull}),
	the gradient $\nabla f(\xsol)$ in this case has the smallest inner product with
	$v_1^\star, \ldots, v_{\rsol}^\star$ among all $v\in \Omega$.
	Also, for small $\rsol$, we can solve
	$\min_{x \in \conv(x_t,v_1^\star,\dots, v_{\rsol}^\star)}f(x)$ efficiently \footnote{Here
		$\conv(v_1^\star,\dots,v_{\rsol}^\star)$ is the convex hull of $v_1,\dots,v_{\rsol}$. }
	to obtain the solution $\xsol$.
	Hence, our key insight to overcome the Zigzag is simply
	\begin{center}
		\textit{Compute all extreme points $v_i^\star$ that minimize $\inprod{\nabla f(\xsol)}{v}$ \\
			and solve the smaller problem $\min_{x \in \conv(x_t,v_1^\star,\dots, v_{\rsol}^\star)}f(x)$.}
	\end{center}
	\begin{figure}[H]
		\begin{subfigure}[(a)]{.5\textwidth}
			\centering
			\includegraphics[width=1\linewidth, height= 0.2\textheight]{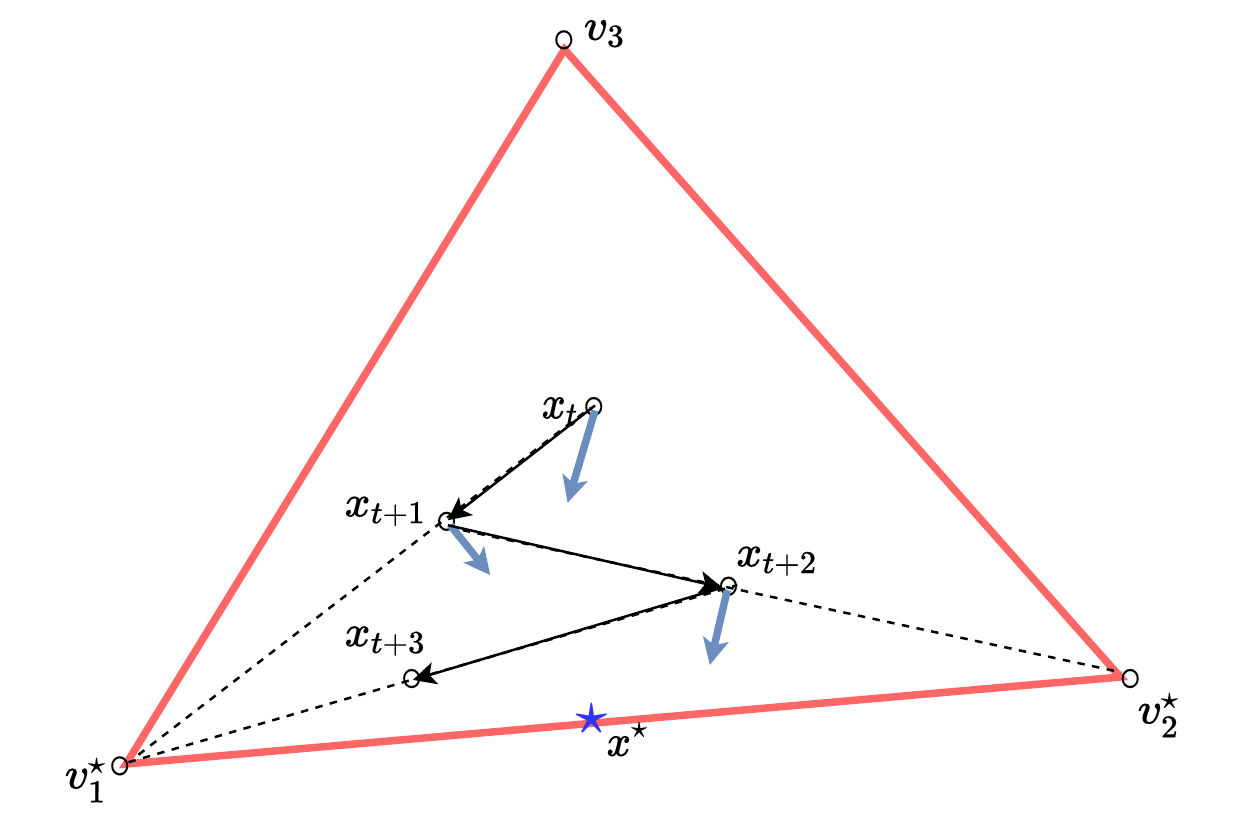}
			\caption{Zig-Zag: black arrows show trajectory of the iterates.}
			\label{fig:figurezigzag}
		\end{subfigure}%
		\hspace{5pt}
		\begin{subfigure}[(b)]{.5\textwidth}
			\centering
			\includegraphics[width=1\linewidth, height= 0.2\textheight]{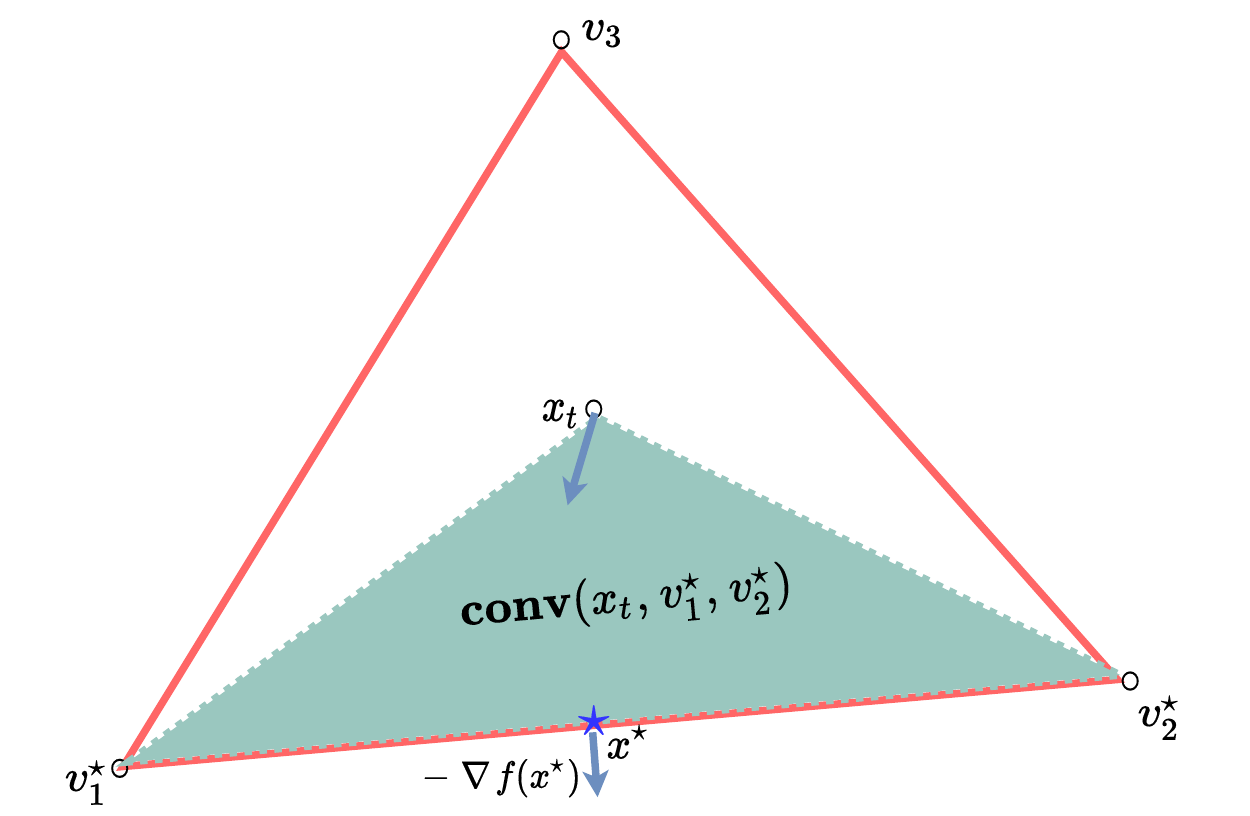}
			\caption{Optimization over $\conv(x_t,v_1^\star,v_2^\star)$ (green).}
			\label{fig:figureconvhull}
		\end{subfigure}
		\caption{The Zigzag phenomenon and optimization over $\conv(x_t,v_1^\star,v_2^\star)$.
			Here, the solution
			$\xsol$ is a convex combination of $v_1^\star$ and $v_2^\star$, and $\rsol=2$.
			The grey arrows are the negative gradients $-\nabla f$.}
	\end{figure}
	%
	%\myparagraph{$k$FW with its stronger subproblems oracles}
	%In this paper, we introduce a method called $k$ Frank-Wolfe ($k$FW)
	%that uses two stronger subproblem oracles to overcome the slow convergence of FW.
	%
	This insight leads us to define a new  algorithmic ways
	to choose extreme points and define a smaller convex search set,
	which we call $k$LOO and $k$DS.
	For polytope $\Omega$, they are defined
	as
	%To explain the idea, let us again suppose that $\Omega \subseteq \real^\dm$ is a polytope.
	%Our first subproblem oracle is $k$ direction search, an analogue of line search:
	\begin{itemize}[noitemsep,topsep=0pt]
		\item $k$ linear optimization oracle ($k$LOO):
		for any $y\in \real^\dm$, compute the $k$ extreme points $v_1,\dots,v_k$ ($k$ best directions) with the smallest $k$ inner products $\inprod{v}{y}$ among
		all extreme points $v$ of $\Omega$.
		\item $k$ direction search ($k$DS):
		given input directions $w, v_1,\dots,v_k\in  \Omega$,
		output
		$x_{k\text{DS}} =\argmin_{x \in \conv(w,v_1,\dots, v_k)}f(x)$.
	\end{itemize}
	In connection with FW, $k$LOO and $k$DS can be considered as stronger
	subproblem oracles compared to LOO and line search respectively.

	Combining the two subproblem oracles,
	we arrive at a new variant of the Frank-Wolfe algorithm:
	$k$FW, presented in Algorithm \ref{alg k-FW, polytope}.
	In Section \ref{sec: subproblems},
	we show that the two subproblems can actually be efficiently solved over many polytopes (for small $k$).
	Moreover, we redefine $k$LOO and $k$DS to incorporate the situation where $k$ best extreme
	points are not well-defined for sets such as group norm ball, spectrahedron, and nuclear norm ball,
	yet sparsity structure still persists.
	Finally, we note that with our terminology, $1$FW is the same as FW.
	Hence our main results,
	Theorem \ref{thm: normalFWresult}, \ref{thm: nonexponentialFiniteConvergence}, and \ref{thm: LinearConvergence},
	% apply to FW when $\rsol=1$ and
	give new insight into the
	fast convergence of FW when $\rsol=1$.

	\myparagraph{Computational efficiency of $k$FW}
	Here we summarize the computational efficiency of $k$FW in terms of its
	per iteration cost, iteration complexity, and storage for polytopes:
	\begin{itemize}[noitemsep,topsep=0pt]
		\item \emph{Per iteration cost:} For many important cases displayed in Section \ref{sec: subproblems}, $k$FW admits efficient subproblem oracles.
		\item \emph{Iteration complexity:} $k$FW achieves the same $\mathcal O(1/t)$ convergence rate of FW.
		Under additional regularity conditions, it achieves nonexponential finite convergence over the polytope and group norm ball and linear convergence over the spectrahedron and nuclear norm ball, as shown in Theorem \ref{thm: normalFWresult}, \ref{thm: nonexponentialFiniteConvergence}, and \ref{thm: LinearConvergence}. These convergence results are beyond the reach of FW and many of its variants \cite{guelat1986some,lacoste2015global,garber2016linear,freund2017extended}.
		\item \emph{Storage:} The storage required by $k$FW is $\mathcal{O}(kn)$, needed to store the $k$ best directions computed in each step, while the pairwise step, away step, and fully corrective step based FW \cite{lacoste2015global} require $\mathcal{O}(\min(tn,n^2))$ storage to accumulate vectors in $n$-dimensional Euclidean space computed from LOO \footnote{The algorithmic Caratheodory procedure described by \cite{beck2017linearly} can reduce the number
			of points stored to $n$.}.
	\end{itemize}
	A comparison of $k$FW with FW, away-step FW \cite{guelat1986some}, and fully corrective FW (FCFW) \cite[Algorithm 4]{jaggi2013revisiting} is shown in Table \ref{table: comparison}. 
	\newcontent{A recent result \cite{garber2020revisiting} shows that away-step FW can have better convergence and storage property after an initial burn-in period under similar assumptions as ours. Interestingly, in our experiments in Section \ref{sec: numerics}, we did not observe much of the benefit.}
	\begin{table}[tb]
		\centering
		\caption{Comparison of $k$FW, FW, FW with away step, and FCFW for Problem \eqref{opt: mainProblem}, smooth convex optimization over a constraint set
			$\Omega$ in a $n$-dimensional Euclidean space. We display the
			per iteration computation (per iter. comp.), storage,
			faster rate (compared to $\bigO(\frac{1}{t})$ rate) under the condition on $\Omega$,
			extra conditions on Problem \eqref{opt: mainProblem}
			to achieve the faster rate (Ex. Cond.),
			and the reference providing the proof of the rate. \newcontent{A recent result \cite{garber2020revisiting} shows that away-step FW can have better convergence and storage after a initial burn-in period under similar assumptions as ours.} However, it is hard to quantify 
			the burn-in peoriod as it depends on the parameter of the face where the solution lies. 
			Even without the extra conditions listed in the table,
			all algorithms admit a $\bigO(\frac{1}{t})$ convergence rate
			(see \cite{jaggi2013revisiting} and Theorem \ref{thm: normalFWresult}).
			Here $t\wedge n =\min(t,n)$.
			Definitions of the sparsity measure $\rsol$,
			strict complementarity (str. comp.),
			and quadratic growth (q.g.) can be found in Section \ref{sec: conditions}.
		} \label{table: comparison}
		\vspace{5pt}
		\begin{tabular}{|c|c|c|cc|c|c|}
			\hline
			Algorithm & Per iter. comp. &  Storage &  \multicolumn{2}{c|}{Rate and $\Omega$ Shape} & Ex. Cond. & Reference \\
			\hline
			\multirow{4}{*}{FW}
			& \multirow{4}{*}{LOO, $1$DS}         & \multirow{4}{*}{$\bigO(n)$}   & finite    &  polytope, group & str. comp.,                   & \multirow{2}{*}{Theorem \ref{thm: nonexponentialFiniteConvergence}} \\
			&                       &             &           &  norm ball                                       & q.g., and                                                 & \\ \cline{4-5}\cline{7-7}
			&                       &             & linear    &  spectrahedron                                   & $\rsol=1$                  & \multirow{2}{*}{Theorem \ref{thm: LinearConvergence}} \\
			&                       &             &           &  $\mathbf{B}_{\nucnorm{\cdot}}$                  &                            & \\
			\hline
			Away-step           & LOO, $1$DS,              & \multirow{3}{*}{$\bigO(n(t\wedge n))$} & \multirow{3}{*}{linear} & \multirow{3}{*}{polytope} & \multirow{3}{*}{q.g.} & \multirow{3}{*}{\cite{lacoste2015global},\cite{garber2020revisiting}} \\
			FW                  & and $t\wedge n$  inner  &                  &     &        &   &\\
			& products    &           &                  &     &        &\\
			\hline
			\multirow{2}{*}{FCFW}
			& LOO, and     & \multirow{2}{*}{$\bigO(n(t\wedge n))$} & \multirow{2}{*}{linear} & \multirow{2}{*}{polytope} & \multirow{2}{*}{q.g} &\multirow{2}{*}{\cite{lacoste2015global}} \\
			&  $(t\wedge n)$DS                                         &  &    &  & & \\
			\hline
			\multirow{4}{*}{$k$FW}&                     &\multirow{4}{*}{$\bigO(kn)$} & finite & polytope, group & str. comp.,   & \multirow{2}{*}{Theorem \ref{thm: nonexponentialFiniteConvergence}} \\
			& $k$LOO, and         &                             &        &  norm ball      &  q.g., and    & \\ \cline{4-5}\cline{7-7}
			& $k$DS              &                             & linear &  spectrahedron     &  $k\geq\rsol$ & \multirow{2}{*}{Theorem \ref{thm: LinearConvergence}} \\
			&                     &                             &        &  $\mathbf{B}_{\nucnorm{\cdot}}$ &   & \\
			\hline
		\end{tabular}
	\end{table}
	\myparagraph{Paper Organization}
	The rest of the paper is organized as follows. In Section \ref{sec: subproblems},
	we explain how to solve the two subproblems over a polytope $\Omega$,
	and how to extend the idea to group norm ball, spectrahedron,
	and nuclear norm ball.
	In Section \ref{sec: theoreticalGuarantees}, we describe a few analytical conditions,
	and then present the faster convergence guarantees of $k$FW under these conditions for
	the polytope, group norm ball, spectrahedron, and nuclear norm ball.
	We demonstrate the effectiveness of $k$FW numerically in Section \ref{sec: numerics}. 
	\newcontent{In Section \ref{sec: discussion}, we conclude the paper and present
	a discussion on related work and future direction is presented.}

	\myparagraph{Notation} The Euclidean spaces of interest in this paper
	are the $n$-th dimensional real Euclidean space $\real^\dm$,
	the set of real matrices $\real^{\dm_1\times \dm_2}$,
	and the set of symmetric matrices $\symMat^\dm$ in $\real^{\dm\times \dm}$.
	We equip the first one with the standard dot product and the latter two with the trace inner product.
	The induced norm is denoted as $\norm{\cdot}$ if not specified.
	%
	% journal version
	For a linear map $\mathcal{B}:\EE_1 \rightarrow \EE_2$ between two Euclidean spaces,
	we define its operator norm as $\opnorm{\mathcal{B}} = \max_{\norm{x}\leq 1}\norm{\mathcal{B}(x)}$.
	We denote the eigenvalues of a symmetric matrix $A\in \symMat^\dm$ as
	$
	\lambda_1(A)\geq \dots \geq \lambda_{\dm}(A).
	$
	The $i$-th largest singular value of a rectangular matrix $B\in \real^{\dm_1\times \dm_2}$
	is denoted as $\sigma_i(B)$.
	%The singular values $\sigma_i(A)$ of a rectangular matrix $A\in \real^{\dm_1\times \dm_2}$ are
	%ordered as
	%$
	%\sigma_1(A)\geq \dots \geq \sigma_{\dm}(A).
	% $ Similarly, the notation $\lambda_i(A)$ reports the $i$-th
	%largest eigenvalue of a symmetric matrix $A$.
	%The nuclear norm $\nucnorm{\cdot}$ is the sum of the singular values.
	A matrix $A\in \symMat^\dm$
	is positive semidefinite if all its eigenvalues are nonnegative
	and is denoted as $A\succeq 0$ or $A\in \symMat_+^\dm$.
	The column space of a matrix $A$ is written as $\range(A)$.
	The $i$-th standard  basis vector with appropriate dimension is denoted as $e_i$.

	\section{Stronger subproblem oracles for polytopes and beyond}\label{sec: subproblems}
	%shorter version
	In Section \ref{sec: subproblemPolytope}, we explain when the subproblem oracles
	can be implemented efficiently for polytopes.
	We then show how to extend $k$FW
	to more complex constraint sets by an appropriate definition of $k$LOO and $k$DS
	in Section \ref{sec: subproblemOtherConsSet}.

	\subsection{Stronger subproblem oracles for polytopes}\label{sec: subproblemPolytope}
	Let us first explain when the $k$LOO can be implemented efficiently for a polytope $\Omega \subseteq \real^{\dm}$.

	\myparagraph{Solving $k$LOO}
	Computing a LOO can be NP-hard for some constraint sets $\Omega$:
	for example, the $0\,\text{-}\,1$ knapsack problem can be formulated as linear optimization over an appropriate polytope.
	Hence we should not expect that we can compute a $k$LOO
	efficiently without further assumptions on the polytope $\Omega\subseteq \real^\dm$.
	Since many polytopes come from problems in combintorics,
	for these polytopes, computing a $k$LOO is equivalent to computing
	the $k$ best solutions to a problem in the combinatorics literature,
	and polynomial time algorithms are available for many polytopes
	\cite{murthy1968algorithm,lawler1972procedure,hamacher1985k,eppstein2014k}.
	We present the time complexity of computing $k$LOO
	for many interesting problems in Table \ref{table: kbestSolution} in the appendix.

	\myparagraph{Efficient $k$LOO}
	Unfortunately, for some polytopes,
	the time required to compute a $k$LOO grows superlinearly in $k$ even if
	$k\leq n$.
	Hence we restrict our attention to special structured polytopes
	for which the time complexity of $k$LOO is no more than $k$ times the
	complexity of LOO.

	Our primary example is the probability simplex $\Delta ^n =\conv(\{e_i\}_{i=1}^n)$ in $\real^{\dm}$.
	Since the vertices of $\Delta^n$ are
	the coordinate vectors $e_i,i=1,\dots,n$,
	the inner product of vertex $e_i$ with a vector $y\in \real^{n}$ is
	$\inprod{y}{e_i} = y_i$.
	Hence in this case, $k$LOO with input $y\in \real^{\dm}$
	simply outputs the coordinate vectors corresponding to the smallest $k$ values of $y$.
	Using a binary heap of $k$ nodes,
	we can scan through the entries of $y$ and
	update the heap to keep the $k$ smallest entries seen so far and their indices.
	Since each heap update takes time $\bigO(\log k)$,
	the time to compute $k$LOO is $\bigO(n\log k)$.
	A more sophisticated procedure called \emph{quickselect} improves the time to $\bigO(n+k)$ \cite{martinez2001optimal}, \cite[Section 2.1]{eppstein2014k}.
	Other examples of efficient $k$LOO includes, the $\ell_1$ norm ball, the spanning tree polytope \cite{eppstein1990finding},
	the Birkhoff polytope, \cite{murthy1968algorithm}, and the path polytope of a directed acyclic graph \cite{eppstein1998finding}.
	More details of each example and its application can be found in Section \ref{sec: tabelCombkLOO} in the appendix.

	Next, we explain how to compute the $k$ direction search.
	\myparagraph{$k$ direction search}
	The $k$ direction search problem optimizes the objective $f(x)$
	over $x \in \conv(w,v_1,\dots,v_k)=\{\sum_{i=1}^k \lambda_i v_i +\eta w \mid (\eta,\lambda)\in \Delta^{k+1}\}$.
	We parametrize this set by $(\eta,\lambda) \in \Delta^{k+1}$
	and employ the accelerated projected gradient method (APG) to solve
	\begin{equation}\label{opt: kdspolytope}
	\min_{(\eta,\lambda)\in \Delta^{k+1}}f\left (\sum_{i=1}^k \lambda_i v_i +\eta w \right).
	\end{equation}
	The constraint set here is a $k+1$ dimensional probability simplex;
	projection onto this set requires time $\bigO(k\log k)$ \cite{chen2011projection}.
	Hence for small $k$, we can solve \eqref{opt: kdspolytope} efficiently.
	We recover the output $x_{k-\text{DS}} =\sum_{i=1}^k \lambda_i^\star v_i +\eta^\star w$
	of $k$DS from the optimal solution $(\eta^\star,\lambda^\star)$ of \eqref{opt: kdspolytope}.

%\mnote{Make remark not italic}
\begin{rem}\label{rmk: psimp}
Optimizing over $\conv(w,v_1,\dots,v_k)=\{\sum_{i=1}^k \lambda_i v_i +\eta w \mid (\eta,\lambda)\in \Delta^{k+1}\}$
using the representation $(\eta,\lambda)$ can be challenging due to the high dimension $k$ of this parametrization.
Here we discuss a few alternative parametrizations that facilitate optimization.

For example, consider the product of simplices $\prod_{j=1}^d \Delta_{k_j}\subset \real^{\sum_{j=1}^d k_j}$, which appears in a variety of applications \cite{lacoste2013block}.
Denote the $j$-th block of $w\in \real^{\sum_{j=1}^d k_j}$ as $[w]^j\in \real^{k_j}$,
and let $e^j_i\in \real^{\sum_{j=1}^d k_j}$ be the indicator of the $i$-th position in $j$-th block (1 there, and 0 everywhere else).
Write $e_{i_1\dots i_j\dots i_d} = \sum_{j=1}^d e_{i_j}^j$.
Define the set $I = \prod_{j=1}^dI_j$ where $I_j\subset \{1,\dots,k_j\}$, which 
might be the generated when we apply $|I_j|$LOO in each block within $\bigO((\sum_{j=1}^dk_j)\log(\max_j|I_j|))$ time. 
Suppose we want to optimize over $\conv(w,\{e_{i_1\dots i_d}\}_{(i_1,\dots,i_d)\in I})$. 
If we use the $(\eta,\lambda)$ representation, the size of $(\eta,\lambda)$ is $\prod|I_j|+1$ which can be much larger than the number of nonzeros, $\sum_{j=1}^n|I_j|$, of the solution $\xsol$ or even larger than the total dimension $\sum_{j=1}^n|k_j|$,
even if the support of $\xsol$ in $j$-th block is exactly $I_j$.

To remedy the situation, we can equivalently parametrize the feasible set
as convex combinations of $w$ and the indicator vectors $e^j_i$.
Explicitly, consider the set $A_{I} =\{
(\{\alpha_{ij}\}_{1\leq j\leq d,i \in I_j },\eta)\mid
\alpha_{ij}\geq 0 \,\forall \,i,j,
\eta \geq 0,
\sum_{i\in I_j} \alpha_{ij} + \eta=1\}$.
It is then straightforward to verify that $\conv(w,\{e_{i_1\dots i_d}\}_{(i_1,\dots,i_d)\in I})= \{\eta w + \sum_{j=1}^d\sum_{i\in I_j} \alpha_{ij}e^j_i\mid (\alpha_{ij},\eta) \in A_{I}\}$. Note the latter set can be parametrized by  $\sum_{j=1}^{d}|I_j| +1$ variables
and hence should be easier to optimize over.
One strategy is to use bisection to choose $\eta$ by optimizing over $\alpha$ for each fixed $\eta$. Jointly optimizing $\eta$ and $\alpha$ might be hard as projection to $A_I$ is not easily computable.
%\mnote{Why is this easier?}

An alternative parametrization introduces a slightly larger convex set.
Indeed, consider $B_{w,I} = \{(\alpha_{ij},\eta_j)_{1\leq j\leq d,i \in I_j }\mid
\alpha_{ij}\geq 0 \,\forall \,i,j,
\eta_j \geq 0 \,\forall j,
\sum_{i\in I_j} \alpha_{ij} + \eta_j\sum_{i\not\in I_j} [w]^j_i) =1\}$.
Then
\[
\conv(w,\{e_{i_1\dots i_d}\}_{(i_1,\dots,i_d)\in I}) \subset \left\{ \sum_{j=1}^d\left(\sum_{i\not\in I_j}\eta_j
[w]^j_ie^j_i + \sum_{i\in I_j} \alpha_{ij}e^j_i\right) \middle| (\alpha_{ij},\eta) \in B_{w,I}\right\}.
\]
%\mnote{Would be nice if that $\mid$ were bigger.}

Note the latter set has $\sum_{j=1}^{d}|I_j| +d$ variables instead of $\sum_{j=1}^{d}|I_j| +1$ variables as in the previous approach.
However, note that optimizing over variables in $B_{w,I}$ is actually easier
as $B_{w,I}$ is a product of scaled simplices which enables faster projection.
\end{rem}
	\begin{algorithm}[H]
		\caption{Frank-Wolfe with line search}
		\label{alg: Frank_wolfe}
		\begin{algorithmic}
			\STATE {\bfseries Input:} initialization $x_0\in \mathbf{E}$
			\FOR{$t=1,2,\dots,$ }% {\bfseries to} $T$}
			\STATE \textbf{Linear optimization oracle (LOO):} Compute $v_t \in \arg\min_{v \in \mathbf{E}} \inprod{v}{\nabla f(x)}$.
			\STATE \textbf{Line search:} solve $\hat{\eta}=\arg\min_{\eta \in [0,1]}f(\eta x_{t}+(1-\eta)v_t )$  and set $x_{t+1} = \hat{\eta} x_{t}+(1-\hat{\eta})v_t$.
			\ENDFOR
		\end{algorithmic}
	\end{algorithm}
	\vspace{-15pt}
	\begin{algorithm}[H]
		\caption{$k$FW for polytope}
		\label{alg k-FW, polytope}
		\begin{algorithmic}
			\STATE {\bfseries Input:} initialization $x_0\in \Omega$, and an integer $k>0$
			\FOR{$t=1,2,\dots,$ }% {\bfseries to} $T$}
			\STATE \textbf{$k$ linear optimization oracle ($k$LOO):} compute $k$ extreme points, $v_1,\dots, v_k$ with smallest
			$\inprod{v}{\nabla f(x_t)}$ among all extreme points of $v\in \Omega$.
			\STATE \textbf{$k$ direction search ($k$DS):} Solve $\min_{x\in \conv(v_1,\dots,v_k,x_t)} f(x)$ to obtain $x_{t+1}$.
			\ENDFOR
		\end{algorithmic}
	\end{algorithm}
	\vspace{-15pt}
	\begin{algorithm}[H]
		\caption{$k$FW for other $\Omega$}
		\label{alg k-FW, otherOmega}
		\begin{algorithmic}
			\STATE {} Same as Algorithm \ref{alg k-FW, polytope},
			replacing the $k$LOO (with input $\nabla f(x_t)$) and
			$k$DS (with input consisting of $x_t$ and the output of $k$LOO, and output $x_{t+1}=x_{k\text{DS}}$),
			as described in Section \ref{sec: subproblemOtherConsSet}.
		\end{algorithmic}
	\end{algorithm}

	\subsection{Stronger subproblem oracle for nonpolytope $\Omega$} \label{sec: subproblemOtherConsSet}
	% journal version
	In this section, we explain how to extend $k$FW to operate on
	the unit group norm ball, spectrahedron, and nuclear norm ball.
	We shall \emph{redefine} the $k$LOO and $k$DS accordingly.
	\subsubsection{Group norm ball}
	Let us first define the
	group norm ball. Given a partition $\mathcal{G}=\{g_1,\dots,g_l\}$
	of the set $[\dm] =\{1,\dots,\dm\}$ ($\cup_{g\in \mathcal{G}} g = [\dm]$ and $g_i\cap g_j =\emptyset$ for $i\not= j$),
	the group norm and the unit group norm ball are
	\begin{equation}
	\norm{x}_\mathcal{G}:\,= \sum_{g\in \mathcal{G}} \norm{x_g},\, \forall x\in \real^{\dm}
	\quad \text{and} \quad
	\mathbf{B}_{\norm{\cdot}_\mathcal{G}}=\{x\in \real^{\dm}\mid\norm{x}_\mathcal{G}\leq 1\},
	\quad
	\text{respectively}.
	\end{equation}
	Here the base norm $\norm{\cdot}$ can be any $\ell_p$ norm, or even some matrix norms.
	We restrict our attention to the $\ell_2$ norm in the main text for ease of
	presentation\footnote{See Section \ref{sec: discussionOntheNorm} in the appendix for further discussion.}.
	The vector $x_g$ is formed by the entries of $x$ with indices in $g$.

	Let us now define $k$LOO and $k$DS for the group norm ball $\mathbf{B}_{\norm{\cdot}_\mathcal{G}}$.

	\myparagraph{$k$LOO}
	Given an input $y \in \real^{\dm}$, $k$LOO outputs
	the $k$ groups $v_1,\dots,v_k \in \mathcal G$ with
	largest $\norm{y_v}$ among all $v\in \mathcal{G}$.
	Here the $k$ best ``directions'' are not vectors, but groups.
	Notice the groups $g \in \mathcal G$ are disjoint, so
	$\sum_{i=1}^k \norm{y_{v_k}} \leq \norm{y}_\mathcal G$.
	In this sense, $k$LOO for the group norm ball generalizes $k$LOO for the simplex.

	\myparagraph{$k$DS}
	Given inputs $w\in  \mathbf{B}_{\norm{\cdot}_\mathcal{G}}$ and $v_1,\dots,v_k \in \mathcal{G}$,
	$k$DS for the group norm ball optimizes the objective $f(x)$ over
	convex combinations of $w$ and
	vectors supported on $\cup_{i=1}^k v_i$.
	To parametrize vectors supported on $\cup_{i=1}^k v_i$,
	we introduce a variable $\lambda^{v_1,\dots,v_k}\in \real^{\dm}$
	supported on $\cup_{i=1}^k v_i$.
	That is, $\lambda^{v_1,\dots,v_k}_i = 0$ for all $i\not \in \cup_{i=1}^k v_i$.
	Our decision variable $x$ is written as
	\begin{equation}\label{eq: groupnormorginalkDS}
	x = \eta w + \lambda^{v_1,\dots,v_k},
	\quad \text{where} \quad
	\eta +\norm{\lambda^{v_1,\dots,v_k}}_{\mathcal{G}} \leq 1,~
	\eta\geq 0.
	\end{equation}
	We solve the following problem to obtain $x_{k\text{DS}}$:
	\begin{equation}\label{eq: groupnormokDS}
	\mbox{minimize} \quad f( \eta w + \lambda^{v_1,\dots,v_k})
	\quad \mbox{subject to}\quad
	\eta +\norm{\lambda^{v_1,\dots,v_k}}_{\mathcal{G}} \leq 1,~ \eta\geq 0.
	\end{equation}
	We can again employ APG to solve this problem,
	as the projection step only requires $\bigO(k\log k+ \sum_{i=1}^k |v_k|)$ time.
	(See more details in Section \ref{sec: projgroupnormball}.)
	%The output of $k$DS $x_{k-\text{DS}}$ can be immediately obtained after solving \eqref{eq: groupnormokDS}.

	\subsubsection{Nuclear norm ball}
	We now define $k$LOO and $k$DS for
	the unit nuclear norm ball $\mathbf{B}_{\nucnorm{\cdot}} =
	\{X\in \real^{\dm_1\times \dm_2} \mid \nucnorm{X}\leq 1\}$,
	where $\nucnorm{X}= \sum_{i=1}^{\min(\dm_1,\dm_2)} \sigma_i (X)$, the sum of singular values.

	\myparagraph{$k$LOO} Given an input matrix $Y\in \real^{\dm_1\times \dm_2}$,
	define the $k$ best directions
	of the linearized objective $\min_{V\in \Omega(\alpha)}\inprod{V}{Y}$ to be
	the pairs $(u_1,v_1),\dots,(u_k,v_k)$,
	the top $k$ left and right orthonormal singular vectors of $Y$.
	Collect the output as $U = [u_1,\dots,u_k] \in \real^{\dm_1\times k}$
	and $V=[v_1,\dots,v_k] \in \real^{\dm_2\times k}$.

	\myparagraph{$k$DS}
	Take as inputs $W\in \mathbf{B}_{\nucnorm{\cdot}}$
	and $(U,V) \in \real^{\dm_1\times k} \times \real^{\dm_2\times k}$
	with orthonormal columns. Inspired by \cite{helmberg2000spectral},
	we consider the spectral convex combinations of $W$ and $u_iv_i^\top$ instead of just convex combination:
	\[
	X = \eta W+ USV^\top
	\quad \text{where}\quad
	\eta \geq 0,~\eta + \nucnorm{S}\leq 1.
	\]
	Next, we minimize the objective $f(X)$ parametrized by $(\eta,S)\in \real^{1+k^2}$ to obtain $X_{k \text{DS}}$:
	\[
	\mbox{minimize} \quad  f(\eta W+ USV^\top)\quad \mbox{subject to}\quad {\eta + \nucnorm{S}\leq 1,\eta\geq 0}.
	\]
	Again, we use APG to solve this problem.
	Projection requires singular value decomposition of a $k^2$ matrix, which
	is tolerable for small $k$.
	(See Section \ref{sec: projSpechedronAndNuclearNorm} for details.)
	%After solving the above minimization over  $(\eta,S)$, we can obtain the output
	%$X_{k-\text{DS}}$ immediately.

	A summary of $k$LOO and $k$DS for these sets appears
	in Table \ref{tb: table: kfwkLMOexample} and \ref{table: kfwkKsearchexample} in the appendix respectively. The case of spectrahedron
	has been addressed very recently by \cite{DingFeiXuYang2020}. We give a self-contained description of its
	$k$LOO and $k$DS in Section \ref{sec: AppendixSubproblemsSpectrahedron}.
	%\lnote{In the appendix maybe. Not enough space in main text.}
	The $k$FW algorithm for unit group norm ball, spectrahedron, and unit nuclear norm ball is presented as Algorithm \ref{alg k-FW, otherOmega}.
	%We discuss how to generalize these ideas to other atomic sets in Section \ref{sec: generalization} in the appendix.
\begin{rem}(Choice of $k$)
Having discussed $k$FW for various constraint sets, here we discuss the choice of $k$. Determining the choice of $k$ is of great importance as our guarantees require $k$ to be larger or equal to the underlying sparsity measure to observe significant speedup, see 
Definition \ref{def: sparsityMeasure} and results in Section \ref{sec: theoreticalGuarantees}. Domain knowledge is then particular helpful in this regard. In our experiments, for synthetic datasets, we have set $k$ to be the ground truth of the sparsity measure. On real data, we determine $k$ according to the expected sparsity level of the data (e.g. the expected number of support vectors in SVM, see Section \ref{sec: numerics} for details.) 

Here we provide a adaptive method to adjust $k$, which is shown in Algorithm \ref{alg: Frank_wolfe_adp_k}. The main idea is to increase $k$ in every iteration until it cannot improve the relative decrease of the objective function. We found that setting the increasing factor $\varsigma=2$ works every well in practice. For example, Figure \ref{fig_lasso_adp_k} shows the results of $k$FW with adaptive $k$ on the Lasso problem. We see that the algorithm is able to find a (possibly) optimal $k$ effectively with different initialization $k_0$. We also found similar performance (not shown here) in the experiments of other tasks discussed in Section \ref{sec: numerics}.
\end{rem}
\begin{algorithm}[H]
		\caption{$k$FW with adaptive $k$}
		\label{alg: Frank_wolfe_adp_k}
		\begin{algorithmic}
			\STATE {\bfseries Input:} initialization $x_0$ and $k_0$, parameter $\varsigma>1$, $inc=\text{True}$
			\FOR{$t=0,1,\dots,$ }% {\bfseries to} $T$}
			\STATE Compute $f(x_t)$
		    		\IF{$t=2$}
		         \STATE $k_t=\varsigma k_{t-1}$
		    		\ENDIF
		    		\IF{$t>2$ and $inc=\text{True}$}
		    		\IF{$(f(x_{t-1})-f(x_{t}))/f(x_{t-1})>(f(x_{t-2})-f(x_{t-1}))/f(x_{t-2})$}
		         \STATE $k_t=\varsigma k_{t-1}$
		         \ELSE
		         \STATE $inc=\text{False}$
		         \ENDIF
		    		\ENDIF
			\STATE \textbf{$k$LOO} step
			\STATE \textbf{$k$DS} step
			\ENDFOR
		\end{algorithmic}
\end{algorithm}

	\begin{figure}
	\centering
	\includegraphics[width=12cm]{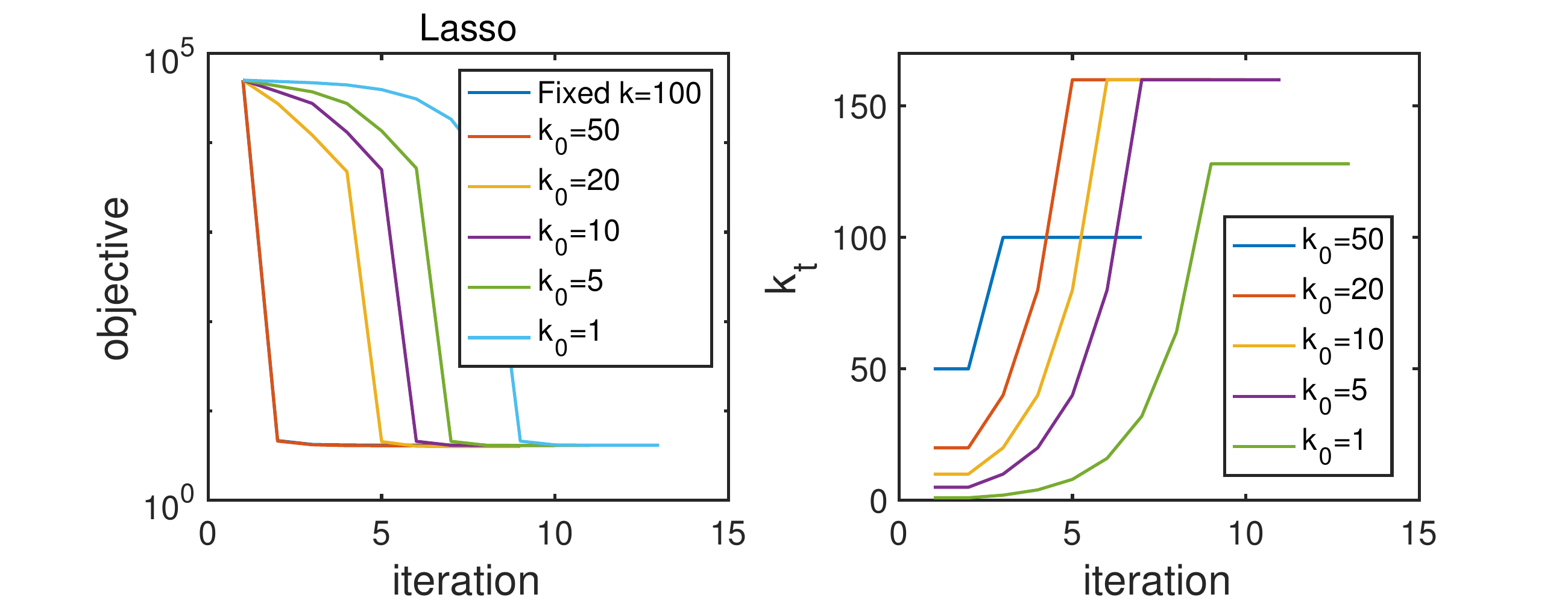}
		\caption{$k$FW with adaptive $k$ (Algorithm \ref{alg: Frank_wolfe_adp_k}) for Lasso}\label{fig_lasso_adp_k}
	\end{figure}
	\section{Theoretical guarantees}\label{sec: theoreticalGuarantees}
	% For journal version
	%In this section, we first define the notion of sparsity measure under different constraint set $\Omega$ of the optimal solution $\xsol$ of \eqref{opt: mainProblem}. Next, we describe the analytical conditions, strict complementarity (and its quantitive measure), quadratic growth. Equipped with these definitions, We present our theorems on convergence rates.
	%
	In this section, we first present a few definitions and conditions required to state our results.
	Then we present the theorems and provide intuitions.
	Proofs are deferred to
	Section \ref{sec: groupnormballproofconvergence} and \ref{sec: spectrahedronNuclearNormproofconvergence}.

	\subsection{Analytical conditions} \label{sec: conditions}
	Here we define the sparsity measure $\rsol$ for each constraint set $\Omega$
	and the complementarity measure $\delta$.

	%containing the optimal solution $\xsol$ of \eqref{opt: mainProblem}.
	% For journal version
	% As we shall see in Theorem \ref{thm: nonexponentialFiniteConvergence} and \ref{thm: LinearConvergence}, once $k\geq \rsol$, $k$FW converges faster than
	% standard FW. In this regard, the sparsity measure $\rsol$ counterbalances the per iteration complexity of $k$FW: we need higher $k$ for higher $\rsol$ to achieve faster convergence, which results in higher computational cost in $k$LOO %and $k$DS.
	%
	\begin{defn}[Sparsity measure $\rsol$] \label{def: sparsityMeasure}
		Suppose the solution $\xsol$ of \eqref{opt: mainProblem} is unique.
		% the measure of sparsity $\rsol$ for polytope, unit group norm ball, and
		% spechedron and unit nuclear norm ball are described as follows:
		%
		\begin{itemize}[noitemsep,topsep=0pt]
			\item \emph{Polytope:} The sparsity measure $\rsol$ is the number of extreme points
			of the smallest face $\mathcal{F}(\xsol)$ of $\Omega$ containing $\xsol$.
			\item \emph{Group norm ball:} The sparsity measure $\rsol$ is
			the number of groups $g \in \mathcal G$ such that $(\xsol)_g\not=0$,
			or equivalently, the cardinality of the set $\mathcal{F}(\xsol):=\{g\mid (\xsol)_g\not=0\}$:
			\item \emph{Spectrahedron and unit nuclear norm ball:}
			\mnote{Replace Spechedron with Spectrahedron globally.}
			The sparsity measure $\rsol$ is $\rank(\Xsol)$,
			or equivalently, the dimension of $\mathcal{F}(\Xsol)= \range(\Xsol)$.
		\end{itemize}
	\end{defn}
	%
	% journal version
	In short, the sparsity is the cardinality or the dimension of
	the support set $\mathcal{F}(\xsol)$. % for $\xsol$.

	\begin{defn}[Strict complementarity]\label{def: strictComplementarity}
		Problem \eqref{opt: mainProblem} admits strict complementarity
		if it has a unique solution $\xsol\in \partial \Omega$
		and $-\nabla f(\xsol) \in \relint(N_{\Omega}(\xsol))$\footnote{Here $\partial \Omega$ is
			the topological boundary of $\Omega$ under the standard topology of $\EE$. The set
			$N_{\Omega}(\xsol)$ is the normal cone of $\Omega$ at $\xsol$, i.e.
			$N_{\Omega}(\xsol) = \{y \mid \inprod{y}{x}\leq \inprod{y}{\xsol}, \,\forall x\in \Omega\}$,
			and $\relint(\cdot)$ is the relative interior.}
		The complementarity measure $\delta$ is the gap between
		the inner products of $\xsol$ and the elements of the \emph{complementary set} $\mathcal{F}^c(\xsol)$ defined below:
		\begin{equation}
		\begin{aligned}\label{eq: polytopeDeltaDef}
		\delta = \min \{ \inprod{u}{\nabla f(\xsol)} -\inprod{\xsol}{\nabla f(\xsol)}\mid u\in \mathcal{F}^c(\xsol)\subseteq\Omega\}.
		\end{aligned}
		\end{equation}
	\end{defn}
	%For journal version
	\myparagraph{The complementary set  $\mathcal{F}^c(\xsol)$}
	Morally, the complementary set $\mathcal{F}^c(\xsol)$ is
	the complement (in $\Omega$) of elements supported in $\mathcal{F}(\xsol)$.
	Our formal definition also respects the vector structure of these sets.
	% we define $\mathcal{F}^c(\xsol)$ as follows:
	\begin{itemize}[noitemsep,topsep=0pt]
		\item Polytope: The complementary space $\mathcal{F}^c(\xsol)$ is the convex hull of
		all vertices not in $\mathcal{F}(\xsol)$.
		\item Group norm ball: The complementary space $\mathcal{F}^c(\xsol)$
		is the set of all vectors in $\Omega$
		not supported in $\mathcal{F}(\xsol)=\{g\mid (\xsol)_g\not=0\}$.
		\item spectrahedron and nuclear norm ball: The complementary space  $\mathcal{F}^c(\Xsol)$
		is the set of all matrices in $\Omega$ with
		column space orthogonal to $\mathcal{F}(\Xsol)=\range(\Xsol)$.
		%Lemma \ref{lem: positiveDelta} reveals that $\delta = \sigma_{\rsol}(\nabla f(\Xsol)) -\sigma_{\rsol +1}(\nabla f(\Xsol))$ for nuclear norm ball, and a similar result for spectrahedron.
	\end{itemize}
	Table \ref{table: supportComplementarityGapDelta} in the appendix
	catalogues $\rsol$, $\mathcal{F}(\xsol)$, $\mathcal{F}^c$, and $\delta$ for several sets $\Omega$.
	Note that the definition of the gap $\delta$ is always nonnegative due to optimality condition of \eqref{opt: mainProblem}.
	It is indeed positive when strict complementarity holds as shown in Lemma \ref{lem: positiveDelta} in the appendix.
	% summarizes the sparsity measure $\rsol$, support and complementary sets $\mathcal{F}(\xsol)$ and $\mathcal{F}^c$ and the gap $\delta$.

	\myparagraph{Remarks on strict complementarity}
	Two aspects of strict complementarity have important implications
	for $k$FW. (See further discussion in \ref{sec: strictComplementarityFurtherDiscussion}.)
	%\begin{itemize}
	%	\item
	First, structurally, the strict complementarity condition ensures
	robustness of $\rsol$ under perturbations of the problem. %to the objective $f$ or the constraint $\Omega$.
	Indeed, consider the problem
	\begin{equation}
	%$
	\min_{x\in \Delta^n} \norm{x-\sigma e_1}^2+\inprod{c}{x}.
	%$
	\end{equation}
	For $c=0$ and any $ \sigma \in [0,1]$,
	the unique solution $\xsol = \sigma e_1$, which is also sparse.
	%The boundary location condition \emph{does not} hold in this case.
	However, it can be easily verified that strict complementarity \emph{fails} in this case.
	As a result, almost any small perturbation $c\not=0$ results in a solution with
	sparsity $\rsol>1$.
	%Note that location assumption alone is not able to guarantee the solution sparsity robustness by considering $\sigma =1$.
	We refer the reader to Example \ref{example: robustitcity} in the appendix
	and to \cite[Table 2]{garber2020revisiting}, and to \cite[Lemma 2 and 10]{garber2019linear} for more discussion on the relationship
	between complementarity and robustness of the solution sparsity.
	%	\item

	Second, algorithmically, the proof of Theorem \ref{thm: nonexponentialFiniteConvergence}
	and \ref{thm: LinearConvergence} reveal that $k$FW identifies the
	support set $\mathcal{F}(\xsol)$ once the iterate is near $\xsol$.
	The gap $\delta$ tells us how close it must be to identify the support.
	% the size of the neighborhood where the support is identifiable.
	% and so appears in the iteration complexity estimate of $k$FW in our theorems.
	%\end{itemize}

	% journal version: say something interesting about quadratic growth, or motivate it.
	We introduce the quadratic growth condition, a strictly weaker version of strong
	convexity. It is has been studied in \cite{drusvyatskiy2018error,necoara2019linear} to ensure
	linear convergence of some first order algorithms, and is also necessary as shown in \cite{necoara2019linear}.
	\begin{defn}[Quadratic growth]\label{def: quadraticGrowth}
		Problem \eqref{opt: mainProblem} admits quadratic growth with parameter $\gamma>0$ if
		it has a unique solution $\xsol$ and for all $x\in \Omega$,
		$
		f(x)-f(\xsol)\geq \gamma \norm{x-\xsol}^2.
		$
	\end{defn}

	\myparagraph{Remarks on quadratic growth}
	For all constraint set $\Omega$ considered in this paper,
	quadratic growth holds under strict complementarity for strongly convex $g$ in \eqref{opt: mainProblem}, $\min_{x\in\Omega} g(\Amap x) +\inprod{c}{x}$.
	(See Theorem \ref{thm: qgundersc} in the appendix for a proof.)
	Quadratic growth also holds for almost all $c$
	if $g$ and $\Omega$ are semi-algebraic \cite[Corollary 4.8]{drusvyatskiy2016generic}.

	%For the algorithm to converge faster than standard FW, the additional condition is that
	%the parameter $k$ should exceed the sparsity level $\rsol$ of
	%$\xsol$ measured described below. We also need to quantify the strength of strict complementarity, which
	%we denoted as a number $\delta>0$.
	%The exact definition of $\delta$ is also described below.

	%\myparagraph{Group norm ball} Let $\mathcal{G}(\xsol)$ be the set of all groups that $\xsol$ is in.
	%The strict complementarity parameter $\delta$ is the gap between
	%the inner products for $\xsol$ and those vector not in $\mathcal{G}(\xsol)$:
	%\begin{equation}
	%\begin{aligned}\label{eq: gruonormDeltaDef}
	%\delta = \min \{ \inprod{u}{\nabla f(\xsol)} -\inprod{\xsol }{\nabla f(\xsol)}\mid u\in \Omega, \,u_g = 0, \forall\, g\in \mathcal{G}(\xsol)\}.
	%\end{aligned}
	%\end{equation}
	%
	%\myparagraph{Spectrahedron and trace norm ball}
	%The gap $\delta$ for the spechdron is defined as
	%\[
	%\delta = \lambda_{n-\rsol-1}(\nabla f(\xsol)) -\lambda_{n-\rsol}(\nabla f(\xsol)).
	%\]
	%For the trace norm ball, we defined it as
	%\[
	%\delta = \sigma_{\rsol}(\nabla f(\xsol)) -\sigma_{\rsol+1}(\nabla f(\xsol)).
	%\]
	\subsection{Guarantees for $k$FW}
	Our first theorem states that $k$FW never requires more iterations than FW.
	% does no worse than FW in terms of iteration complexity.
	\begin{thm}\label{thm: normalFWresult}
		Suppose $f$ is $L_f$-smooth and convex and $\Omega$ is convex compact with diameter $D$.
		Then for any $k\geq 1$ and for all $t\geq 1$,
		the iterate $x_t$ in $k$FW (Algorithm \ref{alg k-FW, polytope} and \ref{alg k-FW, otherOmega}) satisfies
		\begin{equation}
		\begin{aligned}\label{eq: usualFW}
		f(x_t) - f(\xsol)\leq \frac{L_f D^2}{t}.
		\end{aligned}
		\end{equation}
	\end{thm}
	\begin{proof}
		The inequality \eqref{eq: usualFW} follows from the proof of convergence of FW
		as in \cite{jaggi2013revisiting} by noting that
		the vector $v_t =\arg\min_{v\in \Omega} \inprod{\nabla f(x_t)}{v}$ is feasible
		for the $k$DS minimization problem.
	\end{proof}
	The theorem shows that $k$FW converges faster when $k\geq \rsol$ for the polytope and group norm ball.
	\begin{thm} \label{thm: nonexponentialFiniteConvergence}
		Suppose that $f$ is $L_f$-smooth and convex, $\Omega$ is convex compact with diameter $D$, Problem \eqref{opt: mainProblem} satisfies
		strict complementarity and quadratic growth, and $k\geq \rsol$.
		If the constraint set $\Omega$ is a polytope or a unit group norm ball, then the gap $\delta>0$ and
		$k$FW finds $\xsol$ in at most $T+1$ iterations, where $T$ is
		\begin{equation}
		\begin{aligned}
		T = \frac{4L_f^3D^4}{\gamma\delta^2}.
		\end{aligned}
		\end{equation}
	\end{thm}
	\begin{proof}
		The proof follows from the intuition that once $x_t$ is close to $\xsol$, the set
		$\mathcal{F}(\xsol)$ can be identified using $\nabla f(x_t)$. The fact that $\delta>0$ is shown in
		Lemma \ref{lem: positiveDelta}.
		Let us now consider Algorithm \ref{alg k-FW, polytope} whose constraint set $\Omega$ is a
		polytope. Using quadratic growth in the following step $(a)$, and Theorem \ref{thm: normalFWresult} in the following second step $(b)$, and the choice in the following step $(c)$, the iterate $x_t$ with $t\geq T$ satisfies that
		\begin{equation}\label{eq: xtclosetooptimalx}
		\norm{x_t - \xsol} \overset{(a)}{\leq} \sqrt{\frac{1}{\gamma} (f(x_t)-f(\xsol))}
		\overset{(b)}{\leq} \sqrt{\frac{L_f D^2}{\gamma T}} \overset{(c)}{\leq} \frac{\delta}{2L_fD}.
		\end{equation}

		Next, for any $t\geq T$, we have that for any vertex $v$ in $\mathcal{F}(\xsol)$, and any vertices $u$ in $\mathcal{F}^c(\xsol)$,
		\begin{equation}\label{eq: correctExtremePoint}
		\begin{aligned}
		\inprod{\nabla f(x_t)}{v}  -\inprod{\nabla f(x_t)}{u} & = 	\inprod{\nabla f(\xsol)}{v-u} + \inprod{\nabla f(x_t) - \nabla f(\xsol)}{v-u}\\
		&\overset{(a)}{\leq}   -\delta  + \inprod{\nabla f(x_t) - \nabla f(\xsol)}{v-u} \overset{(b)}{\leq} -\frac{\delta}{2}.
		\end{aligned}
		\end{equation}
		Here in step $(a)$, we use the definition of $\delta$ in \eqref{eq: polytopeDeltaDef} and $\inprod{\xsol}{\nabla f(\xsol)}
		= \inprod{v}{\nabla f(\xsol)}$ using the optimality condition for Problem \eqref{opt: mainProblem} and $\mathcal{F}(\xsol)$ being the smallest face containing $\xsol$. In step $(b)$, we use the bound in \eqref{eq: xtclosetooptimalx}
		, Lipschitz continuity of $\nabla f(x)$, and $\norm{v-u}\leq D$.

		Thus, the $k$LOO step will produce all the vertices in $\mathcal{F}(\xsol)$ as $k\geq \rsol$ after $t\geq T$, and so $\xsol$ is a feasible and optimal solution
		of the optimization problem in the $k$direction search step. Hence, Algorithm \ref{alg k-FW, polytope} finds the optimal solution $\xsol$ within $T+1$ many steps. The case for unit group norm ball
		can be similarly analyzed and we defer the detail to Section \ref{sec: groupnormballproofconvergence} in the appendix.
		%	A detailed proof can be found in Section \ref{sec: groupnormballproofconvergence} in the appendix.
		%  \mnote{If you can squeeze it in, would be nice to provide intuition about the role of $\delta$.}
	\end{proof}

\begin{rem}[The burn-in period $T$] 
	The initial ``burn-in'' period scales as $\bigO\left(\frac{4L_f^3D^4}{\gamma\delta^2} \right)$, which is arguably too large for certain choice of $L_f,D,\gamma$, and $\delta$. It is possible to remedy the situation for many polytopes by incoporating the technique from  \cite{garber2016linear} by simply adding the atom $\gamma_t(v_t^+-v_t^-)$ proposed in \cite[Algorithm 3]{garber2016linear} into our $k$DS. Utilizing their convergence rate result \cite[Theorem 1]{garber2016linear}, the time $T$ an be improved to $\frac{\mbox{card} (x_\star )2L_f D^2}{\gamma}\log \left(\frac{4L_f^2D^2}{\delta^2\gamma} \right)$,  where $\mbox{card} (x_\star )$ is the number of nonzeros in $x_\star $. However, we note that in our experiments, the number of iterations of $k$FW is extremely low and the estimate $T$ is too pessimistic.
\end{rem}

\begin{rem}[Subproblem complexity] The finite complexity result for the polytope and group norm ball requires that each $k$DS solves the subproblem \eqref{opt: kdspolytope} exactly. A closer look reveals that the proof basically assumes that $k$FW achieves the worst case rate $\bigO(1/t)$ rate in the beginning, and once the iterate is close to $\xsol$, $k$LOO finds the optimal face and $k$DS finds the solution $\xsol$. For a theoretical analysis purpose of lowering the complexity in terms of gradient computation, LOO, and $k$LOO, one can modify the algorithm (assuming knowing the constant $L_f,D,\gamma,\delta$) so it first perform $T$ many iterations of FW with stepsize rule $\bigO(1/t)$; and then perform one $k$LOO and one $k$DS. This algorithm will require $T$ many gradient computation in the first stage and $k$ many LOOs, and one $k$LOO and one $k$DS in the second stage. If APG is employed in solving the problem in $k$DS, one requires $\bigO{\left(\frac{1}{\sqrt{\epsilon}}\right)}$ gradient computation. In our experiments, we found the subproblem is not too hard to solve to high accuracy and employing $k$DS in the beginning significantly reduces the total time. 
\end{rem} 

	Convergence for the spectrahedron and nuclear norm ball
	differs because any neighborhood of $\Xsol$ contains
	infinitely many matrices with rank $\leq \rsol$. The proof appears in
	Section \ref{sec: spectrahedronNuclearNormproofconvergence} in the appendix.
	\begin{thm}  \label{thm: LinearConvergence}
		Instate the assumption of Theorem \ref{thm: nonexponentialFiniteConvergence}.
		Then if the constraint set is the spectrahedron or the unit nuclear norm ball,
		the gap $\delta>0$ and $k$FW satisfies that for any $t\geq T:\,= \frac{72L_f^{3}}{\gamma\delta^2}$,
		%\lnote{the formula of T and factor for convergence need further check}
		\begin{equation}
		\begin{aligned}
		f(X_{t+1}) - f(\Xsol)\leq \left (1-\min\left\{ \frac{\gamma}{4L_f},\frac{\delta}{12L_f}\right \}\right)\left(f(X_t) - f(\Xsol)\right) .
		\end{aligned}
		\end{equation}
	\end{thm}
	%To explain why we cannot expect finite convergence, we take spectralhedron as an example. First, strict complementarity
	%		is equivalent to
	%		$
	%	\range(V_\star) = \range(\Xsol),
	%		$ (see .. in ... for a detailed derivation),
	%		where $V_\star \in \real ^{\dm \times \rsol}$ is a matrix formed by the eigenvectors of the smallest
	%		$\rsol$ eigenvalues of $\nabla f(\Xsol)$. Hence we could
	%		find $\Xsol$ by solving
	%		$
	%		\min _{X= V_\star SV_\star ^\top, S\in \mathcal{SP}^{\rsol}} f(X).
	%		$
	%	    However in the process of the Algorithm \ref{alg k-FW, polytope},
	%		even if the iterate $X_t$ is very close to the optimal solution $\Xsol$,
	%		the $k$ dimensional eigenspace for the smallest $k$ eigenvalues of the gradient $\nabla f(X_t)$
	%		still won't likely
	%		contain $\range(V_\star)$ even for $k>\rsol$. Hence we should not expect finite
	%		convergence. Similar situation happens for nuclear norm ball.
	\vspace{-15pt}
\subsection{A limitation of $k$FW for polytopes and potential fixes}
As stated in Theorem \ref{thm: nonexponentialFiniteConvergence}, $k$FW needs the parameter $k$ to be greater than or equal to the sparsity level 
$\rsol$, which is the number of vertices of the optimal face instead of $1+\dim (\mathcal{F}(\xsol))$,
	the face dimension plus one.
	The number $\rsol$ can be arbitrarily larger than
	the dimension $1+\dim (\mathcal{F}(\xsol))$ for some sets $\Omega$.
	This in turn means that $k$ has to be very large, at least in theory. 
	
Indeed, the following Example \ref{examp: worstcasekfw} shows that 
	if $k$ is only larger than $1+\dim(\mathcal{F}(\xsol))$, but not 
	larger than $\rsol$, $k$FW can behave as bad as standard Frank-Wolfe. 
	The slowdown occurs for the same reason as the zigzagging slowdown in Frank Wolfe:
	if all the vertices selected are nearly linearly dependent, they do not
	necessarily span the whole face, so $\xsol$ may lie far from their convex hull
	even though they lie on the same face as $\xsol$.
\begin{exmp}[A worst case example]\label{examp: worstcasekfw}
%	Theorem \ref{thm: nonexponentialFiniteConvergence} states after time $T = \frac{4L_f^3D^4}{\gamma\delta^2}$, $k$FW
%	stops if $k\geq \rsol$. The condition $k\geq \rsol$ is actually necessary in the worst case scenario. 
Consider the following problem:
	\begin{equation}\label{opt: worstCaseProblem}
		\begin{array}{ll}
			\mbox{minimize} & f(x,y,z):= x^2+y^2+z \\
			\mbox{subject to} & x\in\Omega: = \{(x,y,z)\mid \sqrt{x^2+y^2}\leq 1-z\}.
		\end{array}
	\end{equation}
	The constraint set $\Omega$ is an ice-cream cone. It is easily verified that the origin $(0,0,0)$ is the solution
	and that strict complementarity holds for this problem.
	Now, if we start at $x^0 = (x_0,y_0,z_0) = (0,0,0.1)$ and use FW to solve the problem,
	it can be shown that FW will produce iterates $x^{1},x^2,x^3,\dots,x^t,\dots$ that converge to the solution $(0,0,0)$ with rate $f(x^t)\geq \frac{c}{t}$ for some constant $c$. A numerical demonstration of the slow convergence is shown by the line ``FWSOC'' in Figure \ref{fig:objectivecomparison} for the objective value of the first 30 iterations.
	
	Next consider the constraint set $\Omega_n = \conv((0,0,1),\{\theta_i\}_{i=1}^n)$ where  $\theta_i=(\cos(2\pi i/n),\sin(2\pi i/n),0)$. In other words, $\Omega_n$ is the polytope being a convex hull of the  vertex $(0,0,1)$  and a $n$-polygon on the $x$-$y$ plane. 
	We now consider
	\begin{equation}\label{opt: worstCaseProblempolygon}
		\begin{array}{ll}
			\mbox{minimize} & f(x,y,z):= x^2+y^2+z \\
			\mbox{subject to} & x\in\Omega_n.
		\end{array}
	\end{equation}
	The $n$-polygon on the $x$-$y$ plane approximates the unit disk for large $n$, also see Figure \ref{fig:constraintregion2dnew} for an illustration of this approximation. Hence, the feasible region $\Omega_n$ approximates $\Omega$ for large $n$ (see \ref{fig:worstcaseregion} for a comparison), and we expect FW applied to Problem \eqref{opt: worstCaseProblem} and \eqref{opt: worstCaseProblempolygon} behaves similarly.
	One can verify that the origin is still the solution and that strict complementarity continues to holds for $n \geq 3$.
	Moreover, the strict complementarity
	parameter $\delta_n$ for this problem is exactly the same for all $n\geq 3$.
	The quadratic growth parameter $\gamma$ also stabilizes for all large $n$.
	Now consider
	running $k$FW starting from $x^0 = (x_0,y_0,z_0) = (0,0,0.1)$.
	Then for any $T$ and $k$,
	there is an $N$ such that for all $n>N$,
	the $k$FW iterate $x^{t,k\text{FW}}$ does not stop
	after $T$ iterations and $f(x^{t,k\text{FW}})\geq \frac{c}{2t}$ for all $t \leq T$.
	This is because for each $T$ and $k$, we can increase $n$ so that $\Omega_n$ is close enough 
	to $\Omega$. The closesness in the feasible region make the vertices found by $k$FW tend to be quite close to each other (just like the vertices found by FW), and they fail to form a convex combination of the optimal solution. Hence $x^{t, k \text{FW}}$ will be very close to $x^{t}$ and by our result from previous paragraph, we should have $f(x^{t,k\text{FW}})\geq \frac{c}{2t}$ for all $t \leq T$. A numerical illustration of this fact for the first 30 iterations of FW and $k$FW can be found in Figure \ref{fig:objectivecomparison}.
	In this case, the dimension of
	the optimal face is always $2$, but
	the number of vertices on the optimal face, $\rsol$, grows with $n$ and can be arbitrarily larger than $2$.
	
	\begin{figure}
		\centering
		\includegraphics[width=0.7\linewidth]{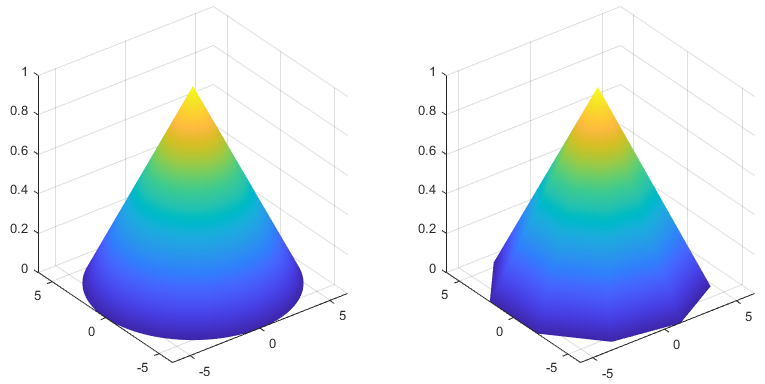}
		\caption{The left plot shows a second order cone in \eqref{opt: worstCaseProblem} . The right plot shows a polyhedron in \eqref{opt: worstCaseProblempolygon} that approximates the second order cone.}
		\label{fig:worstcaseregion}
	\end{figure}
	\begin{figure}
		\centering
		\includegraphics[width=0.7\linewidth]{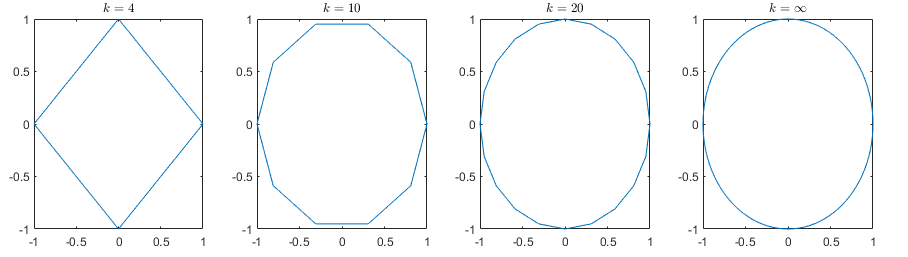}
		\caption{From left to right, we show the $n$-polygon, the bottom of the feasible region \eqref{opt: worstCaseProblempolygon}, that approximates the disk, the bottom of the second order cone.}
		\label{fig:constraintregion2dnew}
	\end{figure}
	\begin{figure}
		\centering
		\includegraphics[width=0.7\linewidth]{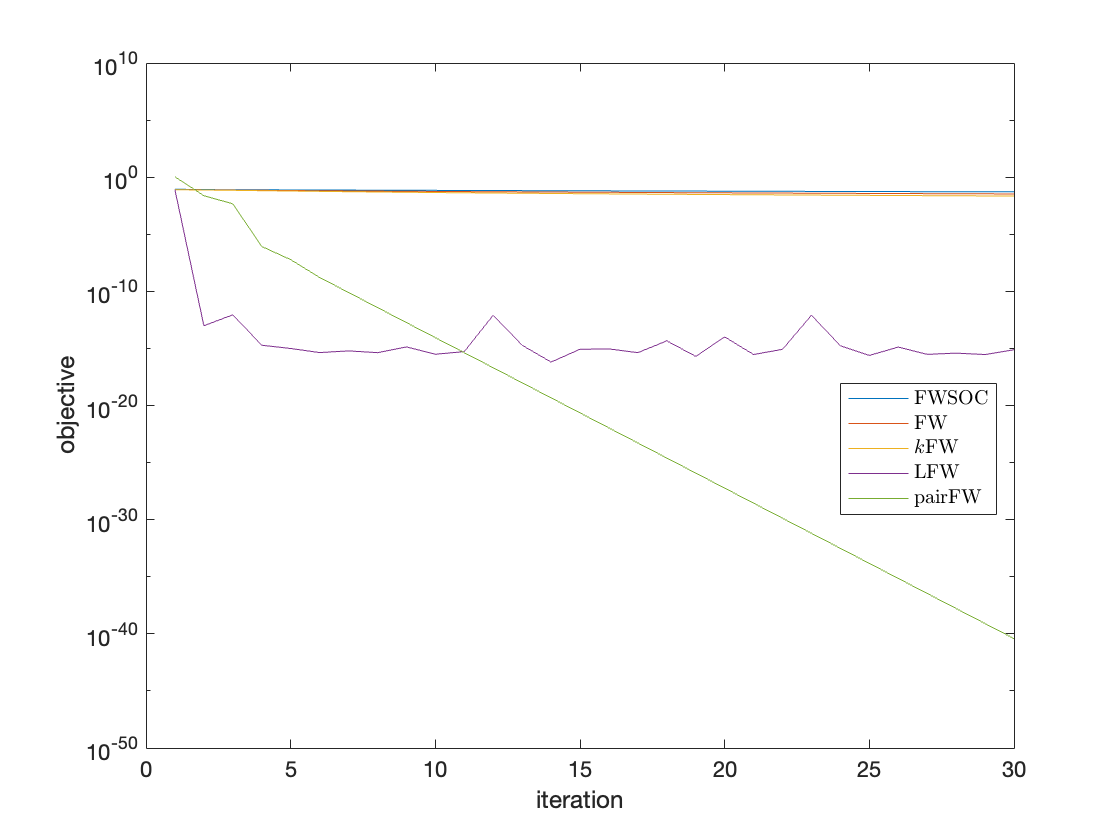}
		\caption{The objective value of different algorithms for \eqref{opt: worstCaseProblem} and \eqref{opt: worstCaseProblempolygon}. FWSOC shows the behavior of FW applied to \eqref{opt: worstCaseProblem}. The other lines are different algorithms applies to \eqref{opt: worstCaseProblempolygon}. LFW stands for limited memory Frank-Wolfe, which
          always keep the most recent $2$ vertices found by LOO, and optimal over the convex hull of the current iterate and these two past vertices. PairFW stands for pairwise FW.}
		\label{fig:objectivecomparison}
	\end{figure}
	
\end{exmp}
	
\paragraph{A quick fix}	We could remedy the problem with a stronger oracle:
	for example, one that outputs vertices that are always linear independent,
	or (even better) an oracle that can output a set of vertices whose
	convex hull contains $\xsol$ whenever the iterate is close to $\xsol$.
	If the oracle can achieve this latter property,
	then $k=1+\dim (\mathcal{F}(\xsol))$ suffices for fast convergence
	by Caratheodory's theorem.
	Hence we avoid the dependence on $\rsol$ defined here.
	These strong oracles exist for some sets, such as
	the simplex, the $\ell_1$ norm ball, and the spanning tree in a graph (in the sense of orthogonality),
	but in general they may be prohibitively expensive to compute.
	
\paragraph{The relation between $\rsol$ and $\dim(\mathcal{F}(\xsol)$} For polytopes encountered in practice,
	the relation between $\rsol$ and the face dimension varies.
	For the probability simplex and the $\ell_1$ norm ball,
	$\rsol=1+\dim (\mathcal{F}(\xsol))$.
	Frank Wolfe is often used to optimize over these sets, and our algorithm
	presents a substantial advantage here.
	% It is no suprise that Frank Wolfe is often used to optimize over these sets!
	For other types of polytope, the dependence of $\rsol$ on
	the face dimension can be polynomial or exponential,
	e.g., products of simplices and Birkhoff polytopes.
	
\paragraph{Our responses to the limitation}	The limitation of $k\geq \rsol$ instead of $k\geq \dim(\mathcal{F}(\xsol))+1$ of 
	our theory need not spell disaster in practice:
	small $k$ can still work.
	% Even though the dependence might look scary,
	% we think the choice of $k$ does not need to be that large.
	First, if $\xsol$ is a vertex or $\xsol$ lies on a 1 dimensional line 
	segment, then $\rsol = 2$ suffices.
	More generally, if $\xsol$ lies on a face with dimension independent of the ambient
	dimension $n$,
	then $\rsol$ is independent of $n$ even if it is exponential in $\dim (\mathcal{F}(\xsol))$.
	Second, even if $\rsol$ is polynomial or exponential in the dimension of the face $\mathcal{F}(\xsol)$,
	it is sometimes still easy to solve the $k$LOO and $k$DS subproblem:
	for example, these subproblem are still easy for
	%optimize over the convex hull of the current iterate and the best $k$ vertices.
	the simplex, a product of simplices, the $\ell_1$ norm ball, or a product of $\ell_1$ norm balls. See Remark \ref{rmk: psimp} for a detailed discussion.
	Finally, as shown by the following example, we found that in numerics, a modified version of $k$FW which incorporates limited past information can limit the choice of $k$ to $\bigO(\dim(\mathcal{F}(\xsol)))$ even though 
	the vanilla version may fail. 
\begin{exmp}
	We consider the problem of projection to the hypercube $[0,1]^n$: 
		\begin{equation}\label{opt: hypercube}
		\begin{array}{ll}
			\mbox{minimize} & f(x):= \twonorm{x-x_0}^2 \\
			\mbox{subject to} & x\in [0,1]^n.
		\end{array}
	\end{equation}
We perform experiments with $n=50$ and set the first $10$ coordinates of $x_0$ to be uniformly chosen from $[0,1]$, 
and set the rest of the coordinates to be $2$. This choice of $x_0$ ensures the strict complementarity condition is satisfied. 
We compute the optimal solution of \eqref{opt: hypercube} via Sedumi \cite{sturm1999using} and found it is has $40$ ones and 
	all the other entries are in $(0,1)$. Note this face has $2^10$ many vertices though. 
	Hence the optimal face is $\mathcal{F}(\xsol)$ is $10$ dimensional. We set $k=10$ and we try 
	(1) FW, (2) $k$FW with $k= 11$, (3) $k$FW with limited memory (L$k$FW), that is, we also keep the most recent $k-1$ vertices found by LOO
	in the past $k-1$ iterations, and then add them together with the output of $k$LOO into $2k-1$-DS.  (4) FW with limited 
	past information (LFW), which call LOO once 
	in every iteration, but keep the most recent $k-1$ vertices found by LOO in the past $k-1$ iterations. The results of the objective 
	value against the iteration is shown in Figure \ref{fig:hypercube}. It can be seen that once past information is incorporated, L$k$FW is able
	to find the optimal solution in very few iterations while the vanilla $k$FW behaves similarly to FW. It is also interesting LFW itself 
	is as fast as L$k$FW for this problem.
\end{exmp}
	%\lnote{I tried hypercube example and this is the case}
	%we have tried on polytopes with exponential many vertices on a face, we found ??
	%\lnote{Need to do a numeric on an example for the exponential dependence problem which is actually not that easy}
	%However, as we demonstrated in Section \ref{sec: numerics}, empirically, we only need $k$ to be slightly larger than $1+\dim (\mathcal{F}(\xsol))$ to achieve
	%faster convergence of $k$FW. [Actually I am not sure about the last sentence]
\begin{figure}
		\centering
		\includegraphics[width=0.7\linewidth]{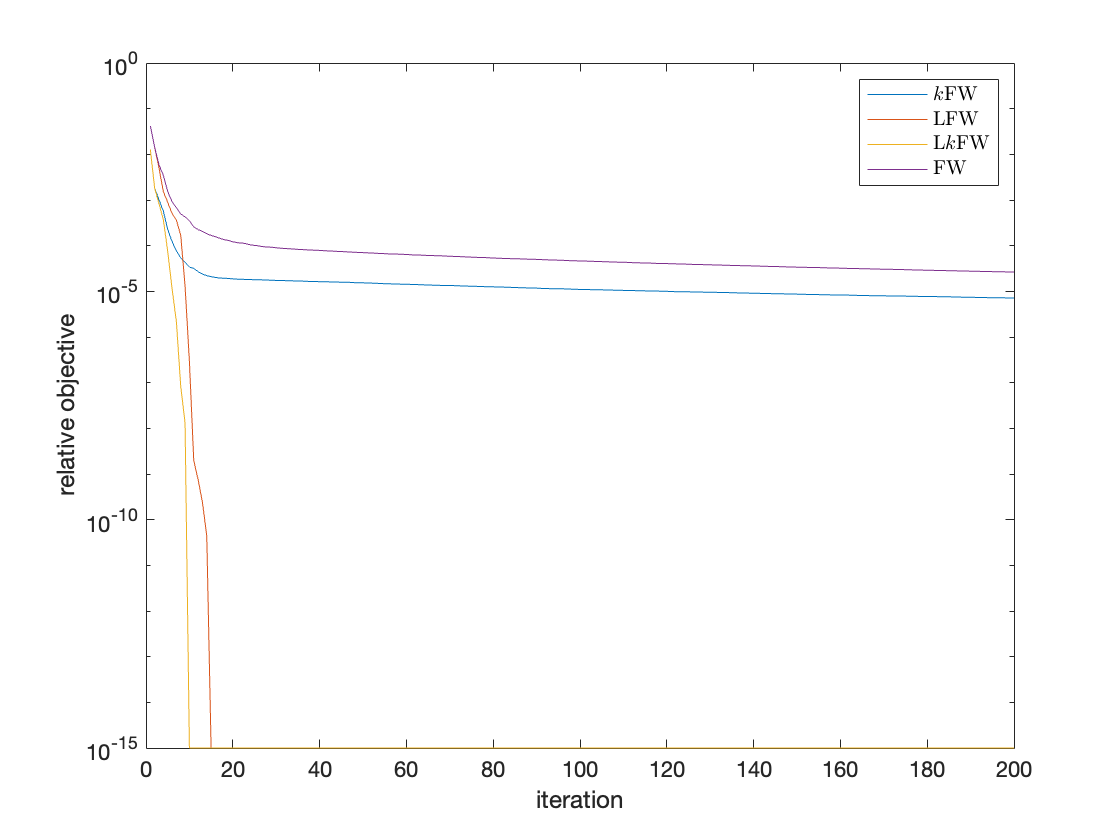}
		\caption{Relative objective $\frac{f(x)-f(x_\star)}{f(x_\star)}$ for the hypercube problem \eqref{opt: hypercube} of different algorithms.}
		\label{fig:hypercube}
\end{figure}

	\section{Numerics} \label{sec: numerics}
	In this section, we perform experiments to see the emprical behavior of $k$FW. We first start with synthetic datasets, where we set $k$ to be the ground truth of the sparsity measure. Next, we experiment on real data, we determine $k$ according to the expected sparsity level of the data (e.g. the expected number of support vectors in SVM).
	
	\subsection{Synthetic data}
	We compare our method $k$FW with
	FW, away-step FW (awayFW) \cite{guelat1986some}, pairwise FW (pairFW)\cite{lacoste2015global},
	DICG \cite{garber2016linear}, and blockFW \cite{allen2017linear}
	for the Lasso, support vector machine (SVM), group Lasso, and matrix completion problems
	on synthetic data.
	Details about experimental settings appear in the Appendix \ref{sec: numericalSectionSettingAppendix}.
	All algorithms terminate when the relative change of the objective
	is less than $10^{-6}$ or after $1000$ iterations.
	As shown in Figure \ref{fig: Figure_compare},
	$k$FW converges in many fewer iterations than other methods.
	Table \ref{table: Table_time} shows that $k$FW also converges faster
	in wall-clock time,
	with one exception (blockFW in matrix completion).
	Note that blockFW is sensitive to the step size while $k$FW has no step size to tune.
	More numerics can be found in Appendix \ref{sec: numericalSectionSettingAppendix}.
	\begin{figure}[H]\label{Figure_compare}
		\hspace{-15pt}
		\begin{subfigure}[(a)]{.24\textwidth}
			\centering
			\includegraphics[width=1.1\linewidth]{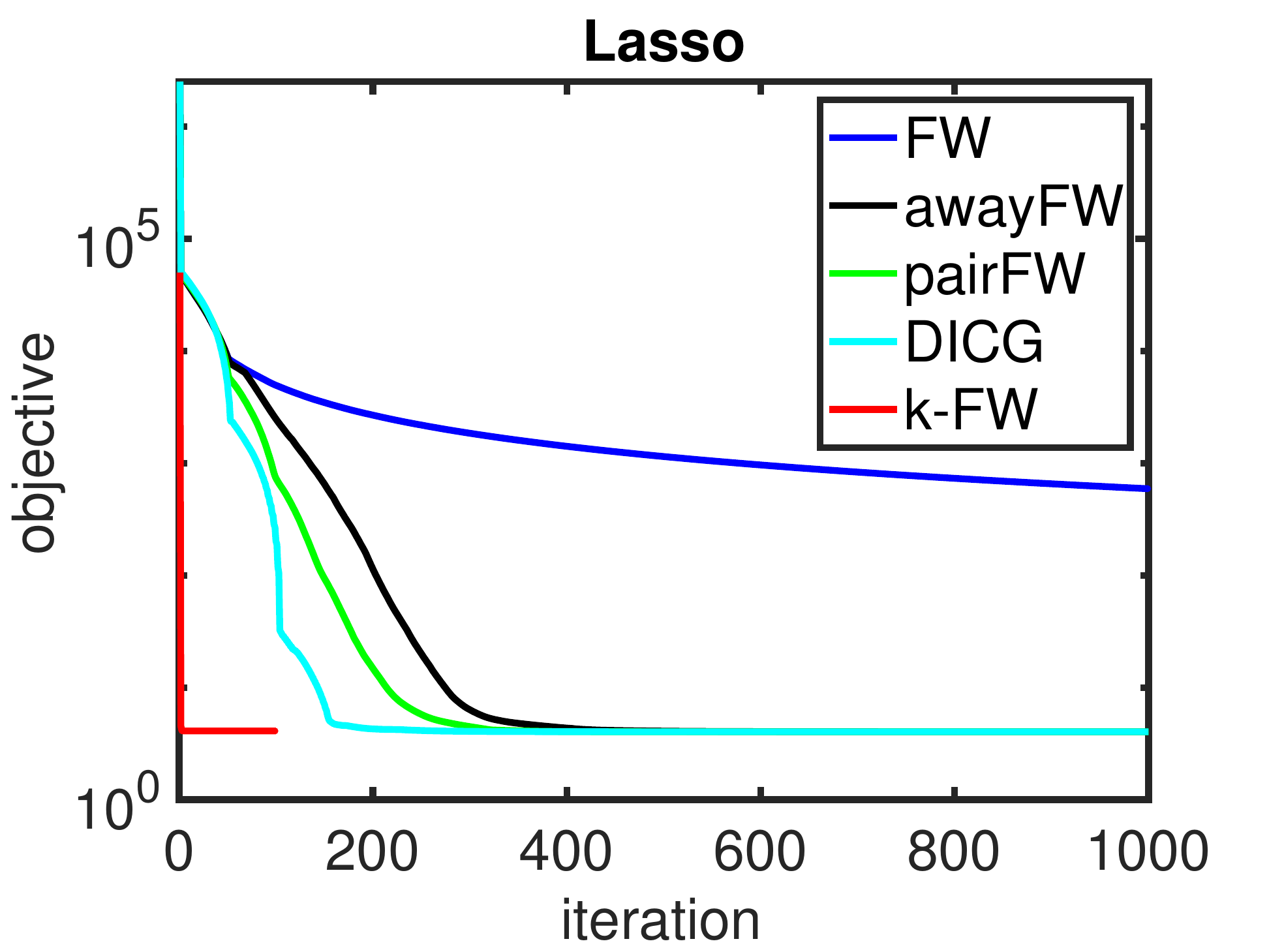}
			%\caption{Lasso}
			\label{fig:fig_lasso}
		\end{subfigure}%
		\hspace{5pt}
		\begin{subfigure}[(b)]{.24\textwidth}
			\centering
			\includegraphics[width=1.1\linewidth]{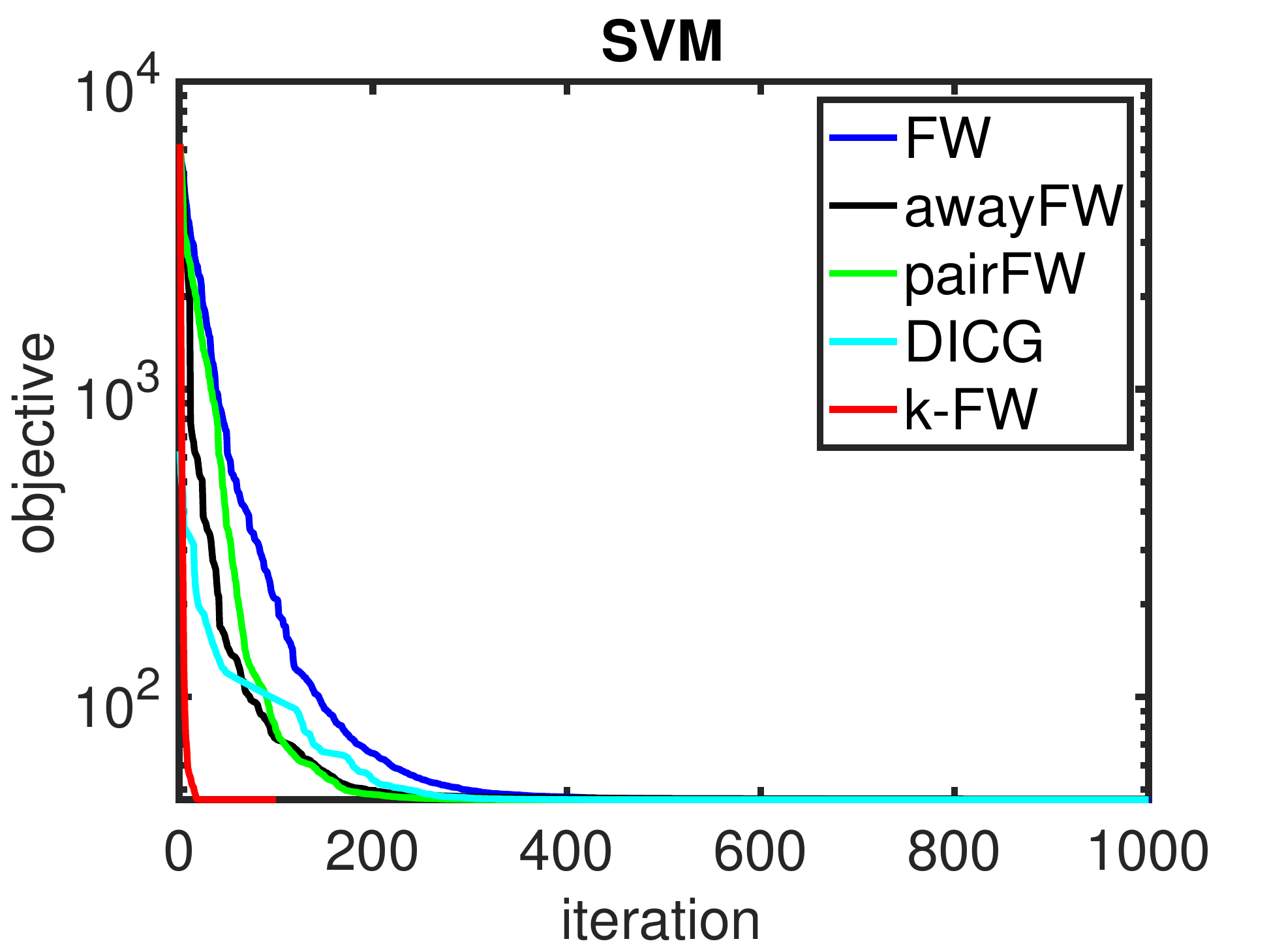}
			%\caption{B}
			\label{fig:fig_SVM}
		\end{subfigure}
		\hspace{3pt}
		\begin{subfigure}[(b)]{.24\textwidth}
			%\centering
			\includegraphics[width=1.1\linewidth]{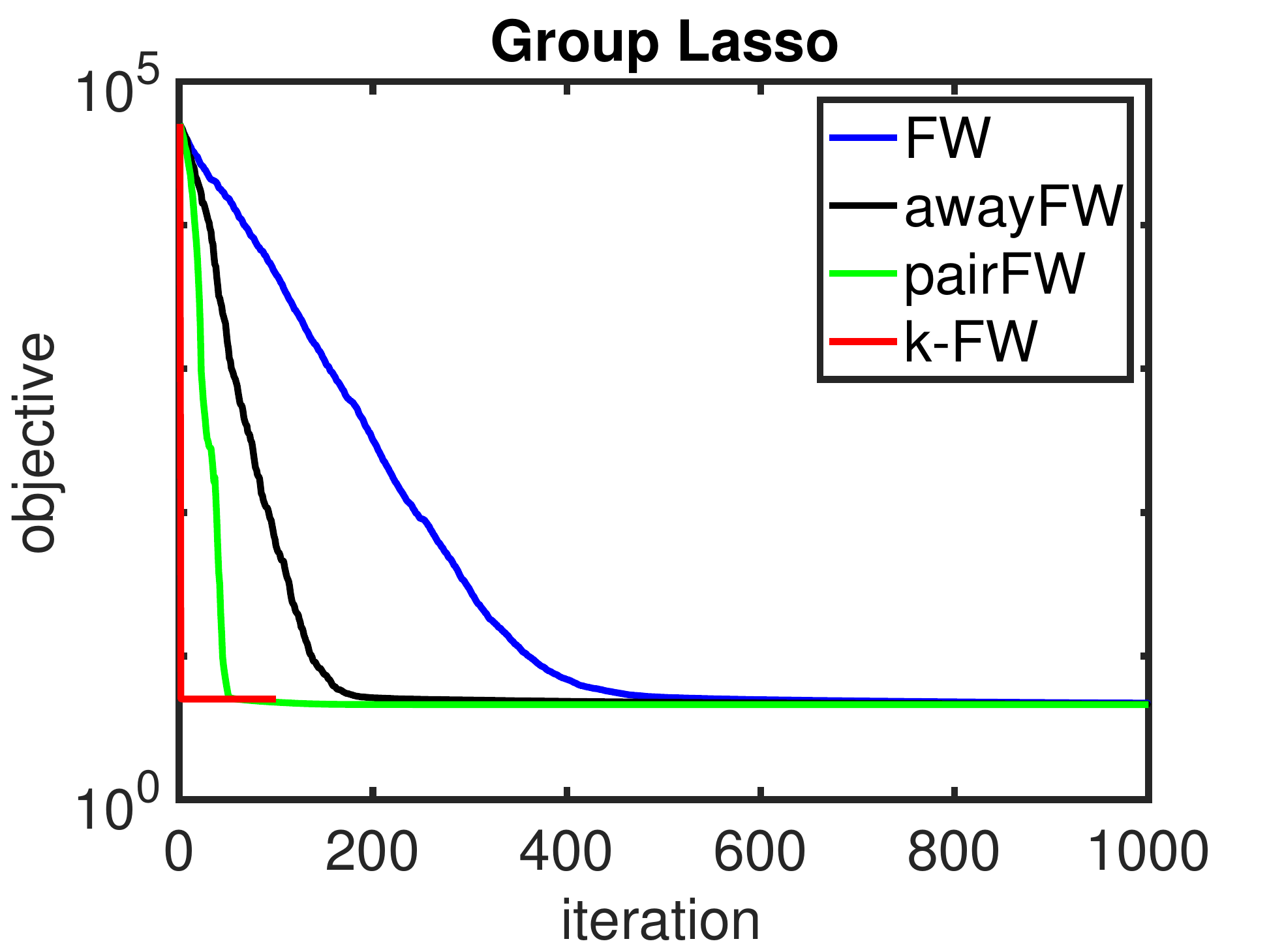}
			%\caption{}
			\label{fig:fig_grouplasso}
		\end{subfigure}
		\hspace{3pt}
		\begin{subfigure}[(b)]{.24\textwidth}
			\centering
			\includegraphics[width=1.1\linewidth]{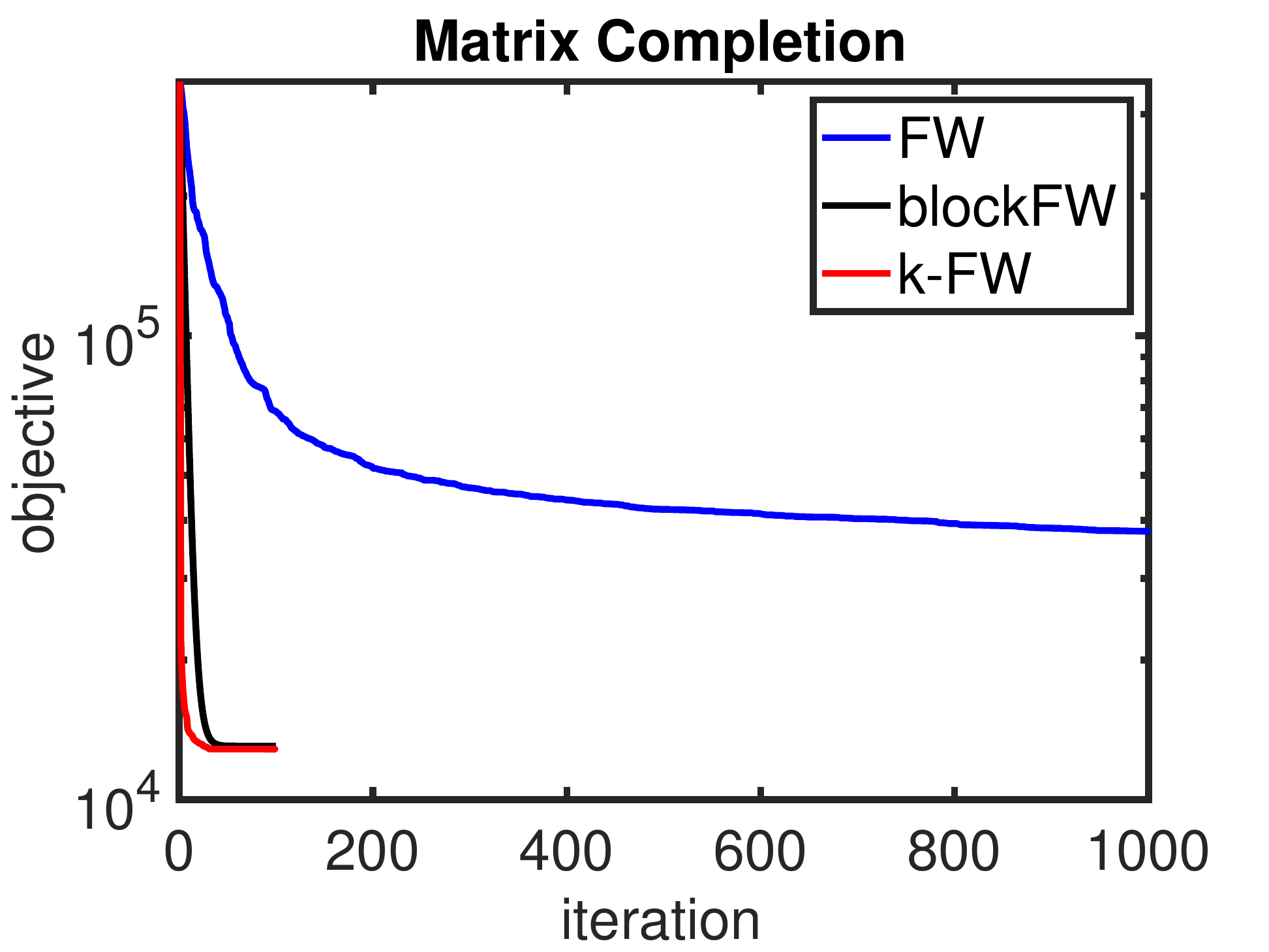}
			%\caption{B}
			\label{fig:fig_MC}
		\end{subfigure}
		\vspace{-15pt}
		\caption{$k$FW vs. FW and its variants }\label{fig: Figure_compare}
		\vspace{-10pt}
	\end{figure}
	\begin{table}[H]\label{Table_time}
		\centering
		\caption{Computation time (seconds):
			the algorithms terminate when the relative change of the objective $<10^{-6}$
			or after $1000$ iterations.
			The dash - means
			the algorithm is not %applicable or not well
			suited to the problem.}\label{table: Table_time}
		\begin{tabular}{lcccccc}
			\hline
			& FW        & awayFW & pairFW & DICG  & blockFW & $k$FW \\
			\hline
			Lasso             &  $>$14    & 7      & 6      & 10    & -       & \textbf{0.5} \\
			SVM               & 6         & 4.5    &  2.9     & 2.5   & -       & \textbf{0.6} \\
			Group Lasso       & 17        & 6      & 1.8   & -     & -       & \textbf{0.3} \\
			Matrix completion & $>$180    & -      &  -     & -     & 1.8     & 4.8 \\
			\hline
		\end{tabular}
	\end{table}
	\paragraph{Time of $k$LOO and $k$DS}
	 {In the experiments, the time cost ratios $k$LOO:$k$DS are approximately: Lasso 0.6:1, SVM 0.03:1, Group Lasso 0.2:1, MC 0.1:1.} \newcontent{The time of $k$FW spends on $k$DS
	occupies a major fraction of the total time. However, the spend is worthwhile as the number of iteration is extremely reduced indicated by our experiments.}
	%\lnote{We need to record the time spent for $k$FW on $k$LOO and $k$DS.}

\subsection{Real data}
First, we randomly choose 5000 samples of each digit of the MNIST \cite{lecun1998gradient} dataset to form a dictionary $A\in\mathbb{R}^{784\times 50000}$. Given an image $b$ from the rest of the dataset, we add Gaussian noise (zero mean and 0.1 variance) to it (denoted by $\bar{b}$) and use sparse coding to denoise, i.e.  $\hat{x}=\textup{argmin}_x\Vert Ax-\bar{b}\Vert^2$ subject to $\Vert x\Vert\leq 2$. In $k$FW, we set $k=50$. In every algorithm, the optimization is terminated if the relative change of the objective function is less than $10^{-4}$ or the iteration number reaches 500. The recovered image is  $\hat{b}=A\hat{x}$. The recovery error is defined as $\textup{RE}=\Vert \hat{b}-b\Vert/\Vert b\Vert$. Table \ref{tab_mnist_denoise} shows three examples. We see that $k$FW is significantly faster than other methods in all cases and the recovery error of $k$FW is much lower than DICG. Figure \ref{fig_dn_027} shows some examples intuitively.

\begin{table}[h!]
\centering
\caption{Examples of denoising on MNIST (TC: time cost (second); RE: recovery error).}
\label{tab_mnist_denoise}
\begin{tabular}{ccccccc}\hline
digit &metric & FW        & awayFW & pairFW & DICG   & $k$FW \\ \hline
\multirow{2}{*}{0} & TC &18.2	&19.1	&18.9 &4.0	&\textbf{1.6} \\ 
& RE & \textbf{0.2634}	& 0.2641	& 0.2645	&0.2665	&0.2639\\ \hline

\multirow{2}{*}{1} & TC &18.7&18.1	&18.2 &5.7	&\textbf{3.2} \\ 
& RE & 0.3725&	0.3764&	0.3731&	0.4010	&\textbf{0.3632}\\ \hline

\multirow{2}{*}{2} & TC &18.3&18.4	&18.1 &6.1	&\textbf{2.3} \\ 
& RE & 0.3272&	0.3281&	0.3267&	0.3383	&\textbf{0.3258}\\ \hline

\multirow{2}{*}{3} & TC &18.2	&18.3	&18.1 &6.1	&\textbf{2.4} \\ 
& RE & \textbf{0.2577}&	0.2607&	0.2587&	0.2653	&0.2581\\ \hline

\multirow{2}{*}{4} & TC &18.4	&18.2	&18.1 &5.3		&\textbf{4.0} \\ 
& RE & \textbf{0.3240}&	0.3273&	0.3263&	0.3256	&0.3249\\ \hline

\multirow{2}{*}{5} & TC &18.7	&18.3	&18.3 &6.5		&\textbf{2.8} \\ 
& RE & 0.3136&	0.3132&	0.3134&	0.3264&	\textbf{0.3110}\\ \hline

\multirow{2}{*}{6} & TC &18.6	&18.4	&18.6	&6.1	&\textbf{3.9} \\ 
& RE & \textbf{0.2772}&	0.2792&	0.2800&	0.2890&	\textbf{0.2772}\\ \hline

\multirow{2}{*}{7} & TC &18.2	&18.2	&18.4	&6.6	&\textbf{3.0} \\ 
& RE & 0.3315&	0.3297&	0.3301&	0.3296&	\textbf{0.3237}\\ \hline

\multirow{2}{*}{8} & TC &17.4	&17.6	&18.4	&2.8	&\textbf{2.3} \\ 
& RE & 0.3072&	\textbf{0.3066}&	0.3072&	0.3555&	0.3069\\ \hline

\multirow{2}{*}{9} & TC &18.3	&18.0	&18.0	&7.6	&\textbf{2.2} \\ 
& RE & 0.3575&	0.3598&	0.3573&	0.3618&	\textbf{0.3535}\\ \hline
\end{tabular}
\end{table}

\begin{figure}
\centering
\includegraphics[width=17cm,trim={120 30 90 10},clip]{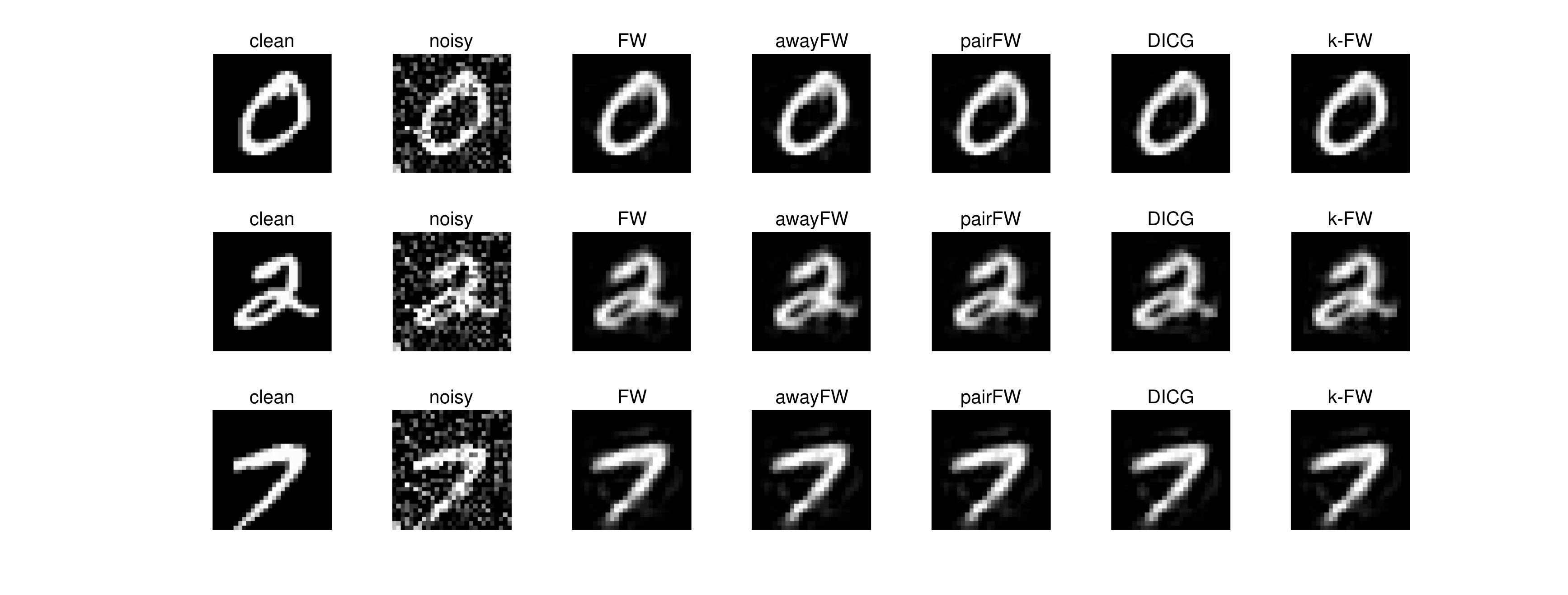}
\caption{}\label{fig_dn_027}
\end{figure}

Second, we consider the SVM classification task on the MNIST dataset. We randomly choose 5000 samples of digit ``0" and 5000 samples of digit ``6". The training-testing ratio is 8:2. In $k$FW, we set $k=50$. The time cost and classification accuracy (average of 10 trials) in the condition of different number of iterations are reported in Table \ref{tab_mnist_class}. With 50 iterations, the SVM solved by $k$FW achieved a classification accuracy of 0.9934 while the accuracies of SVM solved by other algorithms are lower than 0.9. In general, the results in Table \ref{tab_mnist_class} indicate that $k$FW is much more efficient than other algorithms in solving the optimization of SVM.
\begin{table}[h!]
\centering
\caption{SVM (with second-order polynomial kernel, $C=10$, and $\lambda=0.1$) classification for digits ``0" and ``6" of MNIST (TC: time cost (second); Acc: classification accuracy).}
\label{tab_mnist_class}
\begin{tabular}{ccccccc}\hline
iterations &metric & FW        & awayFW & pairFW & DICG   & $k$FW \\ \hline
\multirow{2}{*}{10} & TC &6.6&7.5	&3.8	&3.1	&\textbf{1.2} \\ 
& Acc & 0.5094	& 0.6514	& 0.5691	&0.6650	&\textbf{0.8199}\\ \hline
\multirow{2}{*}{50} & TC &44.0	&48.6	&19.6	&14.9	&\textbf{5.6} \\ 
& Acc & 0.8331&	0.8364&	0.7997&	0.8728&	\textbf{0.9934}\\ \hline
\multirow{2}{*}{200} & TC &239.0	&261.1	&90.1	&60.9	&\textbf{22.6} \\ 
& Acc & 0.9236&	0.9852&	0.9503&	0.9915&	\textbf{0.9966}\\ \hline
\multirow{2}{*}{500} & TC &722.3	&743.5	&265.1	&155.4	&\textbf{55.7} \\ 
& Acc & 0.9834&	0.9931&	0.9916&	0.9957&	\textbf{0.9962}\\ \hline
\end{tabular}
\end{table}

Finally, we consider an inpainting problem for the gray-scale image shown in Figure \ref{fig_inpainting}. We randomly remove $50\%$ of the pixels. Since the image matrix $X_{\text{org}}$ is approximately low-rank, we use $\Vert {X}\Vert_{\text{nuc}}\leq 0.8\alpha$, where $\alpha$ is the value of the nuclear norm of $X_{\text{org}}$. In blockFW and $k$FW, we set $=5$. kIn Figure \ref{fig_inpainting}, we see that blockFW with $\eta=0.1$ outperformed our $k$FW slightly in terms of PSNR. In addition,  the time cost of $k$FW is 2.5 times of blockFW. However, blockFW requires a well determined step size $\eta$.

\begin{figure}[h!]
\includegraphics[width=18cm,trim={50 50 50 50},clip]{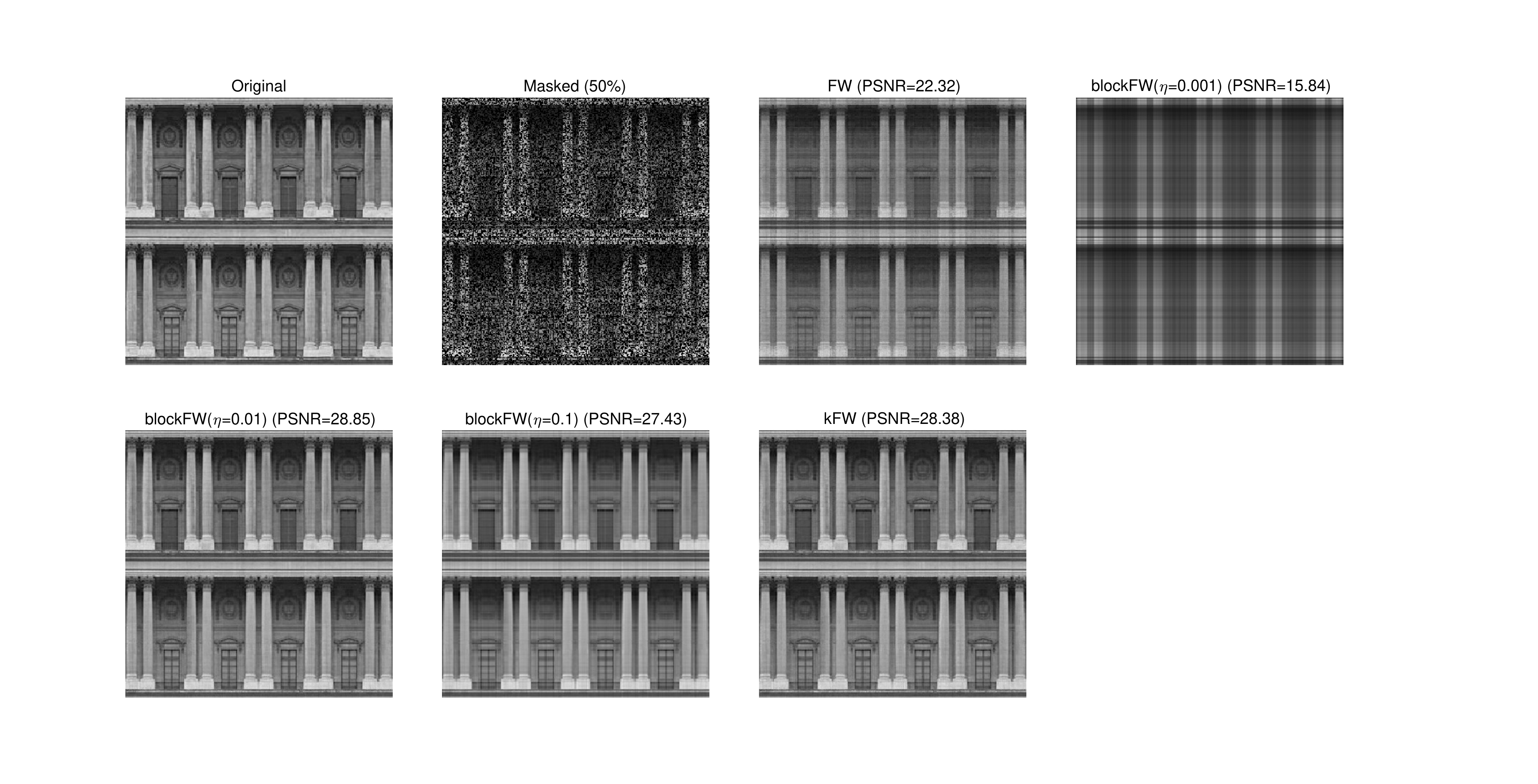}
\caption{Image inpainting by matrix completion with different solvers.}\label{fig_inpainting}
\end{figure}

	\section{Conclusion and discussion}\label{sec: discussion}
	%{what is achieved and then what are directions on future research}
	This paper presented a new variant of FW, $k$FW, that takes advantage
	of sparse structure in problem solutions
	to offer much faster convergence than other variants of FW,
	both in theory and in practice.
	$k$FW avoids the Zigzag phenomenon by optimizing over a convex combination of
	the previous iterate and
	$k$ extreme points of the constraint set, rather than one, at each iteration.
	The method relies on the ability to efficiently compute
	these $k$ extreme points ($k$LOO) and to compute the update ($k$DS),
	which we demonstrate for a variety of interesting problems.

	\newcontent{Apart from the algorithmic advance of the introduction of $k$LOO and $k$DS for various settings, theoretically,  a more uniform, geometric definition of strict complementarity that unifies and extends previous work \cite{DingFeiXuYang2020,garber2019linear,garber2020revisiting}, and allows us to handle a wide range of problems in a coherent framework.}
	
	\paragraph{Related work and comparison}
	\newcontent{A recent line of work \cite{DingFeiXuYang2020,garber2019linear,garber2020revisiting,diakonikolas2020locally,carderera2021parameter} utilizes the concept of 
		strict complementarity or the local geometry of \eqref{opt: mainProblem} near $\xsol$ to show faster convergence when the iterate is near the solution. \cite{garber2019linear} studies vanilla FW for spectrahedron with rank one solution and 
		\cite{DingFeiXuYang2020} shows how to deal with general rank by utilizing $k$LOO and $k$DS (specFW in their language). The work \cite{garber2020revisiting} revisits away-FW and show the method achieves better local convergence rate. In \cite{diakonikolas2020locally,carderera2021parameter}, the authors tries to accelerate away-FW when the iterate is close to the solution for polytope constraint.}

\newcontent{Comparably, these past works are rather specific, in particular, strict complementarity is defined specific to each setting rather than in a uniform way. Nevertheless, the present work is inspired from \cite{DingFeiXuYang2020} and the contribution of the present work is to distill and generalize the ideas there to various settings such as polytope, group norm ball, and nuclear norm ball. In particular, the extension to the nuclear norm from the spectrahedron is important for several reasons: (i) The nuclear norm ball (NNB) formulation is the natural problem form for rectangular matrix recovery problems; (ii) To apply the spectrahedron formulation to NNB formulation would require dilation, which doubles the number of variables. Moreover, a quadratic growth objective does not have quadratic growth after dilation, so existing theory for the spectrahedral case does not apply; (iii) Technically, our analysis is similar to \cite{DingFeiXuYang2020} but introduces several novel elements; note in particular that the SC defined in the present paper generalizes that in \cite{DingFeiXuYang2020}.}
		
\newcontent{The idea of utilizing multiple directions instead of just one is rooted in fully-corrective FW and related variants. It is also explored in recent works such as \cite{allen2017linear,bomze2020active}. \cite{allen2017linear} deals with nuclear norm ball and computes mulitple singular vectors in each iteration in order to make a gradient step. Note that even though \cite{allen2017linear} considers computing $k$ singular vectors, the $k$LOO is not based on the gradient but rather primal iterate - gradient, which may induces some computation difficulty due to the higher rank of iterates. More importantly, it may not converge for $k<r_\star$ as shown in \cite[Figure 1]{allen2017linear}, while kFW converges always as shown in Theorem \ref{thm: normalFWresult}. The work \cite{bomze2020active} considers how to identify the 
vertices on the optimal face via away-FW, however, the result is limited to probability simplex.}

\newcontent{\paragraph{Future work} We expect the core ideas that undergird $k$FW can be generalized
to a wide variety of atomic sets in addition to those
considered in this paper. We also expect the idea of $k$DS and a limited memory Frank-Wolfe, which uses most recent $k$ points found by LOO, 
can still succeed for polytopes with $\rsol$ much larger than the dimension of the optimal face.   
}

  \section*{Acknowledgements}
  This work was supported in part by
NSF Awards IIS-1943131 % NSF CAREER
and CCF-1740822, % TRIPODS phase 1
the ONR Young Investigator Program, %
DARPA Award FA8750-17-2-0101, % D3M
the Simons Institute, % foundations of data science, fall 2018
and Capital One. We would like to thank
Billy Jin and Song Zhou for helpful discussions.

	\bibliography{reference}
	\bibliographystyle{alpha}
%	\newpage

%	\title{Appendices to ``of $k$FW: A Frank-Wolfe style algorithm \\
%		with stronger subproblem oracle''}
	\appendix
%	\maketitle
	\section{Table and Procedures for Section \ref{sec: subproblems}}\label{sec: AppendixSubproblems}
	\subsection{$k$LOO and $k$DS for Spectrahedron}\label{sec: AppendixSubproblemsSpectrahedron}
	%Recall a matrix $A\in \symMat^\dm$
	%is positive semidefinite if all its eigenvalues are nonnegative.
	%We denote a positive semidefinite matrix as $A\succeq 0$ or $A\in \symMat_+^\dm$.
	%We write the eigenvalues of a symmetric matrix $A\in \symMat^\dm$ as
	%$
	%\lambda_1(A)\geq \dots \geq \lambda_{\dm}(A).
	%$

	We define $k$LOO and $k$DS for
	the spectrahedron $\mathcal{SP}^\dm = \{X\in \symMat^{\dm} \mid X\succeq 0,\,\tr(X)=1\}$
	in this section.

	\myparagraph{$k$LOO} Given an input matrix $Y\in \symMat^\dm$, define the $k$ best directions
	of the linearized objective $\min_{V\in \mathcal{SP}_\dm}\inprod{V}{Y}$
	as the bottom $k$ eigenvectors of $Y$,
	the eigenvectors corresponding to the $k$ smallest eigenvalues.
	Call these vectors $v_1,\dots,v_n$ and collect the output as $V=[v_1,\dots,v_k] \in \real^{\dm \times k}$.

	\myparagraph{$k$DS} Take as inputs $W\in \mathcal{SP}^n$ and $V=[v_1,\dots,v_k]\in \real^{n\times r}$ with orthonormal columns.
	Instead of convex combinations of $W$ and $v_iv_i^\top$,
	we consider a spectral variant inspired by \cite{helmberg2000spectral}:
	\[
	X = \eta W+ VS V^\top \quad \text{where} \quad
	\eta \geq 0,~ S \in \symMat_+^k,~\eta + \tr(S)=1.
	\]
	We minimize the objective $f(X)$ over this constraint set to obtain the solution $X_{k\text{DS}}$ to $k$DS:
	\[
	\mbox{minimize} \quad f(\eta W+ V S V^\top)\quad\mbox{subject to}\quad{ \eta \geq 0,~ S \in \symMat_+^k, ~\text{and}~\eta + \tr(S)=1}.
	\]
	Again, we use APG to solve this problem.
	Projection onto the constraint set requires eigenvalue decomposition (EVD)
	of a $k^2$ matrix, which is tolerable for small $k$.
	(See more detail in Section \ref{sec: projSpechedronAndNuclearNorm})% for more details on this projection.)
	%After solving the above minization over  $(\eta,S)$, we can obtain the output
	%$X_{k-\text{DS}}$ immediately.
	\subsection{$k$LOO of combinatorical optimization}\label{sec: tabelCombkLOO}
	In this section, we present Table \ref{table: kbestSolution} of the computational complexity of finding the
	$k$ best solution for combinatorical optimizations. In our setting, the $k$ best
	solution corresponds to the $k$ best directions of $k$LOO. We then point out
	those $k$LOO that can be efficiently computed.

Let us first look at Table \ref{table: kbestSolution} for the complexity of LOO and $k$LOO.
\begin{table}[H]
	\caption{The time complexity of $k$LOO for different combinatorical problems.
		The matroid $M=(E,\mathcal{I})$ consists of the ground set $E$ with $n$ elements and the set of bases $\mathcal{I}$.
		The polytope is the convex hull of all bases in $[0,1]^n$.
		The quantity $\alpha$ is the complexity of checking independence of a set. Here $r(M)$ is the rank of the matroid $M$.
		The $s-t$ cut is for a directed graph with $n$ nodes, $m$  edges, a source node $s$, and a sink node $t$. For each $s-t$ cut, a partition $S,S^c$ of the vertex set with $s\in S$ and $t\in S^c$, we define its cut point as
		a vector in $[-1,1]^{m}$ that has entry $1$ for an edge from $S$ to $S^c$, and an entry $-1$ for an edge from $S^c$ to $S$. The $s-t$ cut polytope is the convex hull of all cut points in $[-1,1]^m$.
		The path polytope considers all simple
		path from $s$ to $t$ for a directed acyclic graph with $n$ nodes and $m$ edges. The polytope is then the convex hull of all
		simple path point in $[0,1]^m$. For an undirected graph with $n$ nodes and $m$ edges, the spanning tree polytope is the convex hull of all spanning tree in $[0,1]^m$.
	}\label{table: kbestSolution}
	\centering
	\begin{tabular}{|c|c|c|}
		\hline
		Polytope name & LOO complexity & $k$LOO complexity\\
		\hline
		Probability simplex & $\bigO(n)$ & $\bigO(n+k)$ \cite{martinez2001optimal} \\
		Polytope of bases of a matroid $M$ & $\bigO(n\log n,n\alpha)$ & $\bigO(n\log n +knr(M)\alpha)$ \cite{hamacher1985k} \\
		The Birkhoff polytope & $\bigO(n^3)$ & $\bigO(kn^3)$ \cite{murthy1968algorithm}\\
		$s-t$ Cut Polytope (Directed Graph)& $\bigO(nm\log n)$ & $\bigO(kn^4)$ \cite{hamacher1985k}\\
		$s-t$ path Polytope(DAG) & $\bigO(m+n\log n)$ &  $\bigO(m+n\log n +kn)$  \cite{eppstein1998finding}\\
		Spanning tree Polytope& $\bigO(m+n\log n)$  & $\bigO(m\log n + k\min(n,k)^{1/2})$ \cite{eppstein1990finding}\\
		\hline
	\end{tabular}
\end{table}

	Let us now list other polytopes with efficient $k$LOO with the assumption that $k\leq n$:
	\begin{itemize}
		%	\item The probability simplex $\Delta ^n =\conv(\{e_i\}_{i=1}^n)$ in $\real^{\dm}$ admits $k$LOO with time complexity $\bigO(n\log k)$
		\item The $\ell_1$ norm ball $\{x\in \real^\dm \mid \sum_{i=1}^n |x_i| \leq \alpha \}$ admits
		a $k$LOO with time complexity $\bigO(n+k)$ by simply considering finding the $k$ largest elements among $2n$ elements.
		\item The spanning tree polytope of an undirected graph $G(V,E)$ in $\real ^{|E|}$ admits
		a $k$LOO with time complexity $\bigO(m\log n + k^2)$,
		where $m =|E|$ and $n = |V|$ \cite{eppstein1990finding}.
		%	\item The $s$-simplex in $\real^{\dm}$, $\{x\in \real^n\mid \conv(\{(\sum_{i=1}^s e_{l_i})\}_{l_1<\dots<l_s})\}$ admits $k$LOO with time complexity $\bigO(\dm \times ? )$. See Section \ref{} for a derivation.
		\item The Birkhoff polytope, the convex hull of permutation matrices in $\real^{\dm \times \dm}$,
		admits a $k$LOO with time complexity $\bigO(kn^3)$ \cite{murthy1968algorithm}
		\item The path polytope of a directed acyclic graph $G(V,E)$ in $\real ^{|E|}$ admits
		a $k$LOO with time complexity $\bigO(m+n\log n +k\min(n,k)^{1/2})$, where $m =|E|$ and $n = |V|$ \cite{eppstein1998finding}.
	\end{itemize}
	Optimization over the probability simplex is useful for fitting support vector machines \cite[Problem (24)]{clarkson2010coresets}.
	The $\ell_1$ norm ball plays a key role in sparse signal recovery \cite{chen2001atomic}.
	The path polytope appears in applications in video-image co-localization \cite{joulin2014efficient}.

	\subsection{Examples of $k$LOO and $k$DS}
	This section presents Table \ref{tb: table: kfwkLMOexample}, which presents examples of efficiently-computable $k$LOO,
  and \ref{table: kfwkKsearchexample}, which presents examples of efficiently-computable $k$DS.
	\begin{table}[tb]
		\caption{$k$LOO examples: The input is a vector $y$ for the polytope and
			unit group norm ball (with base norm $\ell_2$ norm), and a matrix $Y$ for the spectrahedron and unit nuclear norm ball.}
		\label{tb: table: kfwkLMOexample}
		\centering
		\begin{tabular}{lll}
			\toprule
			Name     & $k$ best direction and output & $k$LOO cost \\
			\midrule
			Polytope   &  $k$ extreme points $v_i$s with $k$ smallest &  See Table \ref{table: kbestSolution}  \\
			& $\inprod{v}{y}$ among all extreme points $v$ &  \\
			Unit group & $k$ groups $v_1,\dots,v_k\in \mathcal{G}$ of  & $\bigO((\sum_{i=1}^k|v_i|)+k\log{k})$ \\
			norm  ball & the largest $\ell_2$ norm of $y$  & \\
			Spectral simplex & bottom $k$ eigenvector $v_i$s of $Y$, & Computing  \\
			&  output  $V=[v_1,\dots,v_k]$  & bottom $k$ eigenvectors\\
			Unit nuclear  & top $k$ left, right singular vectors  $(u_i,v_i)$ of $Y$, &  Computing  \\
			norm Ball	&   output $U=[u_1,\dots,u_k]$, $V=[v_1,\dots,v_k]$. &  top $k$ singular vectors\\
			\bottomrule
		\end{tabular}
		\vspace{5pt}
		\caption{$k$ direction search examples.
			We present the parametrization of the vector $x$ or matrix $X$ in the second column.
			The $k$DS optimization problem is to minimize $f(x)$ or $f(X)$ over the parametrization.
			The input is a vector $w$ or a matrix $W$ in $\Omega$ and another of the form output by $k$LOO.
		}
		\label{table: kfwkKsearchexample}
		\centering
		\begin{tabular}{lllll}
			\toprule
			Name     & Parametrization & Parameter & Parameter  & Main cost of\\
			& of $x$ or $X$ & variable & constraint (p.r.) & proj to p.r.\\
			\midrule
			Polytope & $\eta w +\sum_{i=1}^k \lambda_i v_i $  &  $(\eta,\lambda)\in \real^{k+1}$   & $(\eta,\lambda)\in \Delta^{k+1}$  & $\bigO(k\log(k))$ \\
			Unit Group  &  $\eta w + \lambda^{v_1,\dots,v_k}$ &$(\eta, \lambda^{v_1,\dots,v_k})$ & $\eta + \norm{\lambda^{v_1,\dots,v_k}}_{\mathcal{G}}\leq 1$ & $\bigO\left(k\log(k)\right)+$\\
			norm  ball & & $\in \real^{1+n}$ & & $\bigO\left(\sum_{i=1}^k|v_i|\right)$    \\
			Spectrahedron   & $\eta W+ VSV^\top $  & $(\eta,S)\in \real\times\symMat^k$ & $\eta\geq 0,S\succeq 0$  & a  full EVD  \\
			& & &  $\eta+\tr(S)=1$  &  of a $k^2$ matrix \\
			Unit nuclear  & $\eta W+ USV^\top$  &  $(\eta,S)\in \real^{1+k^2}$ & $\eta\geq 0, \eta  +\nucnorm{S}\leq 1$  & a  full SVD of \\
			norm Ball & &  & & a $k^2$ matrix \\

			\bottomrule
		\end{tabular}
	\end{table}

	\subsection{Projection Step in APG for $k$DS of group norm ball} \label{sec: projgroupnormball}
	Here we described the projection procedure in $k$DS for group norm ball when the base norm is $\ell_2$ norm.
	Suppose we want to solve the projection problem given $(\eta_0,\lambda_0^{v_1,\dots,v_k})\in \real^{1+n}$
	with decision variable $\eta$ and $\lambda^{v_1,\dots,v_k}$:
	\begin{equation}%\label{eq: groupnormokDSprojection}
	\mbox{minimize} \quad \twonorm{(\eta_0,\lambda_0^{v_1,\dots,v_k})-(\eta ,\lambda^{v_1,\dots,v_k})}
	\quad \mbox{subject to}\quad
	\eta +\norm{\lambda^{v_1,\dots,v_k}}_{\mathcal{G}} \leq 1,~ \eta\geq 0.
	\end{equation}
	Here we further require that $\lambda_0^{v_1,\dots,v_k}$ and $\lambda^{v_1,\dots,v_k}$
	are supported on  $\cup_{i=1}^k v_i$.
	We denote the optimal solution as $\eta^\star, (\lambda^{v_1,\dots,v_k})^\star$.

	Since $\lambda^{v_1,\dots,v_k}$ is only supported on $\cup_{i=1}^k v_i$, we can consider
	it as a vector in $\real^{v_1+\dots+v_k}$ and $\norm{\lambda^{v_1,\dots,v_k}}_{\mathcal{G}}
	= \sum_{i=1}^{k}\twonorm{\lambda^{v_1,\dots,v_k}_{v_i}}.$ The procedure for projection is as follows:
	\begin{enumerate}
		\item First compute the $(\eta^\star,a^\star)$ that solves
		\begin{equation}
		\begin{array}{ll}
		\mbox{minimize}_{(\eta,a)}& \twonorm{(\eta,a)-(\eta_0,[\norm{[\lambda^{v_1,\dots,v_k}_0]_{v_i}}]_{i=1}^k)} \\
		\mbox{subject to} &(\eta,a)\in \real^{k+1}_+,\quad \eta+\sum_{i=1}^{k}a_i\leq 1.
		\end{array}
		\end{equation}
		Here $\real^{k+1}_+$ is the nonnegative orthant in $\real^{k+1}$.
		\item Next, for each $v_i$, we compute $(\lambda^{v_1,\dots,v_k})^\star_{v_i}$ by solving
		\[
		(\lambda^{v_1,\dots,v_k})^\star_{v_i} =\argmin_{\norm{\lambda_{v_i}^{v_1,\dots,v_k}}\leq a^\star _i}\twonorm{[\lambda_{0}^{v_1,\dots,v_k}]_{v_i}-\lambda_{v_i}^{v_1,\dots,v_k}}.
		\]
	\end{enumerate}
	The first step requires a projection to the convex hull of simplex and $0$ and can be done in time $\bigO(k\log k)$. The second step
	requires projection to $\ell_2$ norm ball which is a simple scaling. The correctness can be verified by decomposing
	each $\lambda_{v_i}^{v_1,\dots,v_k} = \alpha_i w_i$ where $\alpha_i\geq 0$ and $w_i$ has $\ell_2$ norm $1$. For general $\ell_p$ norm,
	one has to find a root of a monotone function. This problem can be solved by bisection \cite{sra2011fast}.
	%Considering the following group lasso problem
	%\begin{equation}
	%\begin{aligned}
	%\mathop{\textup{minimize}}_{X}&\ \dfrac{1}{2}\Vert{Y}-{X}{Z}\Vert_F^2 \\
	%\textup{subject to} &\ \Vert{X}\Vert_{\mathcal{G}}\leq\alpha,
	%\end{aligned}
	%\end{equation}
	%where ${Z}\in\mathbb{R}^{m\times n}$, ${Y}\in\mathbb{R}^{l\times d}$, and ${X}\in\mathbb{R}^{l\times m}$.
	%Then the $k$DS in $k$FW for LRMC aims to solve
	%\begin{equation}\label{pr_kDS_GL_0}
	%\begin{aligned}
	%\mathop{\textup{minimize}}_{{X}}&\ \dfrac{1}{2}\Vert{Y}-(\eta W+{G}){Z}\Vert_F^2 \\
	%\textup{subject to} &\ \Vert{X}\Vert_{\mathcal{G}}\leq\alpha,
	%\end{aligned}
	%\end{equation}
	%where  the $v_1,\ldots,v_k$ columns of ${G}$ are zero. Let ${P}\in\mathbb{R}^{k\times n}$ be a permutation matrix: $P_{iv_i}=1$, $i=1,2,\ldots,k$. Denote
	%$${H}=\left[
	%\begin{matrix}
	%\tfrac{l}{\alpha}{W}{Z}\\
	%{P}{Z}
	%\end{matrix}
	%\right]\quad\text{and} \quad
	%{K}=\left[
	%\begin{matrix}
	%\dfrac{\eta\alpha}{l}{I}_l & \tilde{{G}}
	%\end{matrix}
	%\right],
	%$$
	%where $\tilde{{G}}$ consists of the nonzero columns of ${G}$.
	%Thus ${K}{H}=(\eta{{W}}+{G}){Z}$ and
	%we rewrite \eqref{pr_kDS_GL_0} as
	%\begin{equation}\label{pr_kDS_GL_1}
	%\begin{aligned}
	%\mathop{\textup{minimize}}_{{K}}&\ \dfrac{1}{2}\Vert{Y}-{K}{H}\Vert_F^2 \\
	%\textup{subject to} &\ \Vert{K}\Vert_{\mathcal{G}}\leq\alpha.
	%\end{aligned}
	%\end{equation}
	%Solving \eqref{pr_kDS_GL_1} with APG and get $\eta$ and ${G}$ accordingly.

	\subsection{Discussion on the norm of group norm ball}\label{sec: discussionOntheNorm}
	For the main Theorems \ref{thm: normalFWresult}, \ref{thm: nonexponentialFiniteConvergence}, and
	\ref{thm: LinearConvergence}, the results holds for any arbitrary norm.
	The positive gap in Lemma \ref{lem: positiveDelta} also holds for an arbitrary norm.
	However, the authors have not been able to verify whether strict complementarity
	implies quadratic growth
	for norms other than the $\ell_2$ norm.

	\subsection{Projection Step in APG for $k$DS of spectrahedron, and nuclear norm ball} \label{sec: projSpechedronAndNuclearNorm}
	We consider how to compute the projection step of $k$DS for the spectrahedron and nuclear norm ball.

	\paragraph{Spectrahedron} We want to find $(\eta^\star,S^\star)$ that solves
	\[
	\text{minimize}\quad \twonorm{(\eta,S)-(\eta_0,S_0)},\quad \text{subject to}\quad S\in \symMat_+^k,\, \eta\geq 0,\, \tr(S)+\eta =1.
	\]
	Here $\twonorm{(\eta,S)}= \sqrt{\eta^2+ \fronorm{S}^2}$.
	The procedures are as follows:
	\begin{enumerate}
		\item Compute the eigenvalue decomposition of $S_0 = V\Lambda_0 V^\top$,
		where $\Lambda_0\in \symMat^k$ is a diagonal matrix with diagonal $\vec{\lambda}_0=( \lambda_1,\dots,\lambda_k)$.
		\item Compute $(\eta^\star, \vec{\lambda}^\star) = \arg\min_{(\eta,\vec{\lambda})\in \Delta^{k+1}}\twonorm{(\eta_0,\vec{\lambda}_0)-(\eta,\vec{\lambda})}$.
		\item Form $S^\star = V\diag(\vec{\lambda^\star})V^\top$.
		Here $\diag(\lambda)$ forms a diagonal matrix with the vector $\lambda$ on the diagonal.
		% and leaves other places of the matrix $0$.ß
	\end{enumerate}
	The main computational step is the eigenvalue decomposition which requires $\bigO(k^3)$ time.
	The correctness of the procedure
	can be verified as in \cite[Lemma 3.1]{allen2017linear}
	and \cite[Lemma 6]{garber2019convergence}.

	\paragraph{Unit nuclear ball} We want to find $(\eta^\star,S^\star)$ that solves
	\[
	\text{minimize}\quad \twonorm{(\eta,S)-(\eta_0,S_0)},\quad \text{subject to}\quad \eta+\nucnorm{S}\leq 1,\, \eta\geq 0.
	\]
	The procedures are as follows:
	\begin{enumerate}
		\item Compute the singular value decomposition of $S_0 = U\Lambda_0 V^\top$,
		where $\Lambda_0\in \symMat_+^k$ is a diagonal matrix with diagonal $\vec{\lambda}_0=( \lambda_1,\dots,\lambda_k)$.
		\item Compute $(\eta^\star, \vec{\lambda}^\star) = \arg\min_{(\eta,\vec{\lambda})\in \Delta^{k+1}}\twonorm{(\eta_0,\vec{\lambda}_0)-(\eta,\vec{\lambda})}$.
		\item Form $S^\star = U\diag(\vec{\lambda^\star})V^\top$.
		Here $\diag(\vec{\lambda^\star})$ forms a diagonal matrix with the vector $\vec{\lambda^\star}$ on the diagonal.

	\end{enumerate}
	The main computational step is the singular value decomposition which requires $\bigO(k^3)$ time.
	The correctness of the procedure
	can be verified as in \cite[Lemma 3.1]{allen2017linear}
	and \cite[Lemma 6]{garber2019convergence}.
	%Suppose we observed a few entries of a low-rank or approximately low-rank matrix ${Y}\in\mathbb{R}^{n_1\times n_2}$. We consider the following LRMC problem
	%\begin{equation}
	%\begin{aligned}
	%\mathop{\textup{minimize}}_{{X}}&\ \dfrac{1}{2}\Vert{M}\odot({X}-{Y})\Vert_F^2 \\
	%\textup{subject to} &\ \Vert{X}\Vert_\ast\leq\alpha,
	%\end{aligned}
	%\end{equation}
	%where ${M}$ is a binary matrix indicating the locations of the observed entries. Then the $k$DS in $k$FW for LRMC aims to solve
	%\begin{equation}\label{pr_kDS_LRMC_0}
	%\begin{aligned}
	%\mathop{\textup{minimize}}_{\eta,{S}}&\ \dfrac{1}{2}\Vert{M}\odot(\eta{W}+{U}{S}{V}^T-{Y})\Vert_F^2 \\
	%\textup{subject to} &\ \eta\alpha+\Vert{S}\Vert_\ast\leq\alpha.
	%\end{aligned}
	%\end{equation}
	%Denote
	%$${Z}=\left[\begin{matrix}
	%\hat{{Z}} & {0}\\
	%{0} & \check{{Z}}
	%\end{matrix}\right],$$
	%where $\hat{{Z}}\in\mathbb{R}$ and $\check{{Z}}\in\mathbb{R}^{k\times k}$. Thus $\Vert{Z}\Vert_\ast=\vert{\hat{Z}}\vert+\Vert\check{{Z}}\Vert_\ast$.
	%We reformulate \eqref{pr_kDS_LRMC_0} as
	%\begin{equation}\label{pr_kDS_LRMC_1}
	%\begin{aligned}
	%\mathop{\textup{minimize}}_{{Z}}&\ \dfrac{1}{2}\Vert{M}\odot(\hat{{Z}}{W}/\alpha+{U}\check{{Z}}{V}^T-{Y})\Vert_F^2 \\
	%\textup{subject to} &\ \Vert{Z}\Vert_\ast\leq\alpha.
	%\end{aligned}
	%\end{equation}
	%Then we solve \eqref{pr_kDS_LRMC_1} with APG and have
	%$$\eta={\hat{Z}}/\alpha, \quad {S}={\check{Z}}.$$

	\section{Examples, lemmas, tables, and Proofs for Section \ref{sec: theoreticalGuarantees}}\label{sec: AppendixtheoreticalGuarantees}

	\subsection{Further discussion on strict complementarity}\label{sec: strictComplementarityFurtherDiscussion}
	We give two additional remarks on the strict complementarity.
	\begin{enumerate}
		\item Traditionally, the boundary location condition $x\in \partial \Omega$
		is not included in the definition of strict complementarity.
		We include this condition for two reasons:
		first,  the extra location condition excludes the trivial case that the dual solution of \eqref{opt: mainProblem} is $0$, and $\xsol$ in the interior of $\Omega$,
		in which case FW can be proved to converges linearly \cite{garber2015faster}; second, as we shall see in Example \ref{example: robustitcity},
		such assumption ensures the robustness of the sparsity of $\xsol$.
		\item Strict complementarity (without the boundary location condition) holds generically:
		more precisely, it holds for almost all $c$ in our optimization problem \eqref{opt: mainProblem},
		$\min_{x\in\Omega} g(\Amap x) +\inprod{c}{x}$, \cite[Corollary 3.5]{drusvyatskiy2011generic}.
	\end{enumerate}

	\begin{exmp} \label{example: robustitcity}
		Consider the problem \[
		\min_{x\in \alpha \Delta^n} \frac{1}{2}\norm{x-e_1-\frac{1}{n}\onevec}^2.
		\]
		Here $\onevec$ is the all one vector and $\alpha>0$.
		If we set $\alpha=1$, then $\xsol =e_1$ and the gradient $\nabla f(\xsol) = -\frac{1}{n}\onevec$.
		Hence
		we see that strict complementarity does not hold, using Lemma \ref{lem: positiveDelta}.
		In this case, even though $\xsol = e_1$ is sparse for $\alpha=1$,
		the solution is no longer sparse when $\alpha$ is slightly larger than $1$.
		Hence, we see a perturbation to the constraint can cause instability of the sparsity
		when strict complementarity fails.
	\end{exmp}

	\subsection{Lemmas and tables for strict complementarity} \label{sec: gappositiveproofs}
	In this section, we show that the gap quantity defined in Definition \ref{def: strictComplementarity} is indeed
	positive when strict complementarity holds. We then present a table of summarizing the notations $\mathcal{F}(\xsol)$,
	$\mathcal{F}^c(\xsol)$, and the gap $\delta$.

	Here, for the group norm ball,
  we consider a general norm denoted as $\norm{\cdot}$ which is not necessarily the Euclidean $\ell_2$ norm.
  The dual norm of $\norm{\cdot}$ is defined as $\norm{x}_* = \max _{\norm{y}\leq 1}\inprod{y}{x}$.
  We note here the group norm ball is assumed to have radius one.
	\begin{lem}\label{lem: positiveDelta}
		When $\Omega$ is a polytope, group norm ball, spectrahedron, and nuclear norm ball,
		if strict complementarity holds for Problem \eqref{opt: mainProblem},
		then the gap $\delta$ is positive. Moreover,
		we can characterize the gradient at the solution and the size of the gap in each case:
		\begin{itemize}
			\item Polytope: order the vertices $v\in \Omega$ according to the inner products $\inprod{\nabla  f(\xsol)}{v}$
			in ascending order as $v_1,\dots,v_{\rsol},\dots, v_{l}$ where $l$ is the total number of vertices.
			Then
			$\inprod{\nabla  f(\xsol)}{v_{i}}$, $i=1,\dots,\rsol$ are all equal and the gap
			$\delta$ is
			$\delta = \inprod{\nabla  f(\xsol)}{v_{\rsol+1}}-\inprod{\nabla  f(\xsol)}{v_{\rsol}}$.

			\item Group norm ball for arbitrary base norm:
			order vectors $[\nabla f(\xsol)]_g$, $g\in \mathcal{G}$ according to their dual norm in descending order as $[\nabla f(\xsol)]_{g_1}$,\dots,$[\nabla f(\xsol)]_{g_{|\mathcal{G}|}}$.
			Then $\norm{[\nabla f]_{g_i}}_*$, $i=1,\dots,\rsol$ are all equal, and
			the gap $\delta$ is
			$\delta = \norm{[\nabla f(\xsol)]_{g_{\rsol}}}_*-\norm{[\nabla f(\xsol)]_{g_{\rsol+1}}}_*$.

			\item Spectrahedron: The smallest $\rsol$ eigenvalues of $\nabla f(\Xsol)$ are all equal and
			$
			\delta = \lambda_{\dm-\rsol}(\nabla f(\Xsol)) -\lambda_{\dm-\rsol +1}(\nabla f(\Xsol))
			$
			\item Nuclear norm ball: The largest $\rsol$ singular values of $\nabla f(\Xsol)$ are all equal and
			$
			\delta = \sigma_{\rsol}(\nabla f(\Xsol)) -\sigma_{\rsol +1}(\nabla f(\Xsol)).
			$
		\end{itemize}
	\end{lem}

	\begin{proof}
		Let us first consider the polytope case.
		\myparagraph{Polytope} Since the constraint set is a polytope and $\xsol \in \partial \Omega$, we know the smallest face $\mathcal{F}(\xsol)$ containing $x$
		is proper and admits a face-defining inequality $\inprod{a}{x} \leq b$ for some $a\in \real^{\dm}$ and $b\in \real$. That is, $\mathcal{F}(\xsol) = \{x\mid \inprod{a}{x}=b\}\cap \Omega$ and
		for every $x\in \Omega$, $\inprod{a}{x}\leq b$.
		In particular, this implies that (1) for any vertex $v$ that is not in $\mathcal{F}(\xsol)$,
		$\inprod{a}{v}<b$, and (2) $\inprod{a}{\xsol}=b$.

		Let us now characterize the normal cone $N_{\Omega}(\xsol)$.
		Let $\mathcal{V}$ be the set of vertices in $\Omega$.
		Since $\Omega$ is bounded, we know that every
		point in $\Omega$ is a convex combination of the vertices.
		Hence $N_{\Omega}(\xsol)$ is the set
		of solutions $g$ to the following linear system:
		\begin{equation}\label{eq: normalconepolytope}
		\inprod{g}{v}\leq \inprod{g}{\xsol},\quad \text{for all}\quad v\in \mathcal{V}.
		\end{equation}
		Since $\mathcal{F}(\xsol)$ is the smallest face containing $\xsol$, we know
		that $\xsol \in \mbox{relint}(\mathcal{F}(\xsol))$, and so the description of
		normal cone  $N_{\Omega}(\xsol)$ in \eqref{eq: normalconepolytope} reduces to
		\begin{align}
		\inprod{g}{v_1} & = \inprod{g}{\xsol},\quad \text{for all}\quad v_1\in \mathcal{F}(\xsol), \label{eq: normalconepolytopeReducedF}\\
		\inprod{g}{v_2} & \leq \inprod{g}{\xsol},\quad \text{for all}\quad v_2\text{ being vertices of } \mathcal{F}^c(\xsol). \label{eq: normalconepolytopeReducedNotF}
		\end{align}
		Note that the vector $a$ in the face-defining inequality satisfies \eqref{eq: normalconepolytopeReducedF}
		and satisfies \eqref{eq: normalconepolytopeReducedNotF} with strict inequality as we just argued.
		Hence, the relative interior of $N_{\Omega}(\xsol)$ consists of those
		vectors $g$ that satisfy \eqref{eq: normalconepolytopeReducedF}
		and satisfy \eqref{eq: normalconepolytopeReducedNotF} with a strict inequality.
		As $-\nabla f(\xsol)\in \relint(N_{\Omega}(\xsol))$, we know by the previous argument that
		$-\nabla f(\xsol)$ satisfies \eqref{eq: normalconepolytopeReducedNotF} with strict
		inequality, which is exactly the condition $\delta>0$.
		We arrive at the formula for $\delta$ by noting that
		$\inprod{\nabla f(\xsol)}{v}=\inprod{\nabla f(\xsol)}{\xsol}$
		for every $v\in \mathcal{F}(\xsol)$ due to \eqref{eq: normalconepolytopeReducedF}.

		\myparagraph{Group norm ball}
		Again, recall we here define the group norm ball using any general norm $\norm{\cdot}$.
		% and not just an $\ell_2$ norm induced by the dot product.
		The normal cone at $\xsol$ for unit group norm ball is defined as
		\[
		N_{\Omega}(\xsol)=\{y\mid \inprod{y}{x}\leq \inprod{y}{\xsol}, \;\text{for all}\; \sum_{g\in \mathcal{G}} \|x_g\|\leq 1\}.
		\]
		Standard convex calculus reveals the following properties:
		\begin{enumerate}
			\item The normal cone is a linear multiple of the subdifferential for $\xsol \in \partial \Omega$:
			$N_{\Omega} (\xsol)= \{y\mid y \in \lambda \partial \norm{\xsol}_{\mathcal{G}},\,\lambda \geq 0\}$.
			\item The product rule applies to $\partial \norm{\xsol}_{\mathcal{G}}$ as $\mathcal{G}$ forms a partition: $\partial \norm{\xsol}_{\mathcal{G}} = \prod_{g\in \mathcal{G}}\partial \norm{(\xsol)_g}$.
			\item Any vector in the subdifferential of
			a group $g$ in the support of the solution has norm 1:
			for every $g\in \mathcal{F}(\xsol)$ and every $y_g\in \partial \norm{(\xsol)_g}$,
			$\norm{y_g}_*=1$, and $\inprod{y_g}{(\xsol)_g} = \norm{(\xsol)_g}$.
			\item The subdifferential for groups $g$ not in the support is a unit dual norm ball: for every $g\not\in \mathcal{F}(\xsol)$,
			$\partial \norm{(\xsol)_g} = \mathbf{B}_{\norm{\cdot}_*}:\,=\{y_g\in \real^{|g|} \mid \norm{y_g}_*\leq 1 \}$.
		\end{enumerate}

		The above properties reveal that the normal cone is the set
		\begin{equation}\label{eq: unitgroupnormballnormalcone}
		N_{\Omega}(\xsol)=\big\{y\mid y\in \lambda \left(\prod_{g\in \mathcal{F}(\xsol)}\partial\norm{(\xsol)_g} \times \prod_{g\in \mathcal{G}\setminus\mathcal{F}(\xsol)}\mathbf{B}_{\norm{\cdot}_*}\right), \lambda\geq 0 \big\},
		\end{equation}
		where for every $g\in \mathcal{F}(\xsol)$ and every $y_g\in \partial \norm{(\xsol)_g}$, $\norm{y_g}_*=1$.
		Hence, we know that the relative interior of $N_{\Omega}(\xsol)$ is simply
		\begin{equation}\label{eq: normaconerelativeinterior}
		\begin{aligned}
		&\mbox{relint}\left(N_{\Omega}(\xsol)\right) \\
		=&
		\big\{y\mid y\in \lambda
		\left(
		\prod_{g\in \mathcal{F}(\xsol)}
		\mbox{relint}\left(\partial\norm{(\xsol)_g}\right)
		\times
		\prod_{g\in \mathcal{G}-\mathcal{F}(\xsol)}
		\mbox{relint}\left( \mathbf{B}_{\norm{\cdot}_*}\right)
		\right),
		\lambda > 0
		\big\},
		\end{aligned}
		\end{equation}
		where for every $g\in \mathcal{F}(\xsol)$,
		and every $y_g\in \mbox{relint}\left(\partial \norm{(\xsol)_g}\right)$, $\norm{y_g}_*=1$,
		and for every $g\in \mathcal{G}- \mathcal{F}(\xsol)$,
		and every $y_g \in \relint\left(\mathbf{B}_{\norm{\cdot}_*}\right)$,
		$\norm{y_g}_*<1$. Because of the strict inequality of $\lambda$ in
		\eqref{eq: normaconerelativeinterior}, and strict inequality for
		$\norm{y_g}_*<1$ for $y_g \in \relint\left(\mathbf{B}_{\norm{\cdot}_*}\right)$, we see that
		\begin{equation}\label{eq: nablagradientgroupdualnorm}
		\begin{aligned}
		&\norm{[\nabla f(\xsol)]_{g_{1}}}_* =\dots =\norm{[\nabla f(\xsol)]_{g_{\rsol}}}_*,\text{  and  }\\
		&\norm{[\nabla f(\xsol)]_{g_{\rsol}}}_*-\norm{[\nabla f(\xsol)]_{g_{\rsol+1}}}_*>0
		\end{aligned}
		\end{equation}
		as $-\nabla f(\xsol) \in \mbox{relint}\left(N_{\Omega}(\xsol)\right)$.
		Using the condition that
		for every $g\in \mathcal{F}(\xsol)$ and every $y_g\in \partial \norm{(\xsol)_g}$, $\norm{y_g}_*=1$, and $\inprod{y_g}{(\xsol)_g} = \norm{(\xsol)_g}$,
		we know $\inprod{-\nabla f(\xsol)}{\xsol} = \norm{[\nabla f(\xsol)]_{g_{\rsol}}}_*$.
		Furthermore, using generalized Cauchy-Schwarz, it can be proved that
		$\min_{x \in \mathcal{F}^c(\xsol)}\inprod{\nabla f(\xsol)}{x} = -\norm{[\nabla f(\xsol)]_{g_{\rsol+1}}}_*$.
		Hence, combining the two equalities with \eqref{eq: nablagradientgroupdualnorm},
		we see that $\delta>0$ and arrive at the stated formula for $\delta$.

		\myparagraph{Spectrahedron} We first note that $\Xsol\in \partial \Omega$ and $\tr(X)=1$
		imply that $1\leq  \rsol <\dm$.  To compute the normal cone, we can apply
		the sum rule of subdifferentials to
		\[
		\chi(\{X\in \symMat^\dm\mid \tr(X)=1 \}) + \chi(X \succeq 0),
		\]
		where $\chi$ is the characteristic function,
		which takes value $0$ for elements belonging to the set
		and $+\infty$ otherwise) of $\{X\in \symMat^\dm\mid \tr(X)=1 \}$ and $\symMat_+^\dm$ and reach
		\begin{equation}\label{eq: sumRuleSubdifferentialSpec}
		N_{\Omega}(\Xsol) = \{sI\mid s\in \real\} +\{-Z\mid Z\succeq 0,\range(Z)\subseteq \nullspace(\Xsol)\}.
		\end{equation}
		We note that the sum rule for the relative interior is valid here because $\frac{1}{n}I$ belongs to the interior of both sets.
    Applying the sum rule to \eqref{eq: sumRuleSubdifferentialSpec},
		we find that
		\begin{equation*} %\label{eq: sumRuleRelativeInteriorSpecXnull}
		\mbox{relint}(N_{\Omega}(\Xsol) )= \{sI\mid s\in \real\} + \{-Z\mid Z\succeq 0,\range(Z)= \nullspace(\Xsol)\}.
		\end{equation*}
		Or equivalently,
		\begin{equation*} %\label{eq: sumRuleRelativeInteriorSpec}
		\mbox{relint}(N_{\Omega}(\Xsol) )= \{sI\mid s\in \real\} + \{-Z\mid Z\succeq 0,\nullspace(Z)= \range(\Xsol)\}.
		\end{equation*}
		Using the above equality %\eqref{eq: sumRuleRelativeInteriorSpec}
		and $-\nabla f(\Xsol) \in \mbox{relint}(N_{\Omega}(\Xsol))$, we know there are $\zsol\succeq 0$ and $\ssol\in \real$ that
		\begin{equation}\label{eq: sumRuleRelativeInteriorSpec}
		\nabla f(\Xsol)= -\ssol I +\zsol, \text{ and } \nullspace(\zsol) = \range(\Xsol).
		\end{equation}
		Denote the eigenspace corresponding to the smallest $\rsol$ values of $\nabla f(\Xsol)$ as $\EV_{\rsol}(\nabla f(\Xsol))$.
		From \eqref{eq: sumRuleRelativeInteriorSpec},
		it is immediate that
		\[
		\EV_{\rsol}(\nabla f(\Xsol)) = \range(\Xsol).
		\]
		Moreover, from \eqref{eq: sumRuleRelativeInteriorSpec}, we also have
		\begin{equation} \label{eq: gapOnEigenValues}
		\begin{aligned}
		&\lambda_{\dm-\rsol+1}(\nabla f(\Xsol))- \lambda_{\dm-\rsol}(\nabla f(\Xsol))>0, \quad \text{and}\\
		&\inprod{\nabla f(\Xsol)}{\Xsol} = -\ssol=
		\lambda_{n-i+1}(\nabla f(\Xsol)), \;i=1,\dots,\rsol.
		\end{aligned}
		\end{equation}
		Combining \eqref{eq: gapOnEigenValues} and the well-known fact that
		\[
		\min_{X\in \Omega,\,\range(X)\perp \EV_{\rsol}(\nabla f(\Xsol))} \inprod{\nabla f(\Xsol)}{X} =
		\lambda_{\dm-\rsol+1}(\nabla f(\Xsol)),
		\]
		we see that $\delta$ is indeed positive, and the formula for $\delta$ holds.
		\myparagraph{Nuclear norm ball} We first note that $\Xsol\in \partial \Omega$ imply that
		$1< \rsol <\min(\dm_1,\dm_2)$, and $\nucnorm{\Xsol}=1$. Let the singular value decomposition of $\Xsol$ as
		$\Xsol = U\Sigma V$ with $U\in\real^{\dm_1\times \rsol}$ and $V\in \real^{\dm_2\times \rsol}$.
		The normal cone of the unit nuclear norm ball is
		\begin{equation}
		N_{\Omega}(\Xsol) = \{Y \mid Y =\lambda Z, \, Z= UV^\top + W,\,  W^\top U =0, \,  W V = 0,\,\opnorm{W}\leq 1\,\text{and}\,\lambda\geq 0 \}.
		\end{equation}
		Hence, the relative interior is
		\begin{equation}\label{eq: DescriptionOfRelativeInteriorOfNuc}
		N_{\Omega}(\Xsol) = \{Y \mid Y =\lambda Z, \, Z= UV^\top + W,\,  W^\top U =0,   WV = 0,\opnorm{W} < 1\,\text{and}\,\lambda> 0 \}.
		\end{equation}
		Since $-\nabla f(\Xsol)\in \mbox{relint}(N_{\Omega}(\Xsol))$, we know immediately that
		\begin{align}
		&\sigma_{\rsol}(\nabla f(\Xsol))-\sigma_{\rsol +1}(\nabla f(\Xsol))>0,
		\end{align}
		and the top $\rsol$ left and right singular vectors of $\nabla f(\Xsol)$ are just the columns of
		$-U$ and $V$, and $\inprod{\nabla f(\Xsol)}{\Xsol}=-\sigma_{i}(\nabla f(\Xsol))$ for $i=1,\dots,\rsol$.
		Combining pieces and the standard fact that
		\[
		\min_{\range(X)\perp \range(U),\,\nucnorm{X}\leq 1 }\inprod{\nabla f(\Xsol)}{X}=-\sigma_{\rsol +1}(\nabla f(\Xsol)),
		\]
		we see the gap $\delta$ is indeed positive and the formula is correct.
		%Note that $\inprod{y}{x}$ satisfies that
		%\begin{align}
		%Note that $\inprod{y}{x}$ satisfies that
		%\begin{align}
		%\inprod{y}{x} & =\sum_{g\in \mathcal{G}}\inprod{y_g}{x_g}, & (\mathcal{G}\text{ is a partition of $[\dm]$})  \\
		%              & \leq \sum_{g\in \mathcal{G}}\norm{y_g}_*\norm{x_g},  &(\text{general Cauchy-Schwarz})\label{eq: inequalityNormeq1}\\
		%              & \leq \max_{g}\norm{y_g}_*\cdot \sum_{g\in \mathcal{G}}\norm{x_g} &\label{eq: inequalityNormeq2}\\
		%              &  \leq \max_{g}\norm{y_g}_* \label{eq: inequalityNormeq3} &( \sum_{g\in \mathcal{G}}\norm{x_g}\leq 1).
		%\end{align}
		%Note that the inequality is attained if and only if $\frac{x_g}{\norm{x_g}} \in \partial \norm{y_g}_*$ for each $g\in \argmax_{g\in \mathcal{G}}\norm{y_g}_*$, and $x_g = 0$ for those
		%$g \not\in \argmax_{g\in \mathcal{G}}\norm{y_g}_*$.
	\end{proof}
	A table of the notions $\mathcal{F}(\xsol), \mathcal{F}^c(\xsol)$, and the formula of gap $\delta$ is shown as
	Table \ref{table: supportComplementarityGapDelta}.
	\begin{table}[H]
		\centering
		\caption{For several constraint sets $\Omega$,
			this table describes the support set $\mathcal{F}(\xsol)$, its complementary set $\mathcal{F}^c$,
			and the gap $\delta$.
			Recall the gap $\delta = \min \{ \inprod{u}{\nabla f(\xsol)} -\inprod{\xsol}{\nabla f(\xsol)}\mid u\in \mathcal{F}^c(\xsol)\subseteq\Omega\}$ and admits
			a specific formula as described in Lemma \ref{lem: positiveDelta}.
			We denote the gradient at $\xsol$ as $\nabla _\star$.
			For a polytope,
			we order the vertices $v\in \Omega$ according to their inner products $\inprod{\nabla  f(\xsol)}{v}$
			in ascending order as $v_1,\dots,v_{\rsol},\dots, v_{l}$, where $l$ is the number of vertices.
			For the group norm ball,
			we order vectors $[\nabla f(\xsol)]_g$, $g\in \mathcal{G}$ according to their $\ell_2$ norm
			in descending order
			as $[\nabla f(\xsol)]_{g_1}$,\dots,$[\nabla f(\xsol)]_{g_{|\mathcal{G}|}}$.}\label{table: supportComplementarityGapDelta}.
		\begin{tabular}{|c|c|c|c|}
			\hline
			Constraint $\Omega$ & $\mathcal{F}(\xsol)$ & $\mathcal{F}^c$ & $\delta$ formula \\
			\hline
			polytope	 & smallest face & convex hull of all  &  $\inprod{\nabla  f(\xsol)}{v_{\rsol+1}}$\\
			& containing $\xsol$ & the vertices not in $\mathcal{F}(\xsol)$ & $-\inprod{\nabla  f(\xsol)}{v_{\rsol}}$ \\
			\hline
			group norm ball & $\{g\in \mathcal{G}\mid (\xsol)_g\not=0\}$ & $\{x\mid x_g=0, \forall g\in \mathcal{F}(\xsol)\}$ & $\twonorm{[\nabla_\star]_{g_{\rsol}}}$\\
			&                   & &         $-\twonorm{[\nabla_\star]_{g_{\rsol+1}}}$ \\
			\hline
			Spectrahedron & $\range(\Xsol)$ & $\{X\in [\range(\Xsol)]^\perp\}\cap \mathcal{SP}^n$ & $\lambda_{\dm-\rsol}(\nabla_\star)$    \\
			& &  & $ -\lambda_{\dm-\rsol +1}(\nabla_\star)$    \\
			\hline
			Nuclear norm ball & $\range(\Xsol)$ & $\{X\in [\range(\Xsol)]^\perp\}\cap\mathbf{B}_{\nucnorm{\cdot}}$ & $\sigma_{\rsol}(\nabla_\star)$    \\
			& &  & $ -\sigma_{\rsol +1}(\nabla_\star)$    \\
			\hline
		\end{tabular}
	\end{table}

	\subsection{Quadratic growth under strict complementarity}\label{sec: qgfromstrcomp}
	This section develops that quadratic growth does hold under strict complementarity and
	the condition $g$ in \eqref{opt: mainProblem} is strongly convex.
	\begin{thm}\label{thm: qgundersc}
		Suppose Problem \eqref{opt: mainProblem}, $\min_{x\in \Omega} g(\Amap x)+\inprod{c}{x}$,
		satisfies that
		$g$ is strongly convex and the constraint set $\Omega$ is one of the four sets (i) polytope,
		(ii) unit group norm ball, (iii) spectrahedron, and (vi) unit nuclear norm ball.
		Further suppose that strict complementarity holds.
		Then quadratic growth holds for Problem \eqref{opt: mainProblem} as well.
	\end{thm}
	We will use the machinery developed in \cite{zhou2017unified} for the case of the group norm.
	We define a few notions and notations for later convenience.
	We define the projection to $\Omega$ as $\mathcal{P}_{\Omega}(x):\, = \arg\min_{v\in \Omega}\twonorm{x-v}$.
	The difference
	of iterates for projected gradient with step size $t$ is defined as $\mathcal{G}_t(x):\,=\frac{1}{t}(x-\proj_{\Omega}(x-\nabla f(x)))$.
	Note that
	$\mathcal{G}_t(x)=0$ implies $x=\xsol$. Finally,
	for an arbitrary set $\mathcal{S}$, we define the distance of $x\in \real^{\dm}$ to it as $\dist(x,\mathcal{S}):\,=\inf_{v\in \mathcal{S}}\twonorm{x-v}$.

	\begin{proof}
		The proof for the polytope appears in \cite[Lemma 2.5]{beck2017linearly}.
		The proof of the Spectrahedron appears in
		\cite[Theorem 6]{DingFeiXuYang2020}.
		Here, we address the case of the group norm ball and Nuclear norm ball.
		Let us first consider the case of the group norm ball with the $\ell_2$ norm.

		\myparagraph{Unit group norm ball} Using \cite[Corollary 3.6]{drusvyatskiy2018error}, we know that if the
		error bound condition holds for some $t,\gamma>0$ then the quadratic growth condition holds with some parameter $\gamma'$.
		The error bound condition with parameter $t,\gamma,\epsilon >0$ means that for all $x\in \Omega$ and $\twonorm{x-\xsol}\leq \epsilon$,
		the following the inequality holds:
		\footnote{The error bound condition considered in \cite[Corollary 3.6]{drusvyatskiy2018error}
			actually require the bound \eqref{eq: erb} to hold for all $x$ in
			the intersection of $\Omega$ and a sublevel set of $f$.
			Note there is
			a difference  between a sublevel set and a neighborhood of $\xsol$.
			However, as $f$ is continuous and $\Omega$ is compact,
			when restricted to $\Omega$ any neighborhood of $\xsol$ is contained in a sublevel set and vice versa.
      % we know the two are actually equivalent (with a different choice of $\gamma$).
			Moreover, the quadratic
			growth condition of \cite[Corollary 3.6]{drusvyatskiy2018error} is only required to hold for $x$ in $\Omega$ and a sublevel set of $f$.
			Again, this condition is equivalent to ours as to $\Omega$ is compact and $f$ is continuous.}
		\begin{equation}\label{eq: erb}
		\twonorm{x-\xsol}\leq  \gamma \twonorm{\mathcal{G}_t(x)}.
		\end{equation}
		Define $\bar{y} = \Amap(\xsol)$ and $\bar{h} = \nabla f(\xsol)$.
		Now using \cite[Corollary 1 and Theorem 2]{zhou2017unified},
		we need only verify the following two conditions to establish \eqref{eq: erb}:
		\begin{enumerate}
			\item \emph{Bounded linear regularity:} The two sets $\Gamma_f(\bar{y}) :\,=\{x\in \EE \mid \bar{y} =\Amap(x)\}$ and $\Gamma_\Omega(\bar{h}):\,= \{x\in \EE\mid -\bar{h}\in N_{\Omega}(x) \}$
			satisfy that for every bounded set $B$, there exists a constant $\kappa$ such that
			\[
			\dist(x,\Gamma_f(\bar{y})\cap \Gamma_\Omega (\bar{h})) \leq \kappa \max\{\dist(x,\Gamma_f(\bar{y})), \dist(x,\Gamma_\Omega(\bar{h}))\}, \;\text{for all}\; x\in \Omega.
			\]
			\item \emph{Metric subregularity:}
			there exists $\kappa,\epsilon>0$ such that for all $x$ with $\twonorm{x-\xsol}\leq \epsilon$,
			\begin{equation}\label{eq: metricsubregularity}
			\dist(x,\Gamma_\Omega(\bar{h}))\leq \kappa \dist(-\bar{h},N_{\Omega}(x)).
			\end{equation}
		\end{enumerate}
		Let us first verify bounded linear regularity. First, the subdifferential of
		the Euclidean norm $\twonorm{\cdot}$ is
		\[
		\partial \twonorm{x} = \begin{cases}
		\frac{x}{\twonorm{x}} & x\not= 0,\\
		\mathbf{B}_{\twonorm{\cdot}} & x=0.\\
		\end{cases}.
		\]
		Here $\mathbf{B}_{\twonorm{\cdot}}:\,=\{x \mid \twonorm{x}\leq 1\}$ is unit $\ell_2$ norm ball.

		From the characterization \eqref{eq: normaconerelativeinterior} of the interior of the normal cone, we know that $\bar{h}=\nabla f(\xsol)$
		is nonzero due to strict complementarity, and hence any $x\in \Gamma(\bar{h})$ must satisfy $x\in \partial \Omega$.
		Following the derivation of the normal cone in \eqref{eq: unitgroupnormballnormalcone}, we have
		for any $x\in \partial \Omega$,
		\begin{equation}\label{eq: unitgroupnormballnormalconearbitraryx}
		N_{\Omega}(x)=\big\{y\mid y\in \lambda \left(\prod_{g\in \mathcal{F}(x)}\partial\twonorm{x_g} \times \prod_{g\in \mathcal{G}-\mathcal{F}(x)}\mathbf{B}_{\twonorm{\cdot}}\right),~ \lambda\geq 0 \big\}.
		\end{equation}
		Here the support set $\mathcal{F}(x)$ is the set of groups in the support of $x$.
		Let us pick a $i^\star \in \mathcal{F}(\xsol)$. For each $i\in \mathcal{G}$,
		define the vector $\tilde{h}_i = \frac{-\bar{h}_i}{\twonorm{h_{i^\star}}}\in \real^{|v_i|}$.
		Recall from \eqref{eq: nablagradientgroupdualnorm}, we have
		$\twonorm{\bar{h}_i}$ all equal for $i\in \mathcal{F}(\xsol)$.
		For each $i\in \mathcal{F}(\xsol)$, define $\tilde{h}^i\in \real^{\dm}$ so that it is only supported
		on group $i$ with vector value $\tilde{h}_i$ and is $0$ elsewhere.
		Again, from \eqref{eq: nablagradientgroupdualnorm} and Lemma \ref{lem: positiveDelta},
		we have $\twonorm{\bar{h}^i}$ all equal for $i\in \mathcal{F}(\xsol)$, and is larger than
		those $i$ not in $\mathcal{F}(\xsol)$. To remember the notation, we use $\tilde{h}^i$, upper index $i$, to mean a
		vector in $\real^{\dm}$. We use the notation $\tilde{h}_i$, lower index $i$, to mean the shorter vector in $\real^{|v_i|}$.

		Combining the facts about $\tilde{h}_i$, the formula
		\eqref{eq: unitgroupnormballnormalconearbitraryx}, the
		formula of $\partial \twonorm{\cdot}$, and $x\in \partial \Omega$, we find that actually
		\[
		\Gamma_\Omega(\bar{h})=\{x\mid \sum_{i\in \mathcal{F}(\xsol)} \alpha_i \tilde{h}^i,\alpha_i \in \Delta^{|\mathcal{F}(\xsol)|}\},
		\] which
		is a convex polyhedral. Because $\Gamma_f(\bar{y})$ and $\Gamma_\Omega(\bar{h})$ are both convex polyhedral, we know from
		\cite[Corollary 3]{bauschke1999strong} that bounded linear regularity holds.

		We verify metrical subregularity now. Note that from previous calculation of $\Gamma_\Omega(\bar{h})$, we know
		\[
		\dist(x,\Gamma_\Omega(\bar{h})) ^2= \min_{\alpha_i\in\Delta^{|\mathcal{F}(\xsol)|}}\sum_{i\in \mathcal{F}(\xsol)} \twonorm{x_i - \alpha _i \tilde{h}_i}^2 + \sum_{i\not\in \mathcal{F}(\xsol)}\twonorm{x_i}^2.
		\]
		By choosing
		$\epsilon$ sufficiently small, say $\epsilon < \epsilon_0$,
		we have $\mathcal{F}(x)\supseteq \mathcal{F}(\xsol)$.
		The quantity,  $\dist(\bar{h},N_{\Omega}(x))$, on the RHS of \eqref{eq: metricsubregularity}
		for all $x$ within an $\epsilon$ neighborhood of the solution $\xsol$
		satisfies that
		\[
		\dist^2(\bar{h},N_{\Omega}(x))=\begin{cases}
		+\infty ,& x\not\in \Omega, \\
		\twonorm{\bar{h}}^2, & x\in \mbox{int}(\Omega),
		\end{cases}
		\]
		where $\mbox{int}(\Omega)$ is the interior of $\Omega$. For $x\in \partial \Omega$, $\dist^2(\bar{h},N_{\Omega}(x))$ satisfies that
		\[
		\dist^2(\bar{h},N_{\Omega}(x)) = \twonorm{h_{i^\star}}^2
		\min_{\lambda \geq 0,v_i\in \mathbf{B}_{\twonorm{\cdot}}} \sum_{i\in \mathcal{F}(x)}\twonorm{\tilde{h}_i-  \lambda\tilde{x}_i}^2
		+\sum_{i\notin \mathcal{F}(x)}\twonorm{\tilde{h}_i -\lambda v_i}^2,
		\]
    where $\tilde{x} = \frac{x}{\twonorm{x}}$, and $\tilde{x}_i$ is the vector with components in group $i$.
		The case of $x\not\in \Omega$ is trivial.
		The case of $x\in \mbox{int}(\Omega)$ can be proved by
		choosing a large enough $\kappa$, say $\kappa > K_0$,
		as $\dist(x,\Gamma_\Omega(\bar{h})) ^2$ is upper bounded for any $x\in \mbox{int}(\Omega)$,
		and $\dist^2(\bar{h},N_{\Omega}(x))$ in this case is fixed.
		We are left with the most challenging case $x\in \partial \Omega$,
		where the normal cone is non-trivial.
		First, we upper bound $\dist(x,\Gamma_\Omega(\bar{h})) ^2$ by choosing $\alpha_i = \twonorm{x_i}$.
		The numbers $\alpha_i$ sum to one because $x\in \partial \Omega$.
		In this case,
		$\dist(x,\Gamma_\Omega(\bar{h})) ^2$ satisfies the bound
		\begin{equation}
		\begin{aligned} \label{eq: distxGammaOmegabarg}
		\dist(x,\Gamma_\Omega(\bar{h})) ^2 &\leq \sum_{i\in \mathcal{F}(\xsol)}\twonorm{x_i -\twonorm{x_i}\tilde{h}_i}+ \sum_{i\not\in \mathcal{F}(\xsol)}\twonorm{x_i}^2 \\
		& \overset{(a)}{=}  \sum_{i\in \mathcal{F}(\xsol)}\twonorm{x_i}\twonorm{\tilde{h}_i-\tilde{x}_i}^2+ \sum_{i\in \left(\mathcal{F}(x)-\mathcal{F}(\xsol)\right)}\twonorm{x_i}^2,
		\end{aligned}
		\end{equation}
		where step $(a)$ is due to  $\mathcal{F}(x)\supseteq \mathcal{F}(\xsol)$ by our choice of small enough $\epsilon$.
		We next lower bound $\dist^2(\bar{h},N_{\Omega}(x))$ by ignoring the term not in $\mathcal{F}(x)$:
		\[
		\dist^2(\bar{h},N_{\Omega}(x))\geq \twonorm{g_{i^\star}}^2
		\min_{\lambda \geq 0} \sum_{i\in \mathcal{F}(x)}\twonorm{\tilde{h}_i- \lambda\tilde{x}_i}^2
		\]
		Now if $\mathcal{F}(x)=\mathcal{F}(\xsol)$, then it is tempting to set $\lambda=1$ above and compare the
		inequality with \eqref{eq: distxGammaOmegabarg}
		to claim victory. This does not work directly due to the minimization over $\lambda$ and the fact  $\mathcal{F}(x)\supseteq \mathcal{F}(\xsol)$.

		Let $\lambda_\star = \argmin _{\lambda \geq 0} \sum_{i\in \mathcal{F}(x)}\twonorm{\tilde{h}_i-  \frac{\lambda x_i}{\twonorm{x_i}}}^2$.
		In this case, we have an explicit formula of $\lambda_\star$:
		\[
		\lambda_\star = \max\left \{0,\frac{\sum_{i\in \mathcal{F}(x)}\inprod{\tilde{h}_i}{\tilde{x}_i}}{|\mathcal{F}(x)|}\right \}.
		\]
		If $\lambda_\star =0$, then we can simply pick some $\kappa>K_0$ as done in the case of $x\in \mbox{int}(\Omega)$. So we assume $\lambda_\star >0$ in the following.
		Next let $\lambda_{i} = \argmin_{\lambda\geq 0}\twonorm{\tilde{h}_i-  \frac{\lambda x_{i}}{\twonorm{x_{i}}}}^2$ for each
		$i\in \mathcal{F}(\xsol)$. With such choice of $\lambda_i$ and $\lambda^\star$, we can further lower bound
		$\dist^2(\bar{h},N_{\Omega}(x))$ by splitting the terms in $\mathcal{F}(x)$ and those are not:
		\begin{equation}
		\begin{aligned}\label{eq: distxNormalConeOmegabarg}
		\dist^2(\bar{h},N_{\Omega}(x))\geq \twonorm{g_{i^\star}}^2\left(
		\underbrace{ \sum_{i\in \mathcal{F}(\xsol)} \twonorm{\tilde{h}_i-  \frac{\lambda_{i} x_i}{\twonorm{x_i}}}^2}_{R_1}
		+\underbrace{\sum_{i \in \mathcal{F}(x) \setminus \mathcal{F}(\xsol)}\twonorm{\tilde{h}_i-\frac{\lambda_{\star}x_i}{\twonorm{x_i}}}^2}_{R_2}\right).
		\end{aligned}
		\end{equation}
		We bound the two terms separately. Let us
		first deal with $R_1$. From the expression of normal cone \eqref{eq: unitgroupnormballnormalconearbitraryx}
		and $-\bar{h}\in N_{\Omega}(\xsol)$ by our assumption,
		we know $\tilde{h}_i = \frac{(\xsol)_i}{\twonorm{(\xsol)_i}}$ for every $i\in \mathcal{F}(\xsol)$.
		Hence by choosing a (possibly smaller) $\epsilon$, say $\epsilon < \epsilon_1$,
		we can ensure that for any $x$ within an $\epsilon_1$ neighborhood of the solution $\xsol$,
		$\inprod{\tilde{x}_i}{\tilde{h}_i}\geq 0$ all for $i\in \mathcal{F}(\xsol)$.
		Moreover, for a small enough $\epsilon_1$, we know each $\lambda_i =\inprod{\tilde{x}_i}{\tilde{h}_i}$
		and is very close to $1$. Thus the condition of
		Lemma \ref{lem: additionalLemmaGroupNormBall} is fulfilled, and we have
		\begin{equation}\label{eq: R_1groupnormball}
		R_1\geq \frac{1}{2} \sum_{i\in \mathcal{F}(\xsol)}\twonorm{x_i} \twonorm{\tilde{h}_i-\tilde{x}_i}^2.
		\end{equation}

		Next, to deal with $R_2$, let us examine the expression of
		$
		\lambda_\star= \frac{\sum_{i\in \mathcal{F}(x)}\inprod{\tilde{h}_i}{\tilde{x}_i}}{|\mathcal{F}(x)|}.
		$
		Recall $\inprod{\tilde{h}_i}{\tilde{x}_i}$ is close to $1$ for small enough $\epsilon$.
		Due to strict complementarity, for each $i\in\mathcal{F}(x)\setminus\mathcal{F}(\xsol)$,
		we know $\twonorm{\tilde{h}_i}<1-\delta_0$ for some $\delta_0>0$ that depends only on $\bar{h}$.
		Combining these two facts, we know that  $i' = \arg\min_{i\in \mathcal{F}(x)} \inprod{\tilde{h}_i}{\tilde{x}_i}$ must
		belong to $\mathcal{F}(x)\setminus\mathcal{F}(\xsol)$.
		Moreover, by choosing an even smaller $\epsilon$, say $\epsilon<\epsilon_2$,
		we have $\lambda_\star \geq \delta_1 + \min_i \inprod{\tilde{h}_i}{\tilde{x}_i}$ for some $\delta_1>0$ that
		only depends on $\bar{h}$, $\delta$, and $\epsilon_2$. We can now lower bound $R_2$ as follows:
		\begin{equation}\label{eq: R_2groupnormball}
		\begin{aligned}
		R_2\geq & \twonorm{\tilde{h}_{i'}-(\delta_1+ \inprod{\tilde{h}_i}{\tilde{x}_i})\tilde{x}_{i'}}^2 \\
		       =& \twonorm{\tilde{h}_i - \inprod{\tilde{h}_i}{\tilde{x}_i}\tilde{x}_{i'}}^2+\delta_1^2\\
		        &+ 2\delta\underbrace{\inprod{\tilde{h}_i - \inprod{\tilde{h}_i}{\tilde{x}_i}\tilde{x}_{i'}}{\tilde{x}_{i'}}}_{=0}\\
		\geq    & \delta_1^2.
		\end{aligned}
		\end{equation}
		Combining the bounds \eqref{eq: R_1groupnormball} and \eqref{eq: R_2groupnormball} on $R_1$ and $R_2$, we find that
		\begin{equation}
		\begin{aligned}\label{eq: distxNormalConeOmegabargFinal}
		\dist^2(\bar{h},N_{\Omega}(x))
		&\geq \frac{\twonorm{g_{i^\star}}^2}{2}
		\sum_{i\in \mathcal{F}(\xsol)} \twonorm{\tilde{h}_i-\tilde{x}_i}^2 + \twonorm{g_{i^\star}}^2 \delta_1^2\\
		&\overset{(a)}{\geq}  \frac{\twonorm{g_{i^\star}}^2}{2}
		\sum_{i\in \mathcal{F}(\xsol)} \twonorm{x_i} \twonorm{\tilde{h}_i-\tilde{x}_i}^2 + \twonorm{g_{i^\star}}^2 \delta_1^2\sum_{i \in \mathcal{F}(x) \setminus \mathcal{F}(\xsol)}\twonorm{x_i}^2\\
		\end{aligned}
		\end{equation}
		Here, for the step $(a)$, we use $\twonorm{x_i}\leq 1$ as $x\in \partial \Omega$.
		Hence, by taking $\epsilon=\min(\epsilon_1,\epsilon_2)$ and
		$\kappa = \max\{K_0,\twonorm{g_{i^\star}}^2\delta_1,\frac{\twonorm{g_{i^\star}}^2}{2}\}$, and comparing \eqref{eq: distxNormalConeOmegabargFinal}
		with \eqref{eq: distxGammaOmegabarg}, a bound on $\dist(x,\Gamma_\Omega(\bar{h})) ^2$,
		we see that metric subregularity is satisfied and our proof for unit group norm ball is complete.

		Finally, we consider the unit nuclear norm ball.

		\myparagraph{Unit nuclear norm ball} Let us first illustrate the main idea.
		We shall utilize the quadratic growth result proved in \cite[Theorem 6]{DingFeiXuYang2020} for
		spectrahedron. To transfer our setting to spectrahedron,
		we use a dilation argument with its relating lemmas \cite[Lemma 3]{ding2020regularity}
		and \cite[lemma 1]{jaggi2010simple}. We now spell out all the details.

		Let $\tilde{n} = \dm_1+\dm_2$. For any $\tilde{X}\in \symMat^{\tilde{\dm}}$, denote its eigenvalues as
		$\lambda_1(\tilde{X})\geq \dots  \geq \lambda_{\tilde{\dm}}(\tilde{X})$. Also, for any $X\in \mathbf{B}_{\nucnorm{\cdot}}$,
		denote its singular value decomposition as
		$X = U_X\Sigma_X V_X^\top$ where $U_X\in \real^{\dm_1\times r_X}$,
    $V_X\in \real^{\dm_2\times r_X}$, and
		$r_X = \rank(X)$.
    Define the dilation $X^\sharp\in \symMat^{\tilde{n}}$ of a $X\in \real^{\dm_1\times \dm_2}$
		as
		\begin{equation}\label{lem: dilation}
		X^\sharp =\frac{1}{2} \begin{bmatrix}
		X_1 & X \\
		X^\top & X_2
		\end{bmatrix},
		\end{equation}
		where the $X_1 = U_X(\Sigma_X + \xi_X I)U_X^\top$, and $X_2 =V_X (\Sigma_X +\xi_XI) V_X^\top$. The number $\xi_X\geq 0$
		is chosen so that $X^\sharp$ has trace $1$. Note that $X^\sharp$ is positive semidefinite as  $X^\sharp= \frac{1}{2}\begin{bmatrix}
		U_X \\V_X
		\end{bmatrix}\Sigma _X [U_X^\top V_X^\top]
		+
		\frac{\xi_X}{2}
		\begin{bmatrix}
		U_XU_X^\top & 0 \\
		0 & V_XV_X^\top
		\end{bmatrix}.$
		For any $\tilde{Y} = \frac{1}{2}\begin{bmatrix}
		Y_1 & Y \\
		Y^\top & Y_2
		\end{bmatrix} \in \symMat^{\tilde{\dm}}$ with $Y_1\in \symMat^{\dm_1}$,
		and $Y_2\in \symMat^{\dm_2}$, we denote its off diagonal component as
		\[
		\tilde{Y}_\flat :\,= Y.
		\]
		Note here that $X^\sharp$ denotes the dilation of a matrix $X\in \real^{\dm_1\times\dm_2}$,
		while $\tilde{X}$ means a generic matrix in $\symMat^{\tilde{\dm}}$ which is not necessarily related to
		$X$. %, and is not $X^\sharp$ in general.
    We also have the relation that $(X^\sharp)_\flat = X$ for any  $X\in \real^{\dm_1\times\dm_2}$.

		Consider the problem
		\begin{equation}\label{opt: NucSpectraProblem}
		\begin{array}{ll}
		\mbox{minimize} & \tilde{f}(\tilde{X}):\,= f(\tilde{X}_\flat) =g(\Amap (\tilde{X}_\flat)) +\inprod{C}{(\tilde{X})_\flat} \\
		\mbox{subject to} & \tilde{X} =1\quad \tilde{X}\succeq 0.
		\end{array}
		\end{equation}
		We claim that it satisfies strict complementarity and its solution $\tilde{\Xsol}$
		is unique and is equal to $\Xsol^\sharp$.
    Suppose the claim is proved for the moment. Note that $\Xsol\in \partial\Omega$
		implies that $\rank(\Xsol^\sharp)=\rank(\Xsol)<\tilde{\dm}$. Hence,
		the condition of \cite[Theorem 6]{DingFeiXuYang2020} is fulfilled, and we know there is some
		$\tilde{\gamma}>0$, such that for all $\tilde{X}\in \mathcal{SP}^{\tilde{n}}$, we have
		\[
		\tilde{f}(\tilde{X}) -\tilde{f}(\Xsol^\sharp) \geq\tilde{\gamma} \fronorm{\tilde{X}-\Xsol^\sharp}.
		\]
		Hence, for any $X\in \Omega$, by construction of $\tilde{f}$, we have
		\[
		f(X)-f(\Xsol) =\tilde{f}(X^\sharp)-\tilde{f}(\Xsol^\sharp) \geq \tilde{\gamma} \fronorm{X^\sharp-\Xsol^\sharp}^2\geq \frac{\tilde{\gamma}}{2}\fronorm{X-\Xsol}^2.
		\]
		This proves quadratic growth.

		We now verify our claim that $\Xsol^\sharp$ is the unique solution to \eqref{opt: NucSpectraProblem} and
		$\Xsol^\sharp\in \partial\mathcal{SP}^{\tilde{\dm}}$ with $\nabla \tilde{f}(\Xsol^\sharp)\in \partial N_{\mathcal{SP}^{\tilde{\dm}}}(\Xsol^\sharp)$.
		First, consider feasibility and whether $\Xsol^\sharp \in \partial \mathcal{SP}^{\tilde{\dm}}$.
		The condition $\Xsol\in \partial \Omega$ implies that $1\leq \rsol<\min(\dm_1,\dm_2)$ and $\nucnorm{\Xsol}=1$.
		Hence we do have
		$\tr(X^\sharp)=1$ and $\Xsol^\sharp\in  \partial\mathcal{SP}^{\tilde{\dm}}$ as $\rank(\Xsol^\sharp) =\rank(\Xsol)= \rsol <\dm_1+\dm_2$. Next, consider
		optimality. Given any $\tilde{X}\in \mathcal{SP}^{\tilde{n}}$, we may write it as
		 $\tilde{X} = \frac{1}{2} \begin{bmatrix}
		X_1 & X \\
		X^\top & X_2
		\end{bmatrix}$. By \cite[Lemma 1]{jaggi2010simple}, we have
		\begin{equation}
		\nucnorm{X} \leq 1.
		\end{equation}
		To see $\Xsol^\sharp$ is optimal for \eqref{opt: NucSpectraProblem}, note that
		\[
		f((\tilde{X})_\flat) = f(X)\overset{(a)}{\geq}f(\Xsol)= f((\Xsol^\sharp)_\flat),
		\]
		where step $(a)$ is due to optimality of $\Xsol$ in \eqref{opt: mainProblem} and $X$ is feasible as just argued.
		Thirdly, we argue that $\Xsol^\sharp$ is a unique solution to \eqref{opt: NucSpectraProblem}.
		For any optimal solution $\tilde{\Xsol}=  \frac{1}{2} \begin{bmatrix}
		X_1^\star & X_0 \\
		X^\top_0 & X_2^\star
		\end{bmatrix}$ of \eqref{opt: NucSpectraProblem}, we have $X_0$ is optimal to \eqref{opt: mainProblem} as
		\[
		f(X_0) = f((\tilde{X}_\star)_\flat) \overset{(a)}{=} f((\Xsol^\sharp)_\flat)=f(\Xsol),
		\]
		where step $(a)$ is because $\Xsol^\sharp$ is optimal to  \eqref{opt: NucSpectraProblem}.
		Hence due to uniqueness of $\Xsol$, we know $X_0=\Xsol$. Because
		$\nucnorm{\Xsol}=1$, using \cite[Lemma 3]{ding2020regularity}, we know in fact $\Xsol^\sharp = \tilde{\Xsol}$ and uniqueness of solution to
		\eqref{opt: NucSpectraProblem} is proved.
		Finally, we verify strict complementarity that $\nabla \tilde{f}(\Xsol^\sharp)\in \mbox{relint}\left( N_{\mathcal{SP}^{\tilde{\dm}}}(\Xsol^\sharp)\right)$.
		Recall from \eqref{eq: sumRuleRelativeInteriorSpec}, that we need to show
		\begin{equation*}\label{eq: sumRuleRelativeInteriorSpecNuc}
		-\nabla \tilde{f}(\Xsol^\sharp) \in \mbox{relint}(N_{\mathcal{SP}^{\tilde{\dm}}}(\tilde{\Xsol}) )= \{sI\mid s\in \real\} + \{-\tilde{Z}\mid \tilde{Z}\succeq 0,\range(\tilde{Z})= \nullspace(\Xsol^\sharp)\}.
		\end{equation*}
		Using the definition of $\Xsol^\sharp$, we know
		\[
		\nabla \tilde{f}(\Xsol^\sharp)  = \begin{bmatrix}
		0 & \nabla f(\Xsol) \\
		\nabla f(\Xsol)^\top & 0
		\end{bmatrix}.
		\]
		Recall from Lemma \ref{lem: positiveDelta}, we have
		$\sigma_{1}(\nabla f(\Xsol))=\dots =\sigma_{\rsol}(\nabla f(\Xsol))=\delta +\sigma_{\rsol+1}(\nabla f(\Xsol))$ for some gap $\delta>0$.
		Hence we see that $\nabla \tilde{f}(\Xsol^\sharp)$ has all its smallest $\rsol$ eigenvalues equal as $-\sigma_{\rsol}(\nabla f(\Xsol))$
		and the
		gap between its $\rsol$-th smallest eigenvalue and the $\rsol+1$-th eigenvalue is simply $\delta>0$. Moreover, let the singular value decomposition of
		$\Xsol$ as $\Xsol = U_\star \Sigma V_\star^\top $ with $U_\star \in \real^{\dm_1\times \rsol}$ and $V_\star \in \real^{\dm_2\times \rsol}$.
		From the description of normal cone of nuclear norm ball in \eqref{eq: DescriptionOfRelativeInteriorOfNuc}, we know
		$U_\star ,-V_\star$ are the matrices formed by the top $\rsol$ left and right vectors of $\nabla f(\Xsol)$.
		Hence, the bottom $\rsol$ eigenvector of $\nabla f(\tilde{\Xsol})$ is simply $\frac{1}{\sqrt{2}}\begin{bmatrix}
		U_\star \\ V_\star.
		\end{bmatrix} $. Since $\range(\Xsol^\sharp)= \range(\begin{bmatrix}
		U_\star \\ V_\star.
		\end{bmatrix})$, we may take $s = \sigma_1(\nabla f(\Xsol))$ and $\tilde{Z} =\sigma_1(\nabla f(\Xsol)) I + \nabla\tilde{f}(\Xsol^\sharp)$.
		Using the eigengap condition on $\nabla \tilde{f}(\Xsol^\sharp)$, we see $\range(\tilde{Z}) = \nullspace(\Xsol^\sharp)$ and our claim is proved.
	\end{proof}

	\subsubsection{Additional Lemma for quadratic growth}
	We establish the following lemma for the proof of unit group norm ball.
	\begin{lem}\label{lem: additionalLemmaGroupNormBall}
		For any two $x,y\in \real^d$ with $\ell_2$ norm one, and $a:=\inprod{x}{y}\geq 0$,
		we have
		\[
		2\min_{\lambda\geq 0} \twonorm{x-\lambda y}^2 \geq \twonorm{x-y}^2.
		\]
	\end{lem}
	\begin{proof}
		Simple calculus reveals that the optimal solution $\lambda^\star$ of
		the LHS of the inequality is $\lambda^\star = a\geq 0$.
		We know $a\in[0,1]$ due
		to Cauchy-Schwarz and our assumption on $a$.
		Direct calculation of the difference yields
		\begin{equation}
		\begin{aligned}
		2\min_{\lambda\geq 0} \twonorm{x-\lambda y}^2 -\twonorm{x-y}^2 & =2+2a^2 -4a^2-2+2a \\
		& = -2a^2+2a \geq 0,
		\end{aligned}
		\end{equation}
		where the last line is due to $a\in[0,1]$.
	\end{proof}

	\subsection{Proofs of Theorem \ref{thm: nonexponentialFiniteConvergence} for group norm ball}\label{sec: groupnormballproofconvergence}
	\begin{proof}
		Let us now consider Algorithm \ref{alg k-FW, polytope} whose constraint set $\Omega$ is a
		unit group norm ball with \emph{arbitrary} base norm $\norm{\cdot}$.
		Using quadratic growth $(a)$, Theorem \ref{thm: normalFWresult} in the second step $(b)$,
		and the choice of $T$ in the following step $(c)$, the iterate $x_t$ with $t\geq T$ satisfies that
		\begin{equation}\label{eq: xtclosetooptimalxgrounormball}
		\norm{x_t - \xsol} \overset{(a)}{\leq} \sqrt{\frac{1}{\gamma} h_t}
		\overset{(b)}{\leq} \sqrt{\frac{L_f D^2}{\gamma T}} \overset{(c)}{\leq} \frac{\delta}{2L_fD}.
		\end{equation}
		Next recall the definition of $\mathcal{F}(\xsol)$ implies $(\xsol)_{g_\star}\not=0$ for any $g_\star\in \mathcal{F}(\xsol)$.
		The optimality conditions and
		$\norm{\xsol}_{\mathcal{G}}=1$ (due to $\xsol\in \partial \Omega$)
		implies that for every $g_\star \in \mathcal{F}(\xsol)$,
		\[
		\inprod{[\nabla f(\xsol)]_{g_\star}}{[\xsol]_{g_\star}} = \norm{[\nabla f(\xsol)]_{g_\star}}_*\norm{[\xsol]_{g_\star}}, \quad
		\text{and} \quad
		\inprod{\nabla f(\xsol)}{\xsol} = \norm{[\nabla f(\xsol)]_{g_\star}}_*.
		\]
		For any $g_\star \in \mathcal{F}(\xsol)$, define a vector $\xsol^{g_\star}\in \Omega$ as
		$\xsol^{g_\star}:=
		\begin{cases}
		\left(\frac{[\xsol]_{g_\star}}{\norm{[\xsol]_{g_{\rsol}}}}\right)_i ,& i \in g_\star, \\
		0                     , & i\not\in g_\star.
		\end{cases}
		$ So $ \xsol^{g_\star}\in\Omega$ is an extended vector of the normalized
		vector $\frac{[\xsol]_{g_\star}}{\norm{[\xsol]_{g_{\rsol}}}}$. Combining this definition
		with
		previous two equalities, we see
		\begin{equation}\label{eq: groupnormfractionnormalizeddualnormnablafxsol}
		\inprod{[\nabla f(\xsol)]_{g_\star}}{\xsol^{g_\star}}  = \inprod{\nabla f(\xsol)}{\xsol}.
		\end{equation}

		Now, for any $t\geq T$, we have for any group $g_\star$ in $\mathcal{F}(\xsol)$,
		and any vector $v \in \Omega$ that
		is in $\mathcal{F}^c(\xsol)$,
		\begin{equation}\label{eq: correctExtremePointgroupnormball}
		\begin{aligned}
		\inprod{\nabla f(x_t)}{\xsol^{g_\star}}-\inprod{\nabla f(x_t)}{v} & = \inprod{\nabla f(\xsol)}{\xsol^{g_\star}-v} + \inprod{\nabla f(x_t) - \nabla f(\xsol)}{v-u}\\
		&\overset{(a)}{\leq}   -\delta  + \inprod{\nabla f(x_t) - \nabla f(\xsol)}{v-u} \overset{(b)}{\leq} -\frac{\delta}{2}.
		\end{aligned}
		\end{equation}
		Here in step $(a)$, we use the definition of $\delta$ in \eqref{eq: polytopeDeltaDef} and \eqref{eq: groupnormfractionnormalizeddualnormnablafxsol}.
		In step $(b)$, we use the bound in \eqref{eq: xtclosetooptimalx}
		, Lipschitz continuity of $\nabla f(x)$, and $\norm{v-u}\leq D$.

		Thus, the $k$LOO step will produce all the groups in $\mathcal{F}(\xsol)$ as $k\geq \rsol$ after $t\geq T$, and so $\xsol$ is a feasible and optimal solution
		of the optimization problem in the $k$ direction search step.
		Hence Algorithm \ref{alg k-FW, polytope} finds the optimal solution $\xsol$ within $T+1$ steps.
	\end{proof}

	\subsection{Proofs of Theorem \ref{thm: LinearConvergence}}\label{sec: spectrahedronNuclearNormproofconvergence}
	We state one lemma that is critical to our proof of linear convergence. It is proved in Section \ref{sec: nucnormLemma}.
	\begin{lem}\label{lem: nucnormLemma}
		Given $Y\in \real^{\dm_1\times \dm_2}$ with $\sigma_{r}(Y)-\sigma_{r+1}(Y)=\delta>0$. Denote the matrices formed by the top $r$ left and right singular vectors of $Y$
		as $U\in \real^{\dm_1 \times r}$, $V\in \real^{\dm_2\times r}$ respectively.
		Then for any $X\in \real^{\dm_1\times \dm_2}$ with $\nucnorm{X}= 1$,
		there is an $S\in \real^{r\times r}$ with $\nucnorm{S}=1$ such that
		\[
		\inprod{X-USV^\top}{Y}\geq \frac{\delta}{2}\fronorm{X-USV^\top}^2.
		\]
	\end{lem}
	Equipped with this lemma, let us now prove  Theorem \ref{thm: LinearConvergence}.
	\begin{proof}[Proof of Theorem \ref{thm: LinearConvergence}]
		The case of the spectrahedron
		is proved in \cite[Theorem 3]{{DingFeiXuYang2020}}
		by using the eigengap
		formula in Lemma \ref{lem: positiveDelta} and
		\cite[Section 2.2 ``relation with the eigengap assumption'']{DingFeiXuYang2020}.
		Here, we need only to
		address the case of the unit nuclear norm.
		The proof that we present here for the case of the nuclear norm ball is quite similar.
		For notation convenience, for each $t$, let  $U_t,V_t$ be matrices
			formed by top $\rsol$ left and right singular vectors of $\nabla f(X_t)$.
		Define the set
		$
		\mathcal{N}_{\rsol,t} = \{U_tSV_t^\top \mid\nucnorm{S}\leq 1\}.
		$

		First note that Lemma \ref{lem: positiveDelta} shows that
		$\delta>0$.  Next, using the Lipschitz
		smoothness of $f$, we have for any $t\geq 1$,
		$\eta\in[0,1]$, and any $W\in \mathcal{N}_{\rsol,t}$:
		\begin{equation} \label{eq: QuadraticApproximationDueTosmoothness}
		\begin{aligned}
		f(X_{t+1})  \leq &f(X_{t})+(1-\eta)\inprod{W-X_{t}}{\nabla f(X_{t})}
		\\ &  +\frac{(1-\eta)^{2}L_f}{2}\fronorm{W-X_{t}}^{2}.
		\end{aligned}
		\end{equation}
		For $t\geq T$, we find that $\fronorm{X_t-\Xsol}\leq \frac{\delta}{6\sqrt{2}L_f}$, and
		\begin{equation} \label{eq:gapdifference}\
		\begin{aligned}
		&\sigma_{\rsol}(\nabla f(X_{t}))-\sigma_{\rsol+1}(\nabla f(X_t)) \\
		=& \underbrace{\sigma_{\rsol}\left(\nabla f(\Xsol)\right)-\sigma_{\rsol+1}(\nabla f(\Xsol))}_{\overset{(a)}{=}-\delta}
		+\underbrace{\left(\sigma_{\rsol}(\nabla f(X_{t}))-\sigma_{\rsol+1}\left(\nabla f(\Xsol)\right)\right)}_{\overset{(b)}{\leq}\frac{1}{3}\delta}\nonumber\\
		&+\underbrace{\left(\lambda_{n-\rsol+1}\left(\nabla f(\Xsol)\right)-\lambda_{n-\rsol+1}(\nabla f(X_{t}))\right)}_{\overset{(c)}{\leq}\frac{1}{3}\delta}\\
		\leq&-\frac{1}{3}\delta.\nonumber
		\end{aligned}
		\end{equation}
		Here in step $(a)$, we use the singular value gap formula of $\delta$ in Lemma \ref{lem: positiveDelta}. Step $(b)$ and $(c)$ are due
		to Weyl's inequality, the Lipschitz continutity of $\nabla f$, and the inequality $\fronorm{X_t-\Xsol}\leq \frac{\delta}{6\sqrt{2}L_f}$.

		Now we subtract the inequality \eqref{eq: QuadraticApproximationDueTosmoothness} both sides by $f(\Xsol)$, and denote
		$h_{t}=f(X_{t})-f(\Xsol)$ for each $t$ to arrive at
		\begin{equation}\label{eq:htht+1intermediatequantity}
		\begin{aligned}
		h_{t+1}\leq & h_{t}+(1-\eta)\underbrace{\inprod{W-X_{t}}{\nabla f(X_{t})}}_{R_1}\\
		& +\frac{(1-\eta)^{2}L_f}{2}\underbrace{\fronorm{W-X_{t}}^2}_{R_2}.
		\end{aligned}
		\end{equation}
		Using Lemma \ref{lem: nucnormLemma}, the inequality \eqref{eq:gapdifference}, and the
		assumption $\Xsol\in \partial \Omega$,
		we can choose $W\in\mathcal{N}_{\rsol,t}$ such that
		\begin{align}
		\inprod{W-\Xsol}{\nabla f(X_{t})} & \leq-\frac{\delta}{6}\fronorm{\Xsol-W}^{2}.\label{eq:intermediateinequality2forlinearconvergence}
		\end{align}
		Let us now analyze the term $R_1=\inprod{W-X_{t}}{\nabla f(X_{t})}$
		using (\ref{eq:intermediateinequality2forlinearconvergence}) and convexity of $f$:
		\begin{equation*}
		\begin{aligned} \label{eq:inprodw-xtalgorithmtheorem}
		R_1=   & \inprod{W-X_{t}}{\nabla f(X_{t})}\\
		= & \inprod{W-\Xsol}{\nabla f(X_{t})}+\inprod{\Xsol-X_{t}}{\nabla f(X_{t})}\nonumber \\
		\leq&-\frac{\delta}{6}\fronorm{\Xsol-W}^{2}-h_{t}.
		\end{aligned}
		\end{equation*}
		The term $R_2=\fronorm{X_{t}-W}^{2}$ can be bounded by
		\begin{equation*}
		\begin{aligned}
		R_2 = \fronorm{X_{t}-W}^{2} &\overset{(a)}{\leq}2\left(\fronorm{X_{t}-\Xsol}^{2}+\fronorm{\Xsol-W}^{2}\right)\\
		&\overset{(b)}{\leq}\frac{2}{\gamma}h_{t}+2\fronorm{\Xsol-W}^{2},\label{eq:xt-walgorithminequality}
		\end{aligned}
		\end{equation*}
		where we use the triangle inequality and the basic inequality $(a+b)^{2}\leq2a^{2}+2b^{2}$
		in step $(a)$, and the quadratic growth condition in step $(b)$.

		Now combining (\ref{eq:htht+1intermediatequantity}), and the bounds of $R_1$ and $R_2$, we reach that there is a
		$W\in\mathcal{N}_{\rsol,t}$ such that for any $\xi=1-\eta\in[0,1]$, we have
		\begin{align*}
		h_{t+1}
		\leq & h_{t}+\xi\left(-\frac{\delta}{6}\fronorm{\Xsol-W}^{2}-h_{t}\right)
		+\frac{\xi^{2}L_f}{2}\left(\frac{2}{\gamma}h_{t}+2\fronorm{\Xsol-W}^{2}\right)\\
		=&\left(1-\xi+\frac{\xi^{2}L_f}{\gamma}\right)h_{t}
		+\left(\xi^{2}L_f-\frac{\xi\delta}{6}\right)\fronorm{\Xsol-W}^{2}.
		\end{align*}
		A detailed calculation below and a careful choice of $\xi$ below yields the factor
		$1-\min \{\frac{\gamma}{4L_f},\frac{\delta}{12L_f}\}$ in the theorem.

        We show here how to choose $\xi\in[0,1]$ so that  $1-\xi+\frac{\xi^{2}L_f}{\gamma}$ is minimized while
		keeping $\xi^{2}L_f-\frac{\xi\delta}{6}\leq0$. For
		$\xi^{2}L_f-\frac{\xi\delta}{6}\leq0$, we need
		$\xi\le\frac{\delta}{6L_f}$. The function $q(\xi)=1-\xi+\frac{\xi^{2}L_f}{\gamma}$
		is decreasing for $\xi\leq\frac{\gamma}{2L_f}$ and increasing for
		$\xi\geq\frac{\gamma}{2L_f}$. If $\frac{\gamma}{2L_f}\leq\frac{\delta}{6L_f}$,
		then we can pick $\xi=\frac{\gamma}{2L_f}$, and $q(\xi)=1-\frac{\gamma}{4L_f}$.
		If $\frac{\gamma}{2L_f}\geq\frac{\delta}{6L_f}\implies\frac{\delta}{\gamma}\leq3$,
		then we can pick $\xi=\frac{\delta}{6L_f}$, and
		$q(\xi)=1-\frac{\delta}{6L_f}+\frac{\delta^2}{36\gamma L_f}=1+\frac{\delta}{6L_f}\left(\frac{\delta}{6\gamma}-1\right)\leq1-\frac{\delta}{12L_f}$.
	\end{proof}

	\subsubsection{Additional lemmas for the proof of  Theorem \ref{thm: LinearConvergence}}\label{sec: nucnormLemma}
	Here we give a proof of Lemma \ref{lem: nucnormLemma}.
	\begin{proof}[Proof of Lemma \ref{lem: nucnormLemma}]
		We utilize the result in \cite[Lemma 5]{DingFeiXuYang2020}:
		given any $\tilde{Y}\in \symMat^{\dm}$ with eigenvalues
		$\lambda_n(\tilde{Y})\leq \dots \leq \lambda_{n-r+1}(\tilde{Y})\leq \lambda_{n-r}(\tilde{Y})-\delta'\leq \dots \leq \lambda_1(\tilde{Y})-\delta'$ for some
		$\delta'>0$.
		Denote the matrices by the bottom $r$ eigenvectors of $\tilde{Y}$
		as $\tilde{V}\in \real^{\dm_2\times r}$ respectively.
		Then for any $\tilde{X}\in \symMat^{\dm}_+$ with $\tr(\tilde{X})=1$,
		there is an $S\in \real^{r\times r}_+$ with $\tr(S)=1$ such that
		\begin{equation}\label{eq: lemma5fromMyPreviousPaper}
		\inprod{\tilde{X}-\tilde{V}S\tilde{V}^\top}{\tilde{Y}}\geq \frac{\delta'}{2}\fronorm{\tilde{X}-\tilde{V}S\tilde{V}^\top}^2.
		\end{equation}
		To utilize this result, we consider the dilation of the matrices $X$ and $Y$:
		\begin{equation}
		\tilde{X} :\,= \frac{1}{2}\begin{bmatrix}
		X_1 & X \\
		X^\top & X_2
		\end{bmatrix},\quad \text{and} \quad
		\tilde{Y} :\,=\begin{bmatrix}
		0 & Y \\
		Y^\top & 0
		\end{bmatrix}.
		\end{equation}
		Here the matrices $X_1= U_X \Sigma_XU_X$,
		$X_2 = V_X\Sigma_X V_X^\top$ where
		$U_X\Sigma_XV_X$ is the SVD of $X$ and the number
		$r_X = \rank(X)$.
		Since $\tilde{X}= \begin{bmatrix}
		U_X \\V_X
		\end{bmatrix}\Sigma _X [U_X^\top\; V_X^\top]$, the matrix $\tilde{X}\in \symMat_+^{\dm_1+\dm_2}\succeq 0$.
		The trace of $\tilde{X}$ is $\tr(\tilde{X})=1$ as $\nucnorm{X}=1$.
		Note that the bottom $r+1$ eigenvalues of $\tilde{Y}$ is
		simply $-\sigma_{1}(Y),\dots,-\sigma_r(Y),-\sigma_{r+1}(Y)$, and the matrix
		$\tilde{V}\in \real^{(\dm_1+\dm_2)r}$ defined below is formed by
		the matrix eigenvectors corresponds the smallest $r$ eigenvalues:
		\begin{equation}
		\tilde{V} :\,=\frac{1}{\sqrt{2}} \begin{bmatrix}
		U \\
		-V
		\end{bmatrix}.
		\end{equation}
		Using  \cite[Lemma 5]{DingFeiXuYang2020}, we can find a matrix $S\in \symMat_+^r$ with $\tr(S)=1$ such that
		\eqref{eq: lemma5fromMyPreviousPaper} holds. Writing the equation in block form reveals that
		\begin{equation}
		\inprod{X-U(-S)V^\top}{Y}=\inprod{\tilde{X}-\tilde{V}S\tilde{V}^\top}{\tilde{Y}}\geq \frac{\delta}{2} \fronorm{\tilde{X}-\tilde{V}S\tilde{V}^\top}^2
		\overset{(a)}{\geq} \frac{\delta}{2}\fronorm{X-U(-S)V^\top}^2,
		\end{equation}
		where the last step is due to Lemma \ref{lem: svdfronormdifference}.
		Note that the matrix $U(-S)V^\top$ is the matrix we seek as $\nucnorm{-S}=\tr(S)=1$. Hence the proof is completed.
	\end{proof}
	\begin{lem}\label{lem: svdfronormdifference}
		Suppose two matrices $X,Y\in \real^{\dm_1\times \dm_2}$, and $X= U_1S_1V_1^\top$ and $Y = U_2S_2V_2^\top$ for some
		unitary $U_i,V_i$ that for some integers $r_1,r_2$, they satisfy $U_i\in \real^{\dm_1\times r_{i}}$, $i=1,\,2$ and $V_i \in \real^{\dm_2\times r_i}$, $i=1,\,2$.
		The matrices $S_i\in\symMat_+^{r_i}$ are positive semidefinite. Then
		\[
		\fronorm{U_1S_1U_1^\top -U_2S_2U_2^\top}^2 + \fronorm{V_1S_1V_1^\top-V_2S_2V_2^\top}^2 \geq
		2\fronorm{U_1S_1V_1^\top - U_2S_2V_2^\top}^2.
		\]
	\end{lem}
	\begin{proof}
		This result follows by direct computation. Consider the difference
		$\fronorm{U_1S_1U_1^\top -U_2S_2U_2^\top}^2 + \fronorm{V_1S_1V_1^\top-V_2S_2V_2^\top}^2 -
		2\fronorm{U_1S_1V_1^\top - U_2S_2V_2^\top}^2$. Expanding the square and using
		the orthogonal invariance of the Frobenius norm,
		% the fact that $\fronorm{USV}=\fronorm{S}$ for any matrices $U,V$ with unitary columns,
		we find that
		\begin{align*}
		&\fronorm{U_1S_1U_1^\top -U_2S_2U_2^\top}^2 + \fronorm{V_1S_1V_1^\top-V_2S_2V_2^\top}^2 -
		2\fronorm{U_1S_1V_1^\top - U_2S_2V_2^\top}^2\\
		=&2\tr(S_1U_1^\top U_2S_2(U_2^\top U_1-V_2^\top V_1)) +2\tr(S_1(V_1^\top V_2-U_1^\top U_2)S_2V_2^\top V_1)\\
		\overset{(a)}{=} &2\tr(S_1U_1^\top U_2S_2(U_2^\top U_1-V_2^\top V_1)) +2\tr(V_1^\top V_2S_2(V_2^\top V_1-U_2^\top U_1)S_1)\\
		\overset{(b)}{=} & 2\tr(S_1(U_1^\top U_2-V_1^\top V_2)S_2(U_2^\top U_1-V_2^\top V_1)),
		\end{align*}
		where  step $(a)$ is due to the fact that $\tr(A)=\tr(A^\top)$ and step $(b)$ is due to the cyclic property of
		trace. By factorizing $S_i = S_i^{\frac{1}{2}}$ for $i=1,2$ and the cyclic property of trace again, we find that
		\begin{align*}
		&\fronorm{U_1S_1U_1^\top -U_2S_2U_2^\top}^2 + \fronorm{V_1S_1V_1^\top-V_2S_2V_2^\top}^2 -
		2\fronorm{U_1S_1V_1^\top - U_2S_2V_2^\top}^2\\
		=& \tr(S_1^\frac{1}{2} (U_1^\top U_2-V_1^\top V_2)S_2^{\frac{1}{2}} S_2^{\frac{1}{2}} (U_2^\top U_1-V_2^\top V_1) S_1^{\frac{1}{2}})\\
		=& \fronorm{ S_2^{\frac{1}{2}} (U_2^\top U_1-V_2^\top V_1) S_1^{\frac{1}{2}}}^2\geq 0.
		\end{align*}
		Hence the lemma is proved.
	\end{proof}
	\section{Extension for multiple solutions} \label{sec: uniqueness}
	When the problem has more than one solution, let $\mathcal{X}$ be the solution set.
	This set is convex and closed.
	We change the term $\norm{x-\xsol}$ in the quadratic growth condition to
	$\dist(x,\mathcal{X}) = \min_{\xsol \in \mathcal{X}}\norm{x-\xsol}$. For strict complementarity,
	we remove the condition that $\xsol$ is unique and
	demand instead that some $\xsol \in \mathcal{X}$ satisfies the
	conditions listed in strict complementarity.
	The support set $\mathcal{F}(\xsol)$ and complementary set
	$\mathcal{F}^c(\xsol)$ are defined via the $\xsol$ that satisfies strict complementarity.
	Note that the dual vector $\nabla f(\xsol)$ is the same for every $\xsol \in \mathcal{X}$ \cite[Proposition 1]{zhou2017unified}.
	The algorithmic results, Theorem \ref{thm: normalFWresult}, \ref{thm: nonexponentialFiniteConvergence}, and \ref{thm: LinearConvergence} hold
	almost without any change of the proof using the new definition of $\rsol$ and $\delta$.
	The argument to establish quadratic growth via strict complementarity
	is more tedious and we defer it to future work.

	\section{Numerical Experiment setting for Section \ref{sec: numerics}} \label{sec: numericalSectionSettingAppendix}
	We detail the experiment settings of Lasso, support vector machine (SVM), group Lasso, and matrix completion problems.
	The compared methods include FW, away-step FW (awayFW) \cite{guelat1986some}, pairwise FW (pairFW)\cite{lacoste2015global},
	DICG \cite{garber2016linear}, and blockFW \cite{allen2017linear}. All codes are written by MATLAB and performed on a MacBook Pro with Processor 2.3 GHz Intel Core i5 and Memory 8 GB 2133 MHz LPDDR3.
	In our k-FW, we solve the kDS by the FASTA toolbox \cite{GoldsteinStuderBaraniuk:2014,FASTA:2014}: \textit{https://github.com/tomgoldstein/fasta-matlab}. %All codes used in our experiments are available at xxxxxx.
	In DICG (as well as FW, awayFW, pairFW in Group Lasso and SVM), the step size is determined by backtracking line search. The ball sizes of $\ell_1$ norm, group norm, and nuclear norm are set to be the ground truth respectively.

	\subsection{Lasso}
	The experiment is the same as that in \cite{lacoste2015global} except that the data size in our setting is ten times of that in \cite{lacoste2015global}: $A\in\mathbb{R}^{2000\times 5000}$ and $b\in\mathbb{R}^{2000}$. The large size is more reasonable for comparing the computational costs of FW, awayFW, pairFW, DICG and our k-FW. For FW, awayFW and pairFW, we use the MATLAB codes provide by  \cite{lacoste2015global}: \url{https://github.com/Simon-Lacoste-Julien/linearFW}. In DICG (as well as FW, awayFW, pairFW in Group Lasso and SVM), the step size is determined by backtracking line search.

	\subsection{SVM}
	We generate the synthetic data for two-class classification by the following model
	\begin{equation*}
	X=[X_1\ \ X_2]=[U_1V_1+1\ \ U_2V_2-1], \quad X\leftarrow X+E,
	\end{equation*}
	where the elements of $U_1\in\mathbb{R}^{20\times 5}$, $V_1\in\mathbb{R}^{5\times 500}$, $U_2\in\mathbb{R}^{20\times 5}$, and $V_2\in\mathbb{R}^{5\times 500}$ are drawn from $\mathcal{N}(0,1)$. $E$ consists of noise drawn from $\mathcal{N}(0,0.1\sigma_X)$, where $\sigma_X$ denotes the standard deviation of the entries of $X$. Thus, in $X$, the number of samples is 1000 and the number of features is 20. We use $80\%$ of the data as training data to classify the remaining data. In SVM, we use a polynomial kernel $k(x,y)=(x^\top y+1)^2$.

	\subsection{Group Lasso}
	We generate a $100\times 1000$ matrix $X$ whose entries are drawn from $\mathcal{N}(0,1)$ and a $10\times 100$ matrix $W$ with $10$ nonzero columns drawn from $\mathcal{N}(0,1)$. Then let $Y=WX$ and set $Y\leftarrow Y+E$, where the entries of noise matrix $E$ are drawn from $\mathcal{N}(0,0.01\sigma_Y)$. Then we estimate $W$ from $Y$ and $X$ by solving a Group Lasso problem with $k$FW.
	\subsection{Matrix Completion}
	We generate a low-rank matrix as $X=UV^\top$, where the entries of $U\in\mathbb{R}^{500\times 5}$ and $V\in\mathbb{R}^{5\times 500}$ are drawn from $\mathcal{N}(0,1)$. We sample $50\%$ of the entries uniformly at random and recover the unknown entries by low-rank matrix completion.

	\subsection{Objective function vs running time}
	See Figure \ref{fig: Figure_compare_objtime}. $k$FW uses considerably less time compared to other FW variants for Lasso, SVM, and Group Lasso problems.
	It takes longer time than blockFW for the matrix completion problem.
	\begin{figure}
		\hspace{-15pt}
		\begin{subfigure}[(a)]{.24\textwidth}
			\centering
			\includegraphics[width=1.1\linewidth]{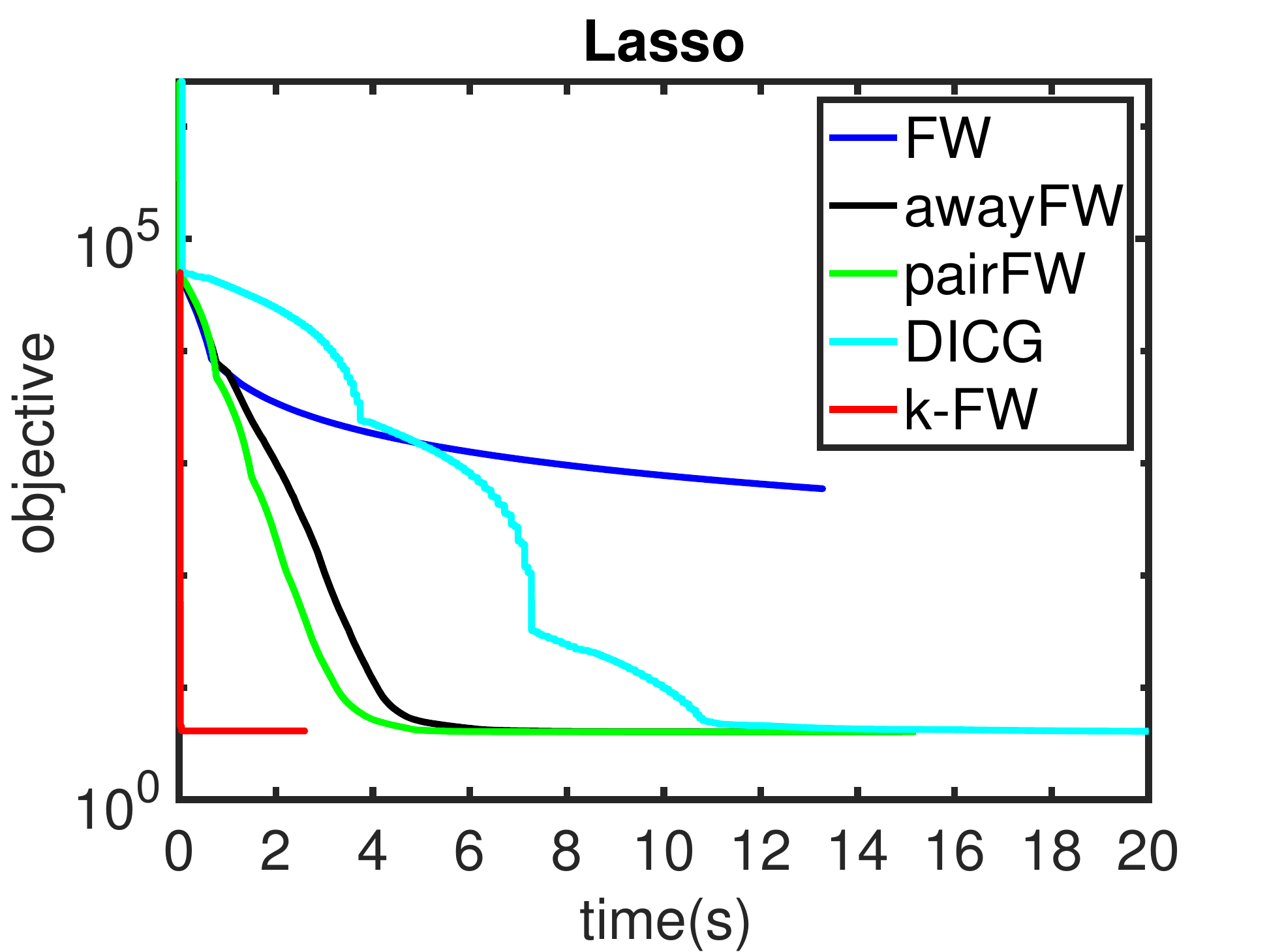}
			%\caption{Lasso}
			\label{fig:fig_lassotime}
		\end{subfigure}%
		\hspace{5pt}
		\begin{subfigure}[(b)]{.24\textwidth}
			\centering
			\includegraphics[width=1.1\linewidth]{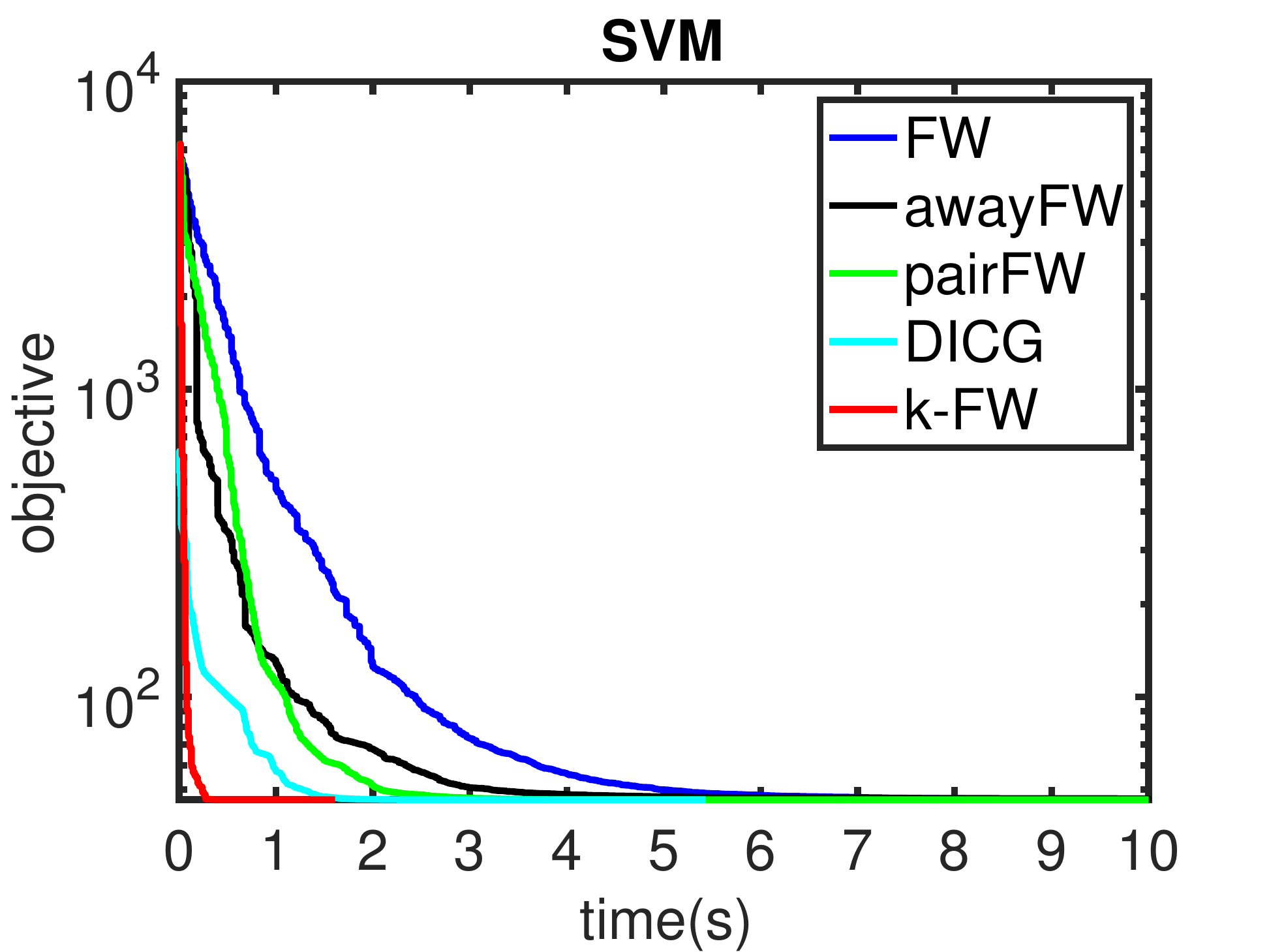}
			%\caption{B}
			\label{fig:fig_SVMtime}
		\end{subfigure}
		\hspace{3pt}
		\begin{subfigure}[(b)]{.24\textwidth}
			%\centering
			\includegraphics[width=1.1\linewidth]{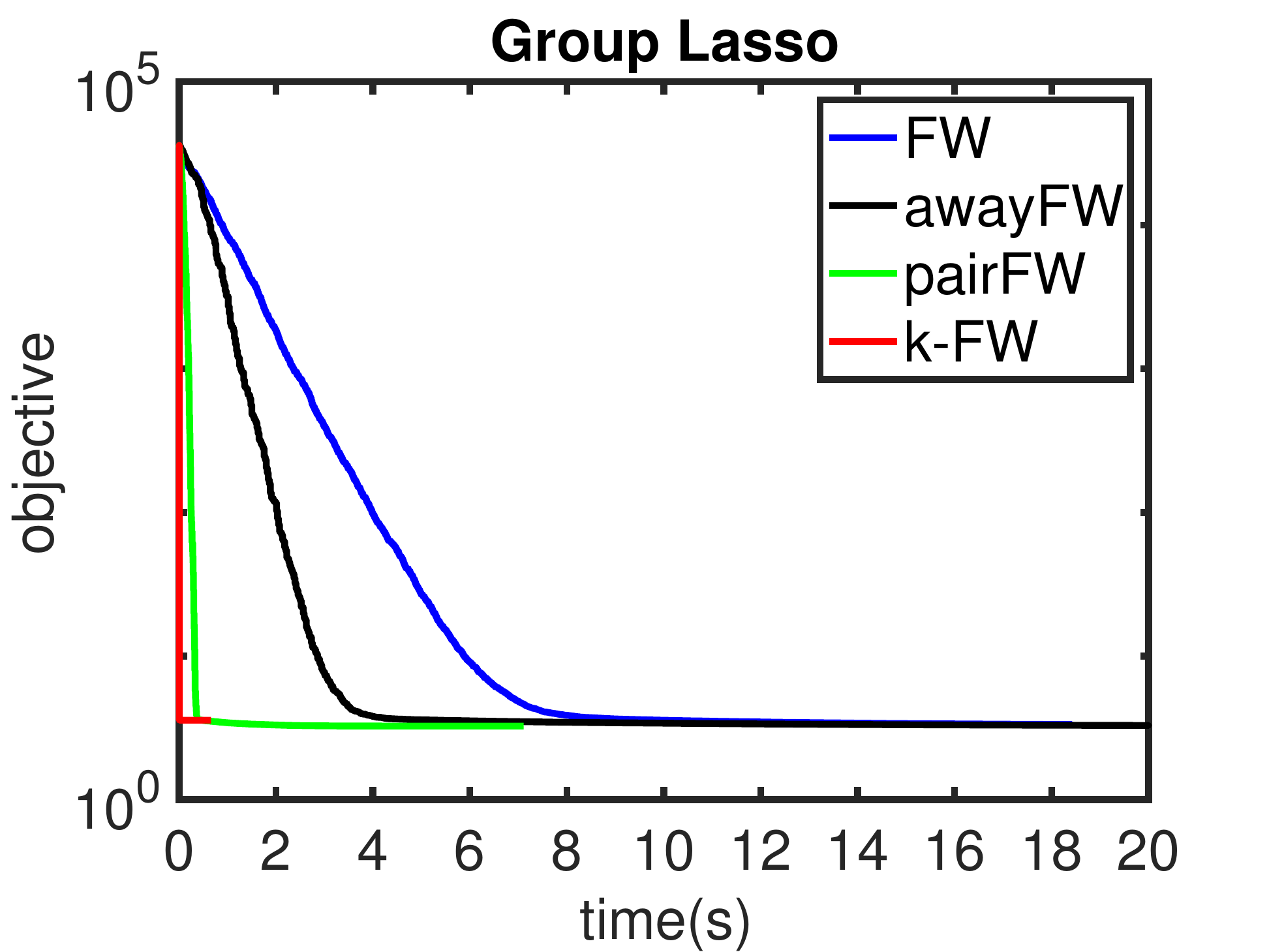}
			%\caption{}
			\label{fig:fig_grouplassoTime}
		\end{subfigure}
		\hspace{3pt}
		\begin{subfigure}[(b)]{.24\textwidth}
			\centering
			\includegraphics[width=1.1\linewidth]{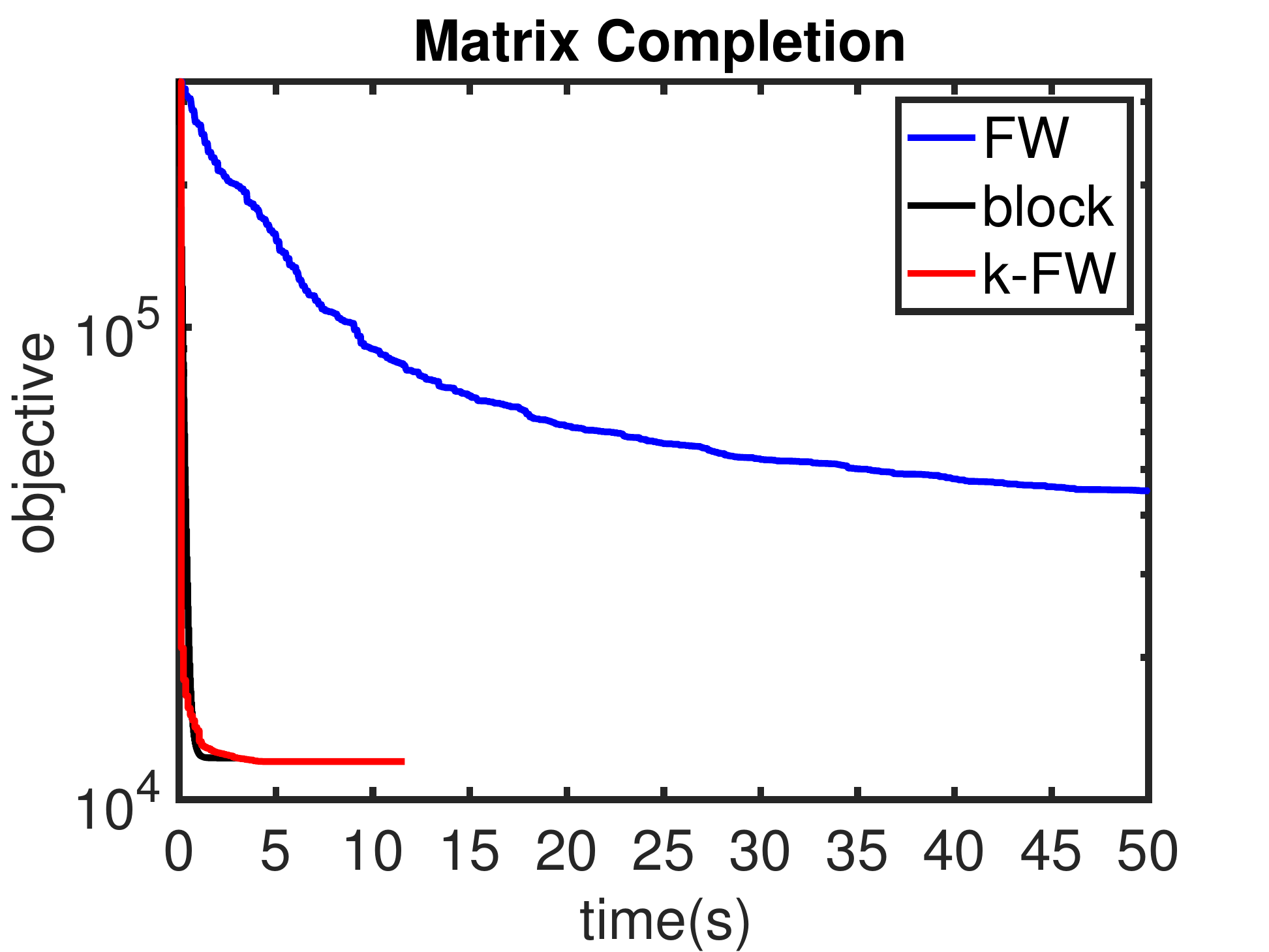}
			%\caption{B}
			\label{fig:fig_MCtime}
		\end{subfigure}
		\vspace{-15pt}
		\caption{Objective against time cost}\label{fig: Figure_compare_objtime}
		\vspace{-10pt}
	\end{figure}

	%\bibliography{reference}
	%\bibliographystyle{alpha}
\end{document}